\documentclass[a4paper,10pt,draft]{article}

\usepackage{fancyhdr,amssymb,amsmath,amsfonts,amscd,amsthm,makeidx}
\usepackage{verbatim}
\makeindex

\def\addsymbol #1: #2{$#1,$~~\pageref{#2}\\}
\def\newnot#1{\label{#1}} 

\def\nt{\noindent}

\newtheorem{thm}{Theorem}[section]
\newtheorem{cor}[thm]{Corollary}

\newtheorem{prop}[thm]{Proposition}
\newtheorem{defn}[thm]{Definition}
\newtheorem{rem}[thm]{Remark}
\newtheorem{exm}[thm]{Example}
\numberwithin{equation}{section}

\def\MM{\mathbb{M}}
\def\N{\mathbb{N}}
\def\Z{\mathbb{Z}}
\def\R{\mathbb{R}}
\def\Rr{\overline{\mathbb R}_+}
\def\C{\mathbb{C}}

\def\a{\mathbf{a}}
\def\g{\mathbf{g}}

\def\A{\mathcal{A}}

\def\D{\mathcal{D}}
\def\E{\mathcal{E}}

\def\H{\mathcal{H}}
\def\K{\mathcal{K}}
\def\L{\mathcal{L}}
\def\M{\mathcal M}
\def\O{\mathcal{O}}
\def\P{\mathcal P}
\def\Q{\mathcal Q}
\def\RR{\mathcal{R}}
\def\S{\mathcal{S}}
\def\U{\mathcal{U}}

\def\WW{\mathcal W}

\def\Re{\mathrm{Re}\,}
\def\Im{\mathrm{Im}\,}

\def\reg{\mathrm{reg}}
\def\c{\mathrm{c}}
\def\clas{\mathrm{cl}}
\def\deg{\mathrm{deg}}
\def\sing{\mathrm{sing}}
\def\dist{\mathrm{dist}}
\def\diag{\mathrm{diag}}
\def\supp{\mathrm{supp}\,}
\def\cone{\mathrm{cone}}

\def\comp{\mathrm{comp}}
\def\loc{\mathrm{loc}}

\def\dim{\mathrm{dim}\,}
\def\Diff{\mathrm{Diff}}
\def\id{\mathrm{id}}
\def\op{\mathrm{op}}
\def\Op{\mathrm{Op}}

\def\bs{\boldsymbol}

\def\X{X^\wedge}
\def\Xw{X^\vartriangle}

\def\HsgX{\mathcal{H}^{s,\gamma}(X^\wedge)}
\def\KsgX{\mathcal{K}^{s,\gamma}(X^\wedge)}
\def\Kl{\kappa_\lambda}
\def\Kll{\set{\kappa_\lambda}_{\lambda\in\R_+}}

\def\TenP{\widehat{\otimes}_\pi}

\newcommand{\norm}[1]{\left\Vert#1\right\Vert}
\newcommand{\abs}[1]{\left\vert#1\right\vert}
\newcommand{\dslash}{d\llap {\raisebox{.9ex}{$\scriptstyle-\!$}}}
\newcommand{\Set}[1]{\left(#1\right)}
\newcommand{\dint}{\int\hspace{-.15cm}\int}
\newcommand{\set}[1]{\left\{#1\right\}}
\newcommand{\virg}[1]{`{#1}'}

\newcommand{\ang}[1]{\langle{#1}\rangle}
\newcommand{\til}[1]{\tilde{#1}}

\newcommand{\Projlim}[1]{\displaystyle\lim_{\substack{\longleftarrow\\ {#1}}}}

\def\Fr{Fr\'{e}chet }

\author{B.-W. Schulze and A.\,Volpato}
\title{Cone and edge calculus with discrete asymptotics}

\begin{document}

\nocite{Benn1,Eski2,Kond1,Melr1,Remp3,Schu34,Schu36,Schu55,Witt2}

\maketitle

\begin{abstract}
\nt This investigation is devoted to the program to characterise continuous and variable discrete asymptotics of solutions to elliptic equations on a manifold with edge, continued in a cicle of forthcoming expositions \cite{Schu62}, \cite{Schu63}. The structure of continuous and variable discrete (in general branching) asymptotics is very complex. Therefore, in order to make things more transparent we present here the approach first in the special constant discrete case, based on meromorphic Mellin symbols.
\end{abstract}



%\tableofcontents

%1.1
%
%
%
%
%
\section{The typical differential operators}

{\footnotesize Operators on a manifold with singularities (of conical, edge, or corner type), expressed in stretched coordinates, are degenerate in a typical way. We illustrate that for the case of differential operators, motivated by the shape of Laplacians belonging to corresponding degenerate Riemannian metrics, or by rephrasing operators with smooth coefficients in anisotropic terms, coming from the geometry or the position of strata in the configuration.\par}

%1.1.1
%
%
%
%
%
\subsection{Manifolds with conical singularities and operators of Fuchs type}\label{S1.1.1}

An example of a cone embedded in the Euclidean space $\R^{1+n}$ for a certain $n\in\N,$ is the set $M$ defined by
$$M:=\set{\til x\in\R^{1+n}:\til x=0,\,\mathrm{or}\,
\,\til x/|\til x|\in X\,\mathrm{when}\,\til x\ne0},$$
where $X$ is a $C^\infty$ submanifold of $S^n=\{\til x\in\R^{1+n}:|\til x|=1\}$ of any codimension.

\begin{defn}\index{manifold!with conical singularities}\label{1.1.1}
A manifold $M$ with conical singularities $S$ is defined as a topological space\footnote{Topological spaces in this context are assumed to be of a simple structure, say, a countable union of compact sets.} with a finite subset $S$ such that
\begin{itemize}
\item[$(\mathrm{i})$] $M\setminus S$ is a $C^\infty$ manifold$;$
\item[$(\mathrm{ii})$] every $v\in S$ has a neighbourhood $V$ in $M$ with a system of so-called singular charts $\chi:V\rightarrow\Xw$ for a closed $C^\infty$ manifold $X=X(v),$ for
\begin{equation}\label{xw1}
\Xw:=(\Rr\times X)/(\{0\}\times X),
\end{equation}
such that $\chi$ restricts to a diffeomorphism
\begin{equation}\label{Chireg}
\chi_\reg:V\setminus\{v\}\rightarrow\X:=\R_+\times X,
\end{equation}
and different singular charts\index{singular charts} $\chi,$ $\til\chi$ from the system have the property that
$$\til\chi_\reg\circ\chi_\reg^{-1}:\R_+\times X\rightarrow\R_+\times X$$
is the restriction of a diffeomorphism $\R\times X\rightarrow\R\times X$ to $\R_+\times X.$
\end{itemize}
\end{defn}

\begin{rem}
Every manifold $M$ with conical singularities $S=\set{v_1,\dots,v_N}$ can be represented as a quotient space
\begin{equation}\label{quo}
M=\MM\,/\sim
\end{equation}
for a $C^\infty$ manifold $\MM$ $($called the stretched manifold\index{stretched manifold} associated with $M)$ with boundary $\partial\MM$ where we assume that $\partial\MM$ is the disjoint union of $C^\infty$ manifolds $X_j,$ $j=1,\dots,N.$ The quotient map $\MM\rightarrow M$ is induced by $X_j\mapsto v_j,$ $j=1,\dots,N,$ and the identity map on $M\setminus S=\MM\setminus\partial\MM.$
\end{rem}

\nt Let $\Diff^\nu(\cdot)$\label{Diffnu} denote the set of all differential operators of order $\nu$ with smooth coefficients on the $C^\infty$ manifold in parantheses. The topology of local $C^\infty$ coefficients gives rise to a natural \Fr topology in $\Diff^\nu(\cdot).$

\begin{defn}\label{Diffdeg}
Let $M$ be a manifold with conical singularities $S.$ We denote by $\Diff^\mu_{\deg}(M)$ the set of all $A\in\Diff^\mu(M\setminus S)$ that are close to every $v\in S$ in the splitting of variables $(r,x)\in\R_+\times X$ of the form
\begin{equation}\label{new2}
A=r^{-\mu}\sum_{j=0}^\mu a_j(r)\Set{-r\frac\partial{\partial r}}^j
\end{equation}
for certain $a_j\in C^\infty(\Rr,\Diff^{\mu-j}(X)),$ $j=0,\dots,\mu.$ Any such operator\index{operators!of Fuchs type} $A$ will be called of Fuchs type.\index{Fuchs type!operator of $-$}
\end{defn}

\nt Observe that manifolds with conical singularities form a category\index{category} with natural (iso)morphisms. If $M$ and $\til M$ are objects in that category, with $S$ and $\til S$ as the respective conical singularities, then a morphism is a continuous map
$$\chi:M\rightarrow\til M$$
restricting to a map $S\rightarrow\til S,$ such that $\chi$ is the quotient map induced by a differentiable map $\chi_\mathrm{s}:\MM\rightarrow\til\MM$ in the category of $C^\infty$ manifolds with boundary (the latter includes that $\chi_\mathrm{s}|_{\partial\MM}:\partial\MM\rightarrow\partial\til\MM$ is a diffentiable map). In particular, an isomorphism $\chi:M\rightarrow\til M$ between manifolds with conical singularities gives rise (via operator push forward) to an isomorphism
$$\chi_*:\Diff_{\deg}^\mu(M)\rightarrow\Diff_{\deg}^\mu(\til M)$$
between the respective spaces of operators in Definition \ref{Diffdeg}.\\

\nt Fuchs type differential operators are motivated by the shape of the Laplacian\index{Laplacian!with respect to a Riemannian metric} on (open stretched) cones $\R_+\times X\ni(r,x)$ with respect to a Riemannian metric\index{Riemannian metric!Laplacian with respect to a $-$} of the form
$$dr^2+r^2g_X(r)\label{RiemMetr}$$
where $g_X(r)$ is a family of Riemannian metrics on $X,$ smooth in $\Rr$ up to $r=0.$ The Laplacian then has the form (\ref{new2}) for $\mu=2$ (including the weight factor $r^{-2}).$\\

\nt Operators of the form (\ref{new2}) appear when we introduce polar coordinates $(r,x)\in\R_+\times S^n$ (with $S^n$ being the unit sphere on $\R^{n+1})$ in a differential operator $A$ in $\R^{1+n}\ni\til x,$
$$A=\sum_{|\alpha|\le\mu}c_\alpha(\til x)D_{\til x}^\alpha$$
with $c_\alpha\in C^\infty(\R^{1+n}).$ For instance, if we consider the Laplace operator
$$\Delta=\sum_{j=1}^{n+1}\frac{\partial^2}{\partial\til x_j^2}$$ then in polar coordinates\index{Laplace operators!in polar coordinates}\index{polar coordinates!Laplace operator in $-$} we obtain
$$\Delta=r^{-2}\set{\Set{r\frac\partial{\partial
r}}^2+(n-1)r\frac\partial{\partial r}+\Delta_{S^n}}.$$
Observe that the operator $\displaystyle\sum_{j=1}^{n+1}\til x_j\frac{\partial}{\partial\til x_j}$ in polar coordinates takes the form $r\displaystyle\frac\partial{\partial r}.$

%1.1.2
%
%
%
%
%
\subsection{Manifolds with edge and edge-degenerate operators}

An example of a wedge embedded in the Euclidean space $\R^{1+n+q}$ for certain $n,q\in\N,$ is the set $M:=\Xw\!\times\Omega,$ $\Xw=\{\til x\in\R^{1+n}:\til x=0,$ or $\til x/|\til x|\in X$ when $x\ne0\},$ where $X$ is a closed $C^\infty$ submanifold of $S^n$ of any codimension and $\Omega\subseteq\R^q$ an open set. We call $\Xw$ (cf. formula \eqref{xw1}) the model cone of the wedge and $X$ its base. Moreover, $\Omega$ is called the edge of $M;$ it is non-trivial only when $q>0,$ otherwise $M$ is simply a cone.\\

\nt In order to define manifolds with edge in general we first note that we can talk about locally trivial $\Xw$-bundles over $Y$ for some $C^\infty$ manifolds $X$ and $Y.$ As before we assume $X$ to be closed. Such a bundle is a topological space $W$ with the following properties:
\begin{itemize}
\item[$(\mathrm{i})$] there is a canonical continuous projection $\pi:W\rightarrow Y,$ such that for every $y\in Y$ the preimage $\pi^{-1}(y)=:W_y$ is isomorphic to $\Xw$ in the category on manifolds with conical singularities;
\item[$(\mathrm{ii})$] $W\setminus Y$ is a $C^\infty$ manifold, and $\pi|_{W\setminus Y}:W\setminus Y\rightarrow Y$ is a differentiable map, $\pi|_Y:Y\rightarrow Y$ the identity (where the tip of the cone $\pi^{-1}(y)$ is identified with $y);$
\item[$(\mathrm{iii})$] every $y\in Y$ has a neighbourhood $U$ such that there is a homeomorphism
$$\tau:\pi^{-1}(U)=:W_U\rightarrow \Xw\times U,$$
which restricts to a diffeomorphism $W_U\setminus Y\rightarrow\X\times U$ and to the identity $U\rightarrow U,$ and for the canonical projections $\pi:W_U\rightarrow U$ and $p:\Xw\times U\rightarrow U,$ we have $\pi=p\circ\tau.$
\end{itemize}

\nt The space $W\setminus Y$ has the structure of a locally trivial $\X$-bundle over $Y,$ with the canonical projection $\pi|_{W\setminus Y}:W\setminus Y\rightarrow Y$ (as a subbundle of $W$ in the sense of the general terminology on fibre bundles).\\

\nt In the following definition we first assume that $Y$ is connected.

\begin{defn}\label{defed}
A topological space $M$ is called a manifold with edge\index{manifold!with edge}\index{edge!manifold with $-$} $Y$ if
\begin{itemize}
\item[$(\mathrm{i})$] $M\setminus Y$ and $Y$ are $C^\infty$ manifolds$;$
\item[$(\mathrm{ii})$] the set $Y$ has a neighbourhood $W$ which is a locally trivial $\Xw$-bundle over $Y$ for some $C^\infty$ manifold $X.$
\end{itemize}
\end{defn}

\nt Observe, as a consequence of Definition \ref{defed}, that every $y\in Y$ has a neighbourhood $V$ in $M$ such that there is a so-called singular chart
\begin{equation}\label{chart}
\chi:V\rightarrow\Xw\times\Omega,\qquad\Omega\subseteq\R^q\quad\textup{open},\qquad q=\dim Y,
\end{equation}
which restricts to a diffeomorphism
\begin{equation}\label{regchart}
\chi:V\setminus Y\rightarrow\X\times\Omega.
\end{equation}

\nt Definition \ref{defed} easily extends to the case when $Y$ has several connected components $Y_j;$ then the respective $X_j$ may depend on $j.$

\begin{rem}
Every manifold $M$ with edge $Y$ $($consisting of connected components $Y_j$ of different dimensions $q_j,$ $j=1,\dots,N)$ can be represented as a quotient space
$$M=\MM\,/\sim$$
for a $C^\infty$ manifold $\MM$ $($called the stretched manifold\index{stretched manifold} associated with $M)$ with boundary $\partial\MM$ consisting of connected components $(\partial\MM)_j,$ $j=1,\dots,N,$ where $(\partial\MM)_j$ is a $($locally trivial$)$ $X_j$-bundle over $Y_j$ for certain $C^\infty$ manifold $X_j.$ The quotient map $\MM\rightarrow M$ is induced by the bundle projections $(\partial\MM)_j\rightarrow Y_j,$ $j=1,\dots,N,$ and the identity map on $M\setminus Y=\MM\setminus\partial\MM.$ 
\end{rem}

\begin{exm}
\begin{itemize}
\item[$(\mathrm{i})$] For $M=\Xw\times\Omega$ the stretched manifold has the form
$$\MM=\Rr\times X\times\Omega.$$
In this case $\partial\MM=\set{0}\times X\times\Omega$ plays the role of the above-mentioned $X$-bundle over the edge $\Omega.$
\item[$(\mathrm{ii})$] A $C^\infty$ manifold $M$ with boundary $\partial M$ can be interpreted as a manifold with edge $\partial M.$ In this case $X$ is a single point$,$ and the neighbourhood $W$ of \textup{Definition} $\ref{defed}$ can be interpreted as the inner normal bundle of the boundary $($with respect to any fixed Riemannian metric$).$
\end{itemize}
\end{exm}

\begin{defn}\label{Diffdegedge}
Let $M$ be a manifold with edge $Y.$ Then $\Diff^\mu_{\deg}(M)$ is defined to be the set of all $A\in\Diff^\mu(M\setminus Y)$ that are close to every $y\in Y$ in the splitting of variables $(r,x,y)\in\R_+\times X\times\Omega$ of the form
\begin{equation}\label{deg}
A=r^{-\mu}\sum_{j+|\alpha|\le\mu}a_{j\alpha}(r,y)\Set{-r\frac\partial{\partial r}}^j\Set{rD_y}^\alpha
\end{equation}
for certain $a_{j\alpha}\in C^\infty(\Rr\times\Omega,\Diff^{\mu-(j+|\alpha|)}(X)).$ Any such operator\index{operators!edge-degenerate $-$} $A$ will be called edge-degenerate\index{operator}.
\end{defn}

\nt Manifolds with edge form a category with natural (iso)morphisms. If $M$ and $\til M$ are objects in that category, with $Y$ and $\til Y$ as the respective edges, then a morphism is represented by a continuous map
$$\chi:M\rightarrow\til M$$
that restricts to a differentiable map $Y\rightarrow\til Y,$ and $\chi$ is the quotient map of a differentiable map $\chi_\mathrm{s}:\MM\rightarrow\til\MM$ between the respective stretched manifolds in the category of smooth manifolds with boundary, where $\chi_\mathrm{s}|_{\MM\setminus\partial\MM}:
\MM\setminus\partial\MM\rightarrow\til\MM\setminus\partial\til\MM$ is diffentiable in the category of smooth manifolds, and $\chi_\mathrm{s}|_{\partial\MM}:\partial\MM\rightarrow\partial\til\MM$ is a differentiable morphism between the respective $X$-bundles over $Y$ and $\til X$-bundles over $\til Y.$ In particular, an isomorphism $\chi:M\rightarrow\til M$ between manifolds with edge is given when $\chi_\mathrm{s}: \MM\rightarrow\til\MM$ and $\chi_\mathrm{s}|_{\partial\MM}:\partial\MM\rightarrow\partial\til\MM$ are isomorphisms. Note that when $\chi:M\rightarrow\til M$ is an isomorphism then the operator push forward gives rise to an isomorphism
$$\chi_*:\Diff_{\deg}^\mu(M)\rightarrow\Diff_{\deg}^\mu(\til M)$$
between the respective spaces of operators in Definition \ref{Diffdegedge}.\\

\nt Edge-degenerate operators are motivated by the shape of the Laplacian\index{Laplacian!with respect to a Riemannian metric} on (open stretched) wedges $\R_+\times X\times\Omega\ni(r,x,y)$ with respect to a Riemannian metric\index{Riemannian metric!Laplacian with respect to a $-$} of the form
$$dr^2+r^2g_X(r)+dy^2$$
where $g_X(r,y)$ is a family of Riemannian metrics on $X,$ smooth in $\Rr\times\Omega$ up to $r=0.$ The Laplacian then has the form (\ref{deg}) for $\mu=2$ (including the weight factor $r^{-2}).$ Other operators of the form (\ref{deg}) appear when we introduce polar coordinates $\R^{1+n}\setminus\{0\}\ni\til x\mapsto(r,x)\in\R_+\times S^n$ in a differential operator $A$ of order $\mu$
\begin{equation}\label{A}
A=\sum_{|\alpha|\le\mu}c_\alpha(\til x,y)D^\alpha_{\til x,y}
\end{equation}
with coefficients $c_\alpha\in C^\infty(\R^{1+n+q}).$ For instance, for the Laplacian\index{Laplace operators!in polar coordinates}\index{polar coordinates!Laplace operator in $-$} in
$\R^{1+n+q}$ we obtain
\begin{equation}\label{Lapl}
\Delta=r^{-2}\set{\Set{r\frac\partial{\partial
r}}^2+(n-1)r\frac\partial{\partial r}+
\Delta_{S^n}+\sum^q_{l=1}r^2\frac{\partial^2}{\partial y^2_l}}.
\end{equation}

%1.2
%
%
%
%
%
\section{The cone algebra}\label{1.2}

{\footnotesize By the cone algebra we understand a pseudo-differential algebra on (the $C^\infty$ part of) a manifold $M$ with conical singularities $S,$ with a principal symbolic hierarchy $\sigma=(\sigma_\psi,\sigma_\mathrm c),$ consisting of the principal interior symbol $\sigma_\psi$ and the conormal symbol $\sigma_\mathrm c,$ contributed by $M\setminus S$ and $S,$ respectively. The cone algebra is defined as a substructure of the standard pseudo-differential calculus on $M\setminus S$ where, close to $S,$ the operators in terms of the Fourier transform are rephrased by operators based on the Mellin transform. According to the expected asymptotics of solutions to elliptic equations in weighted Sobolev spaces, we define so-called discrete asymptotics, also for Mellin symbols and Green operators, and study ellipticity with parametrices.\\

\nt In this section we outline essential elements of the analysis of the cone algebra with discrete asymptotics in a new trasparent way, also as a background for the material in Section 4 below and in \cite{Schu62}, \cite{Schu63}. More details and complete proofs, as far as they are not given here, may be found, for instance, in \cite{Schu20}.\par}

%1.2.1
%
%
%
%
%
\subsection{Tools from the classical pseudo-differential calculus}\label{Sec1.2.1}
Let us first define symbols and operators in an open set of $\R^n.$
\begin{defn}\label{stsym}
The symbol space\index{symbol!space!scalar $-$}\index{symbols!space of scalar $-$}
\begin{equation}\label{stsym2}
S^\mu(U\times\R^n)
\end{equation}
for $U\subseteq\R^m$ open$,$ $\mu\in\R,$ is defined to be the set of all $a(x,\xi)\in C^\infty(U\times\R^n)$ such that
$$\sup_{x\in K,\xi\in\R^n}\ang{\xi}^{-\mu+|\beta|}|D_x^\alpha D_\xi^\beta a(x,\xi)|$$
is finite for every $K\Subset U$ and multi-indices $\alpha\in\N^m,$ $\beta\in\N^n,$ where $\ang\xi:=(1+\abs\xi^2)^{1/2}.$\label{angeta} The space of classical symbols\index{classical!symbols}\index{symbols!classical $-$}
$$S^\mu_\clas( U\times\R^n)$$
is defined as the subspace of all $a(x,\xi)$ in \textup{(\ref{stsym2})} such that there are homogeneous components\index{homogeneous!components} $a_{(\mu-j)}(x,\xi)\in C^\infty(U\times(\R^n\setminus\{0\})),$ $j\in\N,$ with
$$a_{(\mu-j)}(x,\lambda\xi)=\lambda^{\mu-j}a_{(\mu-j)}(x,\xi),\qquad\textup{for all}\;\lambda\in\R_+,$$
such that
$$a(x,\xi)-\chi(\xi)\sum_{j=0}^Na_{(\mu-j)}(x,\xi)\in S^{\mu-(N+1)}(U\times\R^n)$$
for all $N\in\N;$ here $\chi(\xi)$ is any excision function\index{excision function} in $\R^q$ $($i.e.$,$ $\chi\in C^\infty(\R^q),$ $\chi(\eta)=0$ for $|\eta|<c_0,$ $\chi(\eta)=1$ for $|\eta|>c_1,$ with certain $0<c_0<c_1).$
\end{defn}

\nt If an assertion holds both in the classical and general case, we write \virg{(cl)} as subscript.\\

\nt The space $S_{(\clas)}^\mu(U\times\R^n)$ is \Fr in a natural way, and the set $S_{(\clas)}^\mu(\R^n)$ of symbols with constant coefficients is closed in the respective space. Then
$$S_{(\clas)}^\mu(U\times\R^n)=C^\infty(U,S_{(\clas)}^\mu(\R^n)).$$

\nt To associate pseudo-differential operators with symbols, we often consider the case $U=\Omega\times\Omega$ for an open set $\Omega\subseteq\R^n$ and write in this case $(x,x')$ instead of $x.$ Moreover, since in Definition \ref{stsym} the variables and covariables are of independent dimension, we can also form spaces of parameter-dependent symbols\index{parameter-dependent!symbols}\index{symbols!parameter-dependent $-$} $a(x,x',\xi,\lambda)$ with parameter $\lambda,$ where $(\xi,\lambda)\in\R^{n+l},$ $l\in\N,$ is formally treated as the covariables. This gives rise to parameter-dependent pseudo-differential operators\index{parameter-dependent!operators!pseudo-differential $-$}\index{pseudo-differential operators!parameter-dependent $-$}\index{operators!parameter-dependent $-$!pseudo-differential $-$}
\begin{equation}\label{Op}
\Op(a)(\lambda)u(x):=\dint e^{i(x-x')\xi}a(x,x',\xi,\lambda)u(x')dx'\dslash\xi,
\end{equation}
$\dslash\xi=(2\pi)^{-n}d\xi,$ first for $u\in C^\infty_0(\Omega)$ and then extended to various larger function and distribution spaces. In (\ref{Op}) the function $a(x,x',\xi,\lambda)$ is often called a double symbol while $a(x,\xi,\lambda)$ and $a(x',\xi,\lambda)$ are called left and right symbol, respectively. The case $l=0$ corresponds to the \virg{usual} pseudo-differential calculus without parameters. Let us set
\begin{equation}\label{Lmu}
L^\mu_{(\clas)}(\Omega;\R^l):=\set{\Op(a)(\lambda):a(x,x',\xi,\lambda)\in S^\mu_{(\clas)}(\Omega\times\Omega\times\R^{n+l})}.
\end{equation}

\nt The elements of (\ref{Lmu}) are called parameter-dependent pseudo-differential operators in $\Omega$ (classical or general).\\

\nt The standard elements of the pseudo-differential calculus may be found in many text-books, see, for instance, \cite{Horm3}, \cite{Kuma1} and \cite{Shub2}. The main purpose here is to fix notation and then freely employ the known results.\\

\nt The space 
$$L^{-\infty}(\Omega;\R^l):=\bigcap_{\mu\in\R}L^\mu(\Omega;\R^l)\label{Linf}$$
can be identified with $\S(\R^l,L^{-\infty}(\Omega)),$ with $L^{-\infty}(\Omega)$ being the space of all integral operators with kernel in $C^\infty(\Omega\times\Omega)$ (in the \Fr topology from the bijection $L^{-\infty}(\Omega)\cong C^\infty(\Omega\times\Omega)).$\\

\nt Let $X$ be a $C^\infty$ manifold with a Riemannian metric. Then there is a similar identification between $C^\infty(X\times X)$ and the space $L^{-\infty}(X)$\label{Lclass} of smoothing operators\index{smoothing operators}\index{operators!smoothing $-$} on $X,$ namely,
$$Cu(x)=\int_Xc(x,x')u(x')dx',$$
where $dx'$ is the measure from the Riemannian metric on $X.$ This gives rise to
$$L^{-\infty}(X;\R^l)=\S(\R^l,L^{-\infty}(X)),\label{Linf2}$$
the parameter-dependent variant of smoothing operators\index{parameter-dependent!operators!smoothing $-$}\index{operators!smoothing $-$!parameter-dependent $-$} on $X.$ The known coordinate invariance of pseudo-differential operators allows us to define operators globally on $X.$ If $\chi:U\rightarrow\Omega$ is a chart on $X,$ $\Omega\subseteq\R^n$ open, we set
$$L^\mu_{(\clas)}(U;\R^l):=\set{(\chi^{-1})_*A(\lambda):A(\lambda)\in L^\mu_{(\clas)}(\Omega;\R^l)},$$
with $(\chi^{-1})_*$ denoting the operator push forward under $\chi^{-1},$ i.e., $(\chi^{-1})_*A(\lambda)=\chi^*A(\lambda)(\chi^{-1})^*$ with $\chi^*$ being the function pull back.\\

\nt Let us fix a locally finite open covering $\U=(U_j)_{j\in\N}$ of $X$ by coordinate neighbourhoods $U_j,$ $(\varphi_j)_{j\in\N}$ a subordinate partition of unity and $(\psi_j)_{j\in\N}$ a system of functions $\psi_j\in C^\infty_0(U_j)$ such that $\varphi_j\prec\psi_j\newnot{prec}$ for all $j$ (here $\varphi\prec\psi$ means $\psi\equiv1$ on $\supp\varphi$). Then we define
\begin{eqnarray}\label{new10}
\lefteqn{L^\mu_{(\clas)}(X;\R^l):=\big\{\displaystyle\sum_{j\in\N}\varphi_jA_j(\lambda)\psi_j+C(\lambda):A_j(\lambda)\in L^\mu_{(\clas)}(U_j;\R^l),}\nonumber\\[-5mm]
\\
& & {}\hspace{5cm}j\in\N,\textup{ and }C(\lambda)\in L^{-\infty}(X;\R^l)\big\}.\nonumber
\end{eqnarray}

\nt The classical notion of properly supported pseudo-differential operators easily extends to the parameter-dependent case.

\begin{rem}\label{R1.2.2}
Every $A\in L^\mu_{(\clas)}(X;\R^l)$ can be written in the form $A=A_0+C$ where $A_0$ is properly supported and $C\in L^{-\infty}(X;\R^l).$
\end{rem}

\begin{thm}
$A\in L^\mu_{(\clas)}(X;\R^l),$ $B\in L^\nu_{(\clas)}(X;\R^l)$ and $A$ or $B$ properly supported entails $AB\in L^{\mu+\nu}_{(\clas)}(X;\R^l).$
\end{thm}

\nt There is a natural notion of parameter-dependent ellipticity of operator families in (\ref{new10}). Consider, for simplicity, the classical case when operators have a parameter-dependent homogeneous principal symbol $\sigma_\psi(A)(x,\xi,\lambda)$ belonging to $C^\infty(T^*X\times\R^l\setminus0)$ (with $0$ indicating $(\xi,\lambda)=0),$ (positively) homogeneous of order $\mu$ in $(\xi,\lambda).$

\begin{defn}
An $A(\lambda)\in L^\mu_\clas(X;\R^l)$ is called parameter-dependent elliptic if
$$\sigma_\psi(A)(x,\xi,\lambda)\ne0$$
for all $(x,\xi,\lambda)\in T^*X\times\R^l\setminus0.$ 
\end{defn}

\begin{thm}
\begin{itemize}
 \item[\textup{(i)}] A parameter-dependent elliptic $A(\lambda)\in L^\mu_\clas(X;\R^l)$ has a $($properly supported$)$ parameter-dependent elliptic parametrix $P(\lambda)\in L^{-\mu}_\clas(X;\R^l)$ such that
$$1-P(\lambda)A(\lambda),1-A(\lambda)P(\lambda)\in L^{-\infty}(X;\R^l).$$
 \item[\textup{(ii)}] If $X$ is compact$,$ then the parameter-dependent elliptic operator $A(\lambda)$ induces a family of Fredholm operators
\begin{equation}\label{F11}
A(\lambda):H^s(X)\rightarrow H^{s-\mu}(X)
\end{equation}
for all $s\in\R,$ and there is a constant $C>0$ such that \textup{\eqref{F11}} are isomorphisms for all $|\lambda|\ge C.$
\end{itemize}
\end{thm}

\nt Another well-known result in this context is the following theorem.

\begin{thm}\label{red}
Let $X$ be a closed compact $C^\infty$ manifold. Then for every $\mu\in\R$ and $l\ge1$ there is an $\RR^\mu(\lambda)\in L^\mu_\clas(X;\R^l)$ which induces isomorphisms
$$\RR^\mu(\lambda):H^s(X)\rightarrow H^{s-\mu}(X)$$
for all $\lambda\in\R^l,$ $s\in\R,$ and we have $(\RR^\mu(\lambda))^{-1}\in L^{-\mu}_\clas(X;\R^l).$
\end{thm}

\nt Operators of that kind will be also referred to as a (parameter-dependent) order reducing family\index{parameter-dependent!order reducing family}\index{order reducing family!parameter-dependent $-$}.\\

\nt The pseudo-differential calculus on a manifold $M$ with conical singularities $S,$ cf. Definition \ref{1.1.1}, formulated in Section \ref{S1.2.6} below, will be a subcalculus of $L_\clas^\mu(M\setminus S).$ In order to prepare some notation we define the space $L_\deg^\mu(M)\subset L^\mu(M\setminus S)$ consisting of all operators $A$ that are modulo $L^{-\infty}(M\setminus S)$ in the local splitting of variables $(r,x)$ close to $S,$ cf. the formula \eqref{Chireg}, of the form
\begin{equation}\label{A1}
A=r^{-\mu}\Op_{r,x}(a)
\end{equation}
for a symbol $a(r,x,\rho,\xi)=\til a(r,x,r\rho,\xi),$ $\til a(r,x,\til\rho,\xi)\in S^\mu_\clas(\Rr\times\Sigma\times\R^{1+n});$ here $\Sigma\subseteq\R^n$ corresponds to a chart on $X.$ An alternative is to write locally near $S$ (again, modulo $L^{-\infty}(M\setminus S)$)
\begin{equation}\label{A2}
A=r^{-\mu}\Op_r(p)
\end{equation}
for an operator family $p(r,\rho)=\til p(r,r\rho),$ $\til p(r,\til\rho)\in C^\infty(\Rr,L^\mu_\clas(X;\R_{\til\rho})).$\\

\nt The latter representation will be dominating in the considerations below. However, from \eqref{A1} we see that the homogeneous principal symbol $\sigma_\psi(A)$ of order $\mu,$ invariantly defined as a function in $C^\infty(T^*(M\setminus S)\setminus 0),$ has a representation locally near $S$ in the form
$$\sigma_\psi(A)(r,x,\rho,\xi)=r^{-\mu}\til\sigma_\psi(A)(r,x,r\rho,\xi)$$
for a $\til\sigma_\psi(A)(r,x,\til\rho,\xi)$ which is smooth up to $r=0.$ We will call $\til\sigma_\psi(A)$ the reduced principal symbol of $A.$\\

\nt Observe that we have an analogue of Remark \ref{R1.2.2} for the class $L_\deg^\mu(M).$

\begin{thm}\label{comp8}
$A\in L_\deg^\mu(M),$ $B\in L_\deg^\nu(M)$ and $A$ or $B$ properly supported entails $AB\in L_\deg^{\mu+\nu}(M)$ and we have
$$\sigma_\psi(AB)=\sigma_\psi(A)\sigma_\psi(B),\quad\til\sigma_\psi(AB)=\til\sigma_\psi(A)\til\sigma_\psi(B).$$
\end{thm}

\begin{defn}
An $A\in L_\deg^\mu(M)$ is called $\sigma_\psi$-elliptic $($of order $\mu)$ if $\sigma_\psi(A)\ne0$ on $T^*(M\setminus S)\setminus0$ as usual$,$ and locally close to $S,$ $\til\sigma_\psi(A)\ne0$ up to $r=0.$
\end{defn}

\begin{thm}
$A\in L_\deg^\mu(M)$ $\sigma_\psi$-elliptic has a $($properly supported$)$ parametrix $P\in L_\deg^{-\mu}(M)$ in the sense
$$1-PA,1-AP\in L^{-\infty}(M\setminus S),$$
and $\sigma_\psi(P)=\sigma_\psi(A)^{-1},$ $\til\sigma_\psi(P)=\til\sigma_\psi(A)^{-1}.$
\end{thm}

%1.2.2
%
%
%
%
%
\subsection{Mellin operators and weighted spaces}\label{1.2.2}
In the analysis on manifolds with conical or edge singularities it is useful to employ a variant of pseudo-differential operators on the half-axis $\R_+\ni r$ (with $r$ corresponding to the distance to the singularity) based on the Mellin transform rather than the Fourier transform.\\

\nt The Mellin transform\index{Mellin!transform}\index{transform!Mellin $-$} is defined as 
\begin{equation}\label{Mellin}
Mu(z):=\int_0^\infty r^{z-1}u(r)dr,
\end{equation}
first for $u\in C_0^\infty(\R_+),$ where $z$ is a complex variable, and then extended to various classes of weighted distribution spaces, also vector-valued ones, where $z$ varies on some weight line 
$$\Gamma_\beta=\{z\in\C:\Re z=\beta\}.$$\label{Gambeta}

\nt If $u\in C_0^\infty(\R_+),$ $Mu(z)$ is an entire function, and $g(z):=Mu(z)|_{\Gamma_\beta}$ is a Schwartz function on $\Gamma_\beta$ for every $\beta\in\R,$ uniformly in compact $\beta$-intervals. We then have the inversion formula
\begin{equation}\label{Mellinv}
u(r)=\frac{1}{2\pi i}\int_{\Gamma_\beta}r^{-z}g(z)dz
\end{equation}
for every $\beta\in\R.$ We will call
$$M_\gamma:u\mapsto Mu|_{\Gamma_{\frac12-\gamma}}\newnot{symbol:wM}$$
the weighted Mellin transform\index{Mellin!transform!weighted $-$}\index{weighted!Mellin transform}\index{transform!Mellin $-$!weighted $-$} of $u$ with weight $\gamma\in\R.$ As is well-known, $M_\gamma$ extends from $C_0^\infty(\R_+)$ to an isomorphism
$$M_\gamma:r^\gamma L^2(\R_+)\rightarrow L^2(\Gamma_{\frac12-\gamma})$$
and $M_\gamma^{-1}$ has the form (\ref{Mellinv}) for $\beta=\frac12-\gamma.$\\

\nt Mellin pseudo-differential operators based on $M_\gamma$ are defined by
$$\op_M^\gamma(f)u(r):=\int_{-\infty}^\infty\int_0^\infty\Set{\frac{r}{r'}}^{-(\frac12-\gamma+i\rho)}f(r,r',\frac12-\gamma+i\rho)u(r')\frac{dr'}{r'}\dslash\rho,\newnot{opMell}$$
$\dslash\rho=(2\pi)^{-1}d\rho,$\label{dsla} where $f(r,r',z)$ is a symbol in $S^\mu(\R_+\times\R_+\times\Gamma_{\frac12-\gamma}),$ see Definition \ref{stsym}; here $\Gamma_{\frac12-\gamma}$ is identified with a real axis with $\rho=\Im z$ as the covariable. When $\gamma=0$ we will use the notation
\begin{equation}\label{new8}
\op_M(\cdot):=\op_M^0(\cdot)
\end{equation}

\nt We employ Mellin pseudo-differential operators also with operator-valued symbols, e.g.,
\begin{equation}\label{opvalsym}
f(r,r',z)\in C^\infty(\R_+\times\R_+,L^\mu_{(\clas)}(X;\Gamma_{\frac12-\gamma}))
\end{equation}
for a closed compact $C^\infty$ manifold $X,$ and where $f$ takes values in the space of parameter-dependent pseudo-differential operators on $X,$ with parameter $\rho=\Im z,$ $z\in\Gamma_{\frac12-\gamma},$ see (\ref{new10}). In (\ref{opvalsym}) we tacitly employ the natural \Fr topology in the spaces (\ref{new10}).

\begin{rem}\label{con}
The operator $\op_M^\gamma(f)$ for $f$ as in $(\ref{opvalsym})$ belongs to $L_\clas^\mu(\X)$ and as such induces a continuous operator
$$\op_M^\gamma(f):C_0^\infty(\X)\rightarrow C^\infty(\X)$$
for every $\gamma\in\R.$ Then$,$ for $(\delta_\lambda u)(r,x):=u(\lambda r,x),$ $\lambda\in\R_+,$ we have
\begin{equation}\label{coneq}
\delta^{-1}_\lambda\op_M^\gamma(f)\delta_\lambda=\op_M^\gamma(f_\lambda),
\end{equation}
$f_\lambda(r,r',z):=f(\lambda^{-1}r,\lambda^{-1}r',z),$ for every $\lambda\in\R_+.$\\

\nt In particular$,$ if $f=f(z)$ has constant coefficients$,$ $\op_M^\gamma(f)$ commutes with $\delta_\lambda.$ 
\end{rem}

\nt Mellin pseudo-differential operators\index{Mellin!pseudo-differential operators}\index{pseudo-differential operators!Mellin $-$} act in suitable scales of weighted Sobolev spaces\index{Sobolev spaces!weighted $-$}\index{weighted!spaces!Sobolev $-$}\index{spaces!Sobolev $-$!weighted $-$}. In the variant of amplitude functions of the form (\ref{opvalsym}) for a closed compact $C^\infty$ manifold $X,$ $n=\dim X,$ we have the following definition.

\begin{defn}\label{Hsg}
The space $\H^{s,\gamma}(\X),$\newnot{symbol:HsGa} for $s,\gamma\in\R,$ is defined to be the completion of $C^\infty_0(\X)$ with respect to the norm
$$\norm{u}_{\H^{s,\gamma}(\X)}:=\Big\{\frac1{2\pi i}\int_{\Gamma_{\frac{n+1}2-\gamma}}\hspace{-5mm}\norm{\RR^s(\Im z)(Mu)(z)}^2_{L^2(X)}dz\Big\}^{\frac12}.$$

\nt Here $n=\dim X,$ and $\RR^s(\lambda)\in L^s_\clas(X;\R_\lambda)$ is an order reducing family in the sense of \textup{Theorem \ref{red}}.
\end{defn}

\begin{rem}
\begin{description}
\item[$\mathrm{(i)}$] Different choices of order reducing families in \textup{Definition \ref{Hsg}} give rise to equivalent norms in $\H^{s,\gamma}(\X).$
\item[$\mathrm{(ii)}$] We have 
$$\H^{s,\gamma}(\X)=r^\gamma\H^{s,0}(\X)$$
for every $s\in\R,$ and
$$\H^{0,0}(\X)=r^{-\frac n2} L^2(\R_+\times X)$$
where the $L^2$-space is based on $drdx,$ with $dx$ being associated with a fixed Riemannian metric on $X.$
\item[$\mathrm{(iii)}$] We have
$$\H^{s,\gamma}(\X)\subset H^s_\loc(\X)$$
for every $s,\gamma\in\R.$
\end{description}
\end{rem}

\begin{rem}\label{2.6}
The space $\H^{s,\gamma}(\X)$ can also be defined in terms of $\H^{s,\gamma}(\R_+\times\R^n),$ combined with charts $\R_+\times U\rightarrow\R_+\times\R^n$ on $\R_+\times X,$ using a partition of unity$,$ where $\H^{s,\gamma}(\R_+\times\R^n)$ is the completion of $C^\infty_0(\R_+\times\R^n)$ with respect to the norm
$$\Big\{\frac1{2\pi}\int_{\R^n}\int_{\Gamma_{\frac{n+1}2-\gamma}}\hspace{-5mm}\ang{z,\xi}^{2s}|(M_{r\rightarrow z}F_{x\rightarrow\xi}u)(z,\xi)|^2dzd\xi\Big\}^{\frac12},$$
$\ang{z,\xi}:=(1+|z|^2+|\xi|^2)^{\frac12}.$ Here $M_{r\rightarrow z}$ is the Mellin transform on $\R_+$ and $F_{x\rightarrow\xi}$ the Fourier transform on $\R^n.$
\end{rem}

\nt For purposes below we also define cylindrical Sobolev spaces\index{Sobolev spaces!cylindrical $-$}\index{spaces!Sobolev $-$!cylindrical $-$}
\begin{equation}\label{HS13}
H^s(\R\times X)\;\textup{ on }\;\R\times X\ni(t,x)
\end{equation}
as the completion of $C_0^\infty(\R\times X)$ with respect to the norm
$$\norm{u}_{H^s(\R\times X)}:=\big\{\!\int\!\norm{\RR^s(\tau)(Fu)(\tau)}^2_{L^2(X)}d\tau\big\}^{\frac12}$$
with an order reducing family $\RR^s(\tau)\in L_\clas^s(X;\R_\tau)$ and the one-dimensional Fourier transform $(Fu)(\tau)=\int e^{-it\tau}u(t)dt$ (the $x$-variable is suppressed in this notation). Alternatively, we can define (\ref{HS13}) in terms of local expressions combined with a global construction along $X,$ using a partition of unity, similarly as in Remark \ref{2.6}. Let us set
\begin{equation}\label{Fu13}
\hat H^s(\R_\tau\times X):=\{(F_{t\rightarrow\tau}u)(\tau,\cdot):u(t,\cdot)\in H^s(\R_t\times X)\}.
\end{equation}

\nt There are other useful and important variants of Sobolev spaces on $\X$ in our calculus.\\

\nt Let us choose an order reducing family $\til p(\til\rho,\til\eta)\in L^s_\clas(X;\R^{1+q}_{\til\rho,\til\eta}),$ $s\in\R,$ in the sense of Theorem \ref{red}, now with the parameters $(\til\rho,\til\eta)\in\R^{1+q}$ for some $q\ge1,$ and set $p(r,\rho,\eta):=\til p([r]\rho,[r]\eta).$ Here $r\mapsto[r]$\label{kleta} means a fixed strictly positive function in $C^\infty(\R)$ such that $[r]=|r|$ for $r>C,$ $C>0.$ Let us interpret the Cartesian product $\R\times X\ni(r,x)$ as a manifold with conical exits $r\rightarrow\pm\infty,$ in this case denoted by $X^\asymp.$\label{xsymp}\\

\nt By a cut-off function (on the half-axis) we understand any $\omega\in C^\infty_0(\Rr)$ such that $\omega(r)=1$ in a neighbourhood of $r=0.$\index{cut-off function}\\

\nt Given a compact manifold $M$ with conical singularity $v$, we define $H^{s,\gamma}(M)$\label{HsgM} for $s,\gamma\in\R$ to be the subspace of all $u\in H^s_\loc(M\setminus\{v\})$ such that $\omega(\chi^{-1})^*u\in\HsgX$ for any cut-off function $\omega$ on $\R_+$ and a singular chart $\chi:V\rightarrow\Xw.$

\begin{defn}
The space $H^s_\cone(X^\asymp),$\newnot{Hscone} $s\in\R,$ is defined to be the completion of $C_0^\infty(\R\times X)$ with respect to the norm
$$\set{\int_\R\norm{[r]^{-s}\Op_r(p)(\eta^1)u(r)}^2_{L^2(X)}dr}^{\frac12}$$
where $\Op_r(\cdot)$ refers to the pseudo-differential action in $r$ $($see the formula \textup{(\ref{Op})),} and the parameter $\eta$ is fixed at a sufficiently large absolute valued $\eta^1.$ Moreover$,$ we set
\begin{equation}\label{mew11}
H^{s;g}_\cone(X^\asymp):=\ang r^{-g}H^s_\cone(X^\asymp)
\end{equation}
for any $g\in\R.$
\end{defn}

\nt The spaces $H^{s;g}_\cone(X^\asymp)$ are adequate in the theory of pseudo-differential operators on manifolds with conical exits. In our case we need them on $\X,$ and we set
$$H^{s;g}_\cone(\X):=H^{s;g}_\cone(X^\asymp)|_{\X}.$$\label{Hscong}

\nt From the definition we easily see that 
$$H^{s;g}_\cone(\X)\subset H^s_\loc(\X)$$
for all $s,g\in\R.$\\

\nt In the definition of cone operators below we employ pseudo-differential operators that are adapted to the spaces $H^s_\cone(X^\asymp),$ briefly referred to as operators on $X^\asymp$ with exit behaviour, cf. \cite{Schu20}, \cite{Kapa10}. The corresponding subspace $L^{\mu;\nu}_{(\clas)}(X^\asymp)\subset L^\mu_{(\clas)}(\R\times X)$ of operators of order $\mu\in\R$ and exit order $\nu\in\R$ is defined in terms of symbols $a(\til x,\til\xi)\in S^{\mu;\nu}_{(\clas)}(\R^{1+n}_{\til x}\times\R^{1+n}_{\til\xi}),$ $n=\dim X.$ The space $S^{\mu;\nu}(\R^{1+n}\times\R^{1+n})$ is characterised as the set of all smooth $a(\til x,\til\xi)$ such that
$$|D^\alpha_{\til x}D^\beta_{\til\xi}a(\til x,\til\xi)|\le c\ang{\til\xi}^{\mu-|\beta|}\ang{\til x}^{\nu-|\alpha|}$$
for all $\alpha,\beta\in\N^{1+n},$ $(\til x,\til\xi)\in\R^{1+n}\times\R^{1+n},$ and constants $c=c(\alpha,\beta)>0.$ Moreover, we have the space $S^\mu_\clas(\R^{1+n})$ of symbols with constant coefficients with its natural (nuclear) \Fr topology, and we set
$$S^{\mu;\nu}_\clas(\R^{1+n}\times\R^{1+n}):=S^\nu_\clas(\R^{1+n}_{\til x})\TenP S^\mu_\clas(\R^{1+n}_{\til\xi})$$
with $\TenP$ being the completed projective tensor product, called classical symbols in $(\til x,\til\xi)$ of order $(\mu,\nu).$ Now we set
$$L^{\mu;\nu}_{(\clas)}(\R^{1+n}):=\set{\Op_{\til x}(a):a(\til x,\til\xi)\in S^{\mu;\nu}_{(\clas)}(\R^{1+n}\times\R^{1+n})}.$$

\nt In order to define $L^{\mu;\nu}_{(\clas)}(X^\asymp)$ for smooth compact $X$ we choose an open covering of $X^\asymp$ by sets of the form $(-1,1)\times U_j$ and $\R_\pm\times U_j,$ where $U_j,$ $1\le j\le L,$ form an open covering of $X.$ Those are chosen together with charts $x_j:U_j\rightarrow B$ where $B$ is the open unit ball in $\R^n,$ and we define an atlas on $\R\times X(\cong X^\asymp)$ by
$$\begin{array}{l}
\chi_j:(-1,1)\times U_j\rightarrow(-1,1)\times B,\\[4mm]
\chi_j^\pm:\R_\pm\times U_j\rightarrow\Gamma_\pm:=\set{(r,rx_j(y))\in\R^{1+n}:r\in\R_\pm,y\in U_j},
\end{array}$$
for $j=1,\dots,L.$ On $X^\asymp$ we fix a partition of unity $\{\varphi_1,\dots,\varphi_{3L}\}$ obtained in the form
$$\varphi_j:=\frac{\til\varphi_j}{\sum_{j=1}^{3L}\til\varphi_j}$$
where
\begin{itemize}
\item[\textup{(i)}] $\til\varphi_j\in C_0^\infty((-1,1)\times U_j)$ for $j=1,\dots,L,$
\item[\textup{(ii)}] $\til\varphi_j$ for $j=L+1,\dots,2L$ in the coordinates $\til x\in\Gamma_+$ are of the form
$$\til\varphi_j(\til x)=(1-\omega(|\til x|))\Phi_j(\til x)$$
for a cut-off function $\omega(r)$ on the half haxis (say, equal to 1 for $0\le r\le1/2$ and vanishing for $r\ge2/3)$ and a $\Phi_j(\til x)\in C^\infty(\Gamma_+)$ such that $\Phi_j(\lambda\til x)=\Phi_j(\til x)$ for all $\lambda>0,$ and $\supp\Phi_j|_{1\times B}$ compact in $1\times B,$
\item[\textup{(iii)}] $\til\varphi_j$ for $j=2L+1,\dots,3L$ are also of the same form as in (ii), now for $\Phi_j(\til x)\in C^\infty(\Gamma_-)$ with $\Phi_j(\lambda\til x)=\Phi_j(\til x)$ for all $\lambda>0$ and $\supp\Phi_j|_{\{-1\}\times B}$ compact in $\{-1\}\times B,$
\end{itemize}
such that $\sum_{j=1}^{3L}\til\varphi_j$ never vanishes.\\

\nt Moreover, we choose the functions $\psi_j,$ $1\le j\le3L,$ of analogous kind such that $\varphi_j\prec\psi_j$ for every $j.$\\

\nt Now $L^{\mu;\nu}_{(\clas)}(X^\asymp)$ is defined to be the set of all operators
$$A=\sum_{j=1}^{3L}\varphi_jA_j\psi_j+C$$
where $\chi_{j*}(\varphi_jA_j\psi_j)\in L^\mu_{(\clas)}((-1,1)\times B)$ and $\chi_{j*}(\varphi_jA_j\psi_j)\in L^{\mu;\nu}_{(\clas)}(\R^{1+n})$ for $j=1,\dots,L$ and $j=L+1,\dots,3L,$ respectively, while $C$ is an operator with kernel in $\S(\R\times\R,C^\infty(X\times X)).$

\begin{thm}
Every $A\in L^{\mu;\nu}_{(\clas)}(X^\asymp)$ induces continuous operators
$$A:H^{s;g}_\cone(X^\asymp)\rightarrow H^{s-\mu;g-\nu}_\cone(X^\asymp)$$
for all $s,g\in\R.$
\end{thm}

\begin{defn}\label{ksgg}
We set
$$\K^{s,\gamma;g}(\X):=\omega\H^{s,\gamma}(\X)+(1-\omega)H^{s;g}_\cone(\X),\newnot{symbol:Ksgag}$$
$s,\gamma,g\in\R,$ where $\omega$ is any fixed cut-off function on the half-axis$,$ and
$$\K^{s,\gamma}(\X):=\K^{s,\gamma;0}(\X).\newnot{symbol:Ksga}$$

\nt Moreover$,$ we define
$$\S^\gamma(\X):=\Projlim{N\in\N}\K^{N,\gamma;N}(\X),\quad\S_\O(\X):=\Projlim{N\in\N}\K^{N,N;N}(\X)\newnot{symbol:Sga}\newnot{symbol:SgaO}$$
and
\begin{equation}\label{new10a}
\K^{\infty,\gamma}(\X):=\Projlim{N\in\N}\K^{N,\gamma}(\X),\quad\K^{s,\infty}(\X):=\Projlim{N\in\N}\K^{s,N}(\X).
\end{equation}
If $\Theta=(\vartheta,0]$ is finite we also write
$$\S^\gamma_\Theta(\X):=\Projlim{j\in\N}\S^{\gamma-\vartheta-\frac1{1+j}}(\X).\newnot{symbol:SgaTh}$$
\end{defn}

\nt In the space $\K^{0,0}(\X)=r^{-\frac n2}L^2(\R_+\times X)$ we fix the scalar product
\begin{equation}\label{ScPr}
(u,v)=\int u(r,x)\overline{v}(r,x)r^ndrdx
\end{equation}
where $dx$ refers to a fixed Riemannian metric on $X,$ $n=\dim X.$ Then, $(\cdot,\cdot):C_0^\infty(\X)\times C_0^\infty(\X)\rightarrow\C$ extends to a non-degenerate sesquilinear pairing
\begin{equation}\label{Pair}
(\cdot,\cdot):\K^{s,\gamma;g}(\X)\times\K^{-s,-\gamma;-g}(\X)\rightarrow\C
\end{equation}
for every $s,\gamma,g\in\R.$ Later on we often refer to groups of isomorphisms on our spaces, for instance,
\begin{equation}\label{Klhom}
\Kl^g:\K^{s,\gamma;g}(\X)\rightarrow\K^{s,\gamma;g}(\X),\quad(\Kl u)(r,x):=\lambda^{\frac{n+1}2+g}u(\lambda r,x),
\end{equation}
$\lambda\in\R_+.$

%1.2.3
%
%
%
%
%
\subsection{Discrete asymptotics in cone spaces}\label{1.2.3}
\begin{defn}\index{asymptotic!discrete $-$ type}\index{discrete asymptotic type}
A sequence
\begin{equation}\label{discras}
\P=\set{(p_j,m_j)}_{j=0,\dots,N}
\end{equation}
of pairs $(p_j,m_j)\in\C\times\N,$ for $N=N({\P})\in\N\cup\{\infty\},$ is said to be a discrete
asymptotic type$,$ associated with the weight data
$(\gamma,\Theta),$ with a weight $\gamma\in\R$ and a $($half-open$)$ weight interval $\Theta=(\vartheta,0],$ $-\infty\le\vartheta<0,$ if the set $\pi_\C{\P}=\{p_j\}_{j=0,\dots,N}\subset\C$\label{Pip} is contained in
$\big\{\frac{n+1}{2}-\gamma+\vartheta<\Re
z<\frac{n+1}{2}-\gamma\big\},$ where $n=\dim X,$ furthermore
$N({\P})<\infty$ for $\vartheta>-\infty,$ and $\Re
p_j\rightarrow-\infty$ for $j\rightarrow\infty$ in the case
$\vartheta=-\infty$ and $N(\P)=\infty.$
\end{defn}

\nt We write 
$$T^\delta\P:=\{(p_j+\delta,m_j)\}_{j=0,\dots,N}\label{TdP}$$
for any $\delta\in\R.$ A discrete asymptotic type is said to satisfy the shadow condition\index{shadow condition} if
$$(p,m)\in\P\quad\Rightarrow\quad(p-l,m)\in\P$$
for all $l\in\N$ with $\Re p-l>\frac{n+1}2-\gamma+\vartheta.$\\

\nt The complex conjugate of a discrete asymptotic type (\ref{discras}) is defined as
$$\overline\P:=\{(\overline p_j,m_j)\}_{j=0,\dots,N}.$$\label{Pcon}

\nt If $\P$ is a discrete asymptotic type associated with $(\gamma,\Theta),$ $\vartheta>-\infty,$ and $\omega$ a fixed cut-off function, we set
\begin{equation}\newnot{singdiscr}
{\E}_{\P}(\X):=\Big\{\sum^N_{j=0}\sum^{m_j}_{k=0}
c_{jk}(x)\omega(r)r^{-p_j}\log^kr:c_{jk}\in C^\infty(X)\Big\}.
\end{equation}
This space is isomorphic to a direct sum of $\sum^N_{j=0}(m_j+1)$ copies of the space $C^\infty(X)$ and as such a \Fr space. Observe that if we set 
$${\E}_{\P}(\R_+):=\Big\{\sum_{j,k}c_{jk}\omega(r)r^{-p_j}\log^kr:c_{jk}\in\C\Big\},$$
which is of finite dimension $\sum_{j=0}^N(m_j+1),$ we can also write
\begin{equation}\newnot{asnew11}
{\E}_{\P}(\X)=C^\infty(X,{\E}_{\P}(\R_+)).
\end{equation}

\begin{rem}
We have 
$$\E_\P(\X)\subset\K^{s,\gamma;\infty}(\X)$$
for every $s\in\R.$
\end{rem}

\nt Let us define the space of flat functions (of flatness $-\vartheta-0$ relative to the weight\index{flatness!relative to a weight}\index{weight!flatness relative to a $-$} $\gamma)$ as
$$\K^{s,\gamma;g}_\Theta(\X):=\Projlim{m\in\N}\K^{s,\gamma-\vartheta-\frac1{m+1};g}(\X)\newnot{symbol:KTh}$$
in the \Fr topology of the projective limit.

\begin{defn}
\begin{itemize}
\item[$\mathrm{(i)}$] Let $\P$ be a discrete asymptotic type associated with $(\gamma,\Theta),$ $\Theta$ finite$;$ we set
\begin{equation}\label{new14}
\K^{s,\gamma;g}_\P(\X):=\K^{s,\gamma;g}_\Theta(\X)+\E_\P(\X)
\end{equation}
$($which is a direct sum$);$
\item[$\mathrm{(ii)}$] if $\P$ is a discrete asymptotic type$,$ $\Theta$ infinite$,$ we form $$\P_l:=\set{(p,m)\in\P:\Re p>\frac{n+1}2-\gamma-(l+1)},$$
$k\in\N,$ which is associated with $\Theta_l=(-(l+1),0],$ and we set
$$\K^{s,\gamma;g}_\P(\X):=\Projlim{l\in\N}\K^{s,\gamma;g}_{\P_l}(\X).$$ 
\end{itemize}
\nt Analogously as before we write $\K^{s,\gamma}_\Theta(\X)$\newnot{Kcfl} and $\K^{s,\gamma}_\P(\X)$\newnot{Kcas} when $g=0.$ Moreover$,$ we set
$$\S^\gamma_\P(\X):=\Projlim{N\in\N}\K^{N,\gamma;N}_\P(\X).$$\label{SgP}
\end{defn}

\begin{rem}
Setting 
$$E^j:=\K^{s,\gamma-\beta_j;g}(\X)+\E_{\P_j}(\X)$$
for $\beta_j:=\max(-(j+1),\vartheta)+\frac1{j+1},$ and $\P_j:=\{(p,m)\in\P:\Re p\ge\max(-(j+1),\vartheta)+\frac1{j+2}\},$ we obtain a Hilbert space for every $j\in\N,$ with continuous embeddings $E^{j+1}\hookrightarrow E^j\hookrightarrow\dots\hookrightarrow\K^{s,\gamma;g}(\X),$ where $\K^{s,\gamma;g}_\P(\X)=\Projlim{j\in\N}E^j.$ Observe that $(\ref{Klhom})$ restricts to a group of isomorphisms $\Kl^g:\K^{s,\gamma;g}_\P(\X)\rightarrow\K^{s,\gamma;g}_\P(\X),$ $\lambda\in\R_+,$ induced from corresponding isomorphisms $\Kl^g:E^j\rightarrow E^j,$ for every $j\in\N.$
\end{rem}

\begin{thm}\label{Th12}
Let $\P$ be a discrete asymptotic type associated with the weight data\index{weight!data} $(\gamma,\Theta)$ and $\chi$ be a $\pi_\C\P$-excision function. 
\begin{itemize}
\item [\textup{(i)}] Let $u(r,x)\in\K^{s,\gamma}_\P(\X)$ and $\omega(r)$ a cut-off function. Then
$$M_{\gamma-\frac n2,r\rightarrow z}(\omega u)(z,x)$$
extends from $\Gamma_{\frac{n+1}2-\gamma}$ to a $H^s(X)$-valued meromorphic function $f(z,x)$ in $z,$ $\frac{n+1}2-\gamma+\vartheta<\Re\,z<\frac{n+1}2-\gamma,$ with poles at the points $p_j$ of multiplicity $m_j+1$ and Laurent coefficients at $(z-p_j)^{-(k+1)},$ $0\le k\le m_j,$ belonging to $C^\infty(X)$ such that $\chi(z)f(z,x)|_{\Gamma_\beta}\in\hat H^s(\Gamma_\beta\times X)$ for every $\frac{n+1}2-\gamma+\vartheta<\beta\le\frac{n+1}2-\gamma,$ uniformly in compact $\beta$-intervals $($cf. the notation in $(\ref{Fu13})$ where $\tau$ is to be replaced by $\Im z$ for $z\in\Gamma_\beta).$
\item [\textup{(ii)}] Let $f(z,x)\in\hat H^s(\Gamma_{\frac{n+1}2-\gamma}\times X)$ be a function that extends to a $H^s(X)$-valued function which is meromorphic as in \textup{(i),} where $\chi(z)f(z,x)|_{\Gamma_\beta}\in\hat H^s(\Gamma_\beta\times X)$\label{HhatG} for every $\frac{n+1}2-\gamma+\vartheta<\beta\le\frac{n+1}2-\gamma,$ uniformly in compact $\beta$-subintervals. Then for every cut-off function $\omega(r)$ we have
$$\omega(r)\Set{M_{\gamma-\frac n2,r\rightarrow z}f}(r,x)\in\K^{s,\gamma}_\P(\X).$$
\end{itemize}
\end{thm}

\begin{proof}
(i) Without loss of generality we assume that $\Theta$ is finite. By virtue of (\ref{new14}) every $u\in\K^{s,\gamma}_\P(\X)$ has a decomposition
$$u(r,x)=u_{\mathrm{flat}}(r,x)+u_{\mathrm{sing}}(r,x)$$
for $u_{\mathrm{flat}}\in\K^{s,\gamma}_\Theta(\X),$ $u_{\mathrm{sing}}\in\E_\P(\X).$ We have $\omega u_{\mathrm{flat}}\in r^\delta\H^{s,\gamma}(\X)$ for every $0\le\delta<-\vartheta,$ and from the nature of $\H^{s,\gamma}$-spaces we know that $(M\omega u_{\mathrm{flat}})(z,x)$ extends from $\Gamma_{\frac{n+1}2-\gamma}$ to a holomorphic $H^s(X)$-valued function in the strip $\frac{n+1}2-\gamma+\vartheta<\Re\,z<\frac{n+1}2-\gamma.$ Therefore, it remains to characterise $Mu_{\mathrm{sing}}(z,x).$ By definition $u_{\mathrm{sing}}$ is a finite linear combination of expressions of the form
$$\omega(r)c(x)r^{-p}\log^kr$$
for $p\in\pi_\C\P,$ $c\in C^\infty(X),$ $k\in\N.$ The Mellin transform of such a function is known to be a $C^\infty(X)$-valued meromorphic function with a pole at $p$ of multiplicity $k+1,$ and $\chi M(\omega cr^{-p}\log^kr)$ is a Schwartz function on every line $\Gamma_\beta,$ $\beta\in\R,$ uniformly in finite $\beta$-intervals. This gives us altogether the assertion $\textup{(i)}.$\\

\nt (ii) Let $f(z,x)$ be a function with the assumed properties. Then for every $\varepsilon>0$ there is a $\sigma<\gamma-\vartheta$ such that $\gamma-\vartheta-\sigma<\varepsilon,$ and the finite set $\pi_\C\P$ is contained in $\{z:\Re\,z>\frac{n+1}2-\sigma\}.$ We have
\begin{equation}\label{int2}
\omega(r)\frac1{2\pi i}\int_{\Gamma_{\frac{n+1}2-\gamma}}\hspace{-1cm}r^{-z}f(z,x)dz=\omega(r)\frac1{2\pi i}\Big(\int_{\Gamma_{\frac{n+1}2-\sigma}}\hspace{-1cm}r^{-z}f(z,x)dz+\sum_{j=0}^N\int_{|z-p_j|=\delta}\hspace{-1cm}r^{-z}f(z,x)dz\Big)
\end{equation}
for any sufficiently small $\delta>0,$ with counter clockwise taken integration contours. The left hand side of (\ref{int2}) belongs to $\KsgX$ and the first summand in the right to $\K^{s,\gamma}_\Theta(\X)$ (since $\sigma$ with the indicated properties is arbitrary). It remains to characterise the other summands on the right of (\ref{int2}). Close to $p_j$ the function $f(z,x)$ is a finite sum
$$\sum_{k=0}^{m_j}c_{jk}(x)(z-p_j)^{-(k+1)}+h_j(z,x)$$
for coefficients $c_{jk}\in C^\infty(X)$ and $h_j$ holomorphic in $z$ near $p_j,$ $k=0,\dots,m_j.$ By Cauchy's theorem the integrals over $r^{-z}h_j(z,x)$ vanish. Therefore, we have to characterise
\begin{equation}\label{int3}
\omega(r)\frac1{2\pi i}\int_{|z-p_j|=\delta}\hspace{-1cm}r^{-z}c_{jk}(x)(z-p_j)^{-(k+1)}dz
\end{equation}
for every $j,k.$ Writing $r^{-z}=r^{-p_j}r^{p_j-z}=r^{-p_j}e^{(p_j-z)\log r}$ the expression (\ref{int3}) takes the form
$$\omega(r)r^{-p_j}c_{jk}(x)\frac1{2\pi i}\sum_{l=0}^\infty\frac1{l!}\int_{|z-p_j|=\delta}\hspace{-1cm}(p_j-z)^l\log^lr(z-p_j)^{-(k+1)}dz.$$
Only the terms for $l=k$ give a non-vanishing contribution, and we just obtain the singular functions of asymptotics of type $\P.$
\end{proof}

\begin{rem}
Assume that a function $f_0$ on some weight line $\Gamma_{\frac{n+1}2-\gamma}$ has an extension to a meromorphic function $f$ in some strip of the complex plane$;$ then for notational convenience we often identify $f_0$ and $f.$ In particular$,$ in modification of the original meaning of $M_{\gamma-\frac n2}:u\mapsto Mu|_{\Gamma_{\frac{n+1}2-\gamma}}=:f_0,$ we interpret $M_{\gamma-\frac n2}$ also as a mapping from $u$ to $f$ when such an extension exists.
\end{rem}

\begin{rem}\label{Mer}
Let $u(r,x)$ be as in \textup{Theorem \ref{Th12} (i)} and $\varphi(r)\in C_0^ \infty(\R_+).$ Then
$$f(z,x):=M_{\gamma-\frac n2,r\rightarrow z}(\varphi u)(z,x)$$
extends from $\Gamma_{\frac{n+1}2-\gamma}$ to a $H^s(X)$-valued meromorphic function in $z$ with poles at the points $p_j-l,$ $l\in\N,$ $\frac{n+1}2-\gamma+\vartheta<\Re p_j-l,$ of multiplicity $m_j+1$ and Laurent coefficients in $C^\infty(X),$ for all $j,l.$ This phenomenon is a motivation for the above definition of the shadow condition of asymptotic types. 
\end{rem}

\begin{defn}
Let $\P$ be a discrete asymptotic type associated with the weight data $(\gamma,\Theta),$ $M$ a compact manifold with conical singularity $v$ and $\chi:V\rightarrow\Xw$ a fixed singular chart. Then
$$H^{s,\gamma}_\P(M)\label{HsgP}$$
$s\in\R,$ is defined to be the subspace of all $u\in H^s_\loc(M\setminus\{v\})$ such that for any cut-off function $\omega$ on $\R_+$ we have $\omega(\chi^{-1}_\reg)^*u\in\K^{s,\gamma}_\P(\X).$
\end{defn}

\nt The space $H^{s,\gamma}_\P(M)$ is a \Fr space in a natural way, and the same is true for $H^{\infty,\gamma}_\P(M)=\bigcap_{s\in\R}H^{s,\gamma}_\P(M).$ From the scalar product of $H^{0,0}(M)$ we have a non-degenerate sesquilinear pairing
$$H^{s,\gamma}(M)\times H^{-s,-\gamma}(M)\rightarrow\C$$
for every $s,\gamma\in\R.$ This allows us to define formal adjoints\index{formal adjoint operator} $A^*$ of operators\index{operators!formal adjoint $-$}
$$A:H^{s,\gamma}(M)\rightarrow H^{s-\mu,\til\gamma}(M)$$
that are continuous for all $s\in\R,$ for some reals $\gamma,\til\gamma,\mu,$ namely, by the relation
$$(Au,v)_{H^{0,0}(M)}=(u,A^*v)_{H^{0,0}(M)}$$
for all $u,v\in C_0^\infty(M\setminus\{v\}).$ Then $A^*$ induces continuous operators
$$A^*:H^{s,-\til\gamma}(M)\rightarrow H^{s-\mu,-\gamma}(M)$$
for all $s\in\R.$\\

\nt Analogously, the $\K^{0,0}(\X)$-scalar product gives us a non-degenerate sesquilinear pairing
$$\K^{s,\gamma}(\X)\times\K^{-s,-\gamma}(\X)\rightarrow\C$$
for every $s,\gamma\in\R,$ and any operator
$$A:\K^{s,\gamma}(\X)\rightarrow \K^{s-\mu,\til\gamma}(\X)$$
which is continuous for all $s\in\R$ and some $\gamma,\til\gamma,\mu,$ has a formal adjoint $A^*$ defined by
$$(Au,v)_{\K^{0,0}(\X)}=(u,A^*v)_{\K^{0,0}(\X)}$$
for all $u,v\in C_0^\infty(\X).$ Then $A^*$ induces continuous operators
$$A^*:\K^{s,-\til\gamma}(\X)\rightarrow\K^{s-\mu,-\gamma}(\X)$$
for all $s\in\R.$

%1.2.4
%
%
%
%
%
\subsection{Background from the calculus for conical singularities}\label{1.2.4}

Let $M$ be a manifold with conical singularities $S.$ In the general part of this exposition for simplicity we assume that $S$ consists of a single point $v$ (most of the constructions have straightforward generalisations to the case of finitely many conical singularities). Locally near $v,$ we fix a singular chart $\chi$ and a corresponding splitting of variables $(r,x)\in\R_+\times X,$ cf. Definition \ref{1.1.1} (ii). If we motivate our constructions by the task to express parametrices of elliptic operators $A\in\Diff_{\deg}^\mu(M)$ (see Definition \ref{Diffdeg}) the specific novelties compared with the smooth case have to be expected in a neighbourhood of $v.$ Given the operator $A$ in the form (\ref{new2}) we can write
$$A=r^{-\mu}\op_M^{\gamma-\frac n2}(h)$$
for $h(r,z):=\sum_{j=0}^\mu a_j(r)z^j;$ here $n=\dim X,$ and $\gamma\in\R$ is a weight that we fix later on in connection with ellipticity. The Mellin symbol $h(r,z)$ is $\Diff^\mu(X)$-valued and holomorphic in the covariable $z.$ We have $h(r,z)\in C^\infty(\Rr,M_\O^\mu(X)),$ cf. Definition \ref{1.2.16} below. More generally, as noted in Section \ref{1.2.2} we may admit $L^\mu_{(\clas)}(X;\Gamma_{\frac{n+1}2-\gamma})$-valued Mellin symbols for any $\gamma\in\R,$ cf. (\ref{opvalsym}). In the following definition we admit operator functions depending on $\eta\in\R^q$ which already prepares material for the edge calculus.

\begin{defn}\label{1.2.16}
The space $M_\O^\mu(X;\R^q)$ for a closed compact $C^\infty$ manifold $X,$ $\mu\in\R,$ $q\in\N,$ is defined to be the set of all $h(z,\eta)\in\A(\C,L^\mu_{\clas}(X;\R^q))$ such that
$$h(\beta+i\rho,\eta)\in L^\mu_{\clas}(X;\Gamma_\beta\times\R^q)$$
for every $\beta\in\R,$ uniformly in compact $\beta$-intervals.
\end{defn}

\nt From the definition it is immediate that $M_\O^\mu(X;\R^q)$ is a \Fr space in a natural way. We set
$$M_\O^\mu(X):=M_\O^\mu(X;\R^0),\label{MOmX}$$
and we simply write $M_\O^\mu(\R^q)$\label{MOm} when $\dim X=0$ (in this case $L^\mu_{\clas}(X;\R^q)$ is replaced by $S^\mu_{\clas}(\R^q)$ in the canonical \Fr topology), and $M_\O^\mu$ when $q=\dim X=0.$

\begin{thm}\label{Me}
For every $\til p(r,\til\rho)\in C^\infty(\Rr,L^\mu_\clas(X;\R_{\til\rho}))$ there exists an $h(r,z)$ in the space $C^\infty(\Rr,M^\mu_\O(X))$ such that for $p(r,\rho)=\til p(r,r\rho)$ we have
$$\Op_r(p)=\op_M^{\gamma-\frac n2}(h)\qquad\mod L^{-\infty}(\X)$$
for all $\gamma\in\R$ $($as standard pseudo-differential operators $C_0^\infty(X)\rightarrow C^\infty(X)).$
\end{thm}

\nt Theorem \ref{Me} is crucial for the cone calculus (and later on, in parameter-dependent form also for the edge calculus). A proof may be found, for instance, in \cite{Schu20}. A particularly trasparent alternative proof is given in \cite{Krai2}.

\begin{defn}\label{Ell13}
An $h(z)\in M^\mu_\O(X)$ is called elliptic\index{elliptic!operator}\index{operators!elliptic $-$} $($of order $\mu)$ if there is a $\beta\in\R$ such that $h(\beta+i\rho)$ is parameter-dependent elliptic in $L^\mu_\clas(X;\Gamma_\beta).$
\end{defn}

\begin{rem}
The latter definition of ellipticity of $h(z)\in M^\mu_\O(X)$ $($of order $\mu)$ is independent of $\beta,$ i.e.$,$ if $h(\beta+i\rho)$ is parameter-dependent elliptic in $L^\mu_\clas(X;\Gamma_\beta)$ then $h(\delta+i\rho)$ is parameter-dependent elliptic in $L^\mu_\clas(X;\Gamma_\delta)$ for every $\delta\in\R.$
\end{rem}

\begin{thm}
Let $\til p(\til\rho)\in L^\mu_\clas(X;\R_{\til\rho})$ be parameter-dependent elliptic$,$ and form $h(z)\in M^\mu_\O(X)$ according to \textup{Theorem \ref{Me}}. Then $h(z)$ is elliptic in the sense of \textup{Definition \ref{Ell13}}.
\end{thm}

\begin{thm}
Let $h(z)\in M_\O^\mu(X)$ be elliptic $($of order $\mu);$ then there is a discrete set $D\subset\C$ such that $D\cap\set{c\le\Re\,z\le c'}$ is finite for every $c\le c'$ where
$$h(z):H^s(X)\rightarrow H^{s-\mu}(X)$$
is an isomorphism for every $z\in\C\setminus D$ and all $s\in\R.$
\end{thm}

\begin{defn}
A sequence
$$\RR=\set{(r_j,n_j)}_{j\in\Z}\index{Mellin!asymptotic type}\index{discrete asymptotic type!for Mellin symbols}\index{asymptotic!discrete $-$ type!for Mellin symbols}$$
of pairs $(r_j,n_j)\in\C\times\N$ is said to be a discrete asymptotic type for Mellin symbols $($also referred to as a Mellin asymptotic type$)$ if the set $\pi_\C{\RR}=\{r_j\}_{j\in\Z}$\label{Pir} intersects the strip $\set{c\le\Re z\le c'}$ in a finite set for every reals $c\le c'.$
\end{defn}

\begin{defn}\label{1.2.24}
Let $\RR$ be a Mellin asymptotic type. Then $M_\RR^{-\infty}(X)$\label{MRinf} is defined as the set of all $f(z)\in\A(\C\setminus\pi_\C\RR,L^{-\infty}(X))$ such that
\begin{itemize}
\item[$\mathrm{(i)}$] $f$ is meromorphic with poles at the points $r_j$ of multiplicity $n_j+1,$ $j\in\Z,$ and in a neighbourhood $U_j$ of $r_j$ we have
$$f(z)=\sum_{k=0}^{n_j}c_{jk}(z-r_j)^{-(k+1)}+h_j(z)$$
for some $h_j\in\A(U_j,L^{-\infty}(X))$ and coefficients $c_{jk}\in L^{-\infty}(X)$ that are operators of finite rank for every $k=0,\dots,n_j,$ $j\in\Z;$
\item[$\mathrm{(ii)}$] for every $\pi_\C\RR$-excision function $\chi(z)$ $($i.e.$,$ $\chi\in C^\infty(\C),$ $\chi(z)=0$ for $\dist(z,\pi_\C\RR)<\varepsilon_0,$ and $\chi(z)=1$ for $\dist(z,\pi_\C\RR)>\varepsilon_1$ for some $0<\varepsilon_0<\varepsilon_1)$ we have
$$\chi(z)f(z)|_{\Gamma_\beta}\in L^{-\infty}(X;\Gamma_\beta)$$
for every $\beta\in\R,$ uniformly on compact $\beta$-subintervals.
\end{itemize}
\end{defn}

\begin{exm}
Let $c\in L^{-\infty}(X)$ be an operator of finite rank and $\omega$ a cut-off function$,$ $p\in\C,$ $\gamma<\frac12-\Re p.$ Then
$$f(z):=M_\gamma(c\,\omega(r)r^{-p}\log^kr)(z)$$
extends to an element in the space $M_\RR^{-\infty}(X)$ for $\RR=\set{(p,k)}$ $($the choice of $\gamma$ is not important$).$
\end{exm}

\begin{defn}
For a given Mellin asymptotic type $\RR$ we set
$$M_{\RR}^{\mu}(X):=M_{\RR}^{-\infty}(X)+M_{\O}^{\mu}(X).\label{MRm}$$
\end{defn}

\begin{rem}\label{2.27}
$f_j\in M_{\RR_j}^{\mu_j}(X),$ $j=1,2,$ implies $f_1+f_2\in M_{\S}^{\max\{\mu_1,\mu_2\}}(X)$ when $\mu_1-\mu_2\in\Z,$ $f_1f_2\in M_{\P}^{\mu_1+\mu_2}(X)$ for resulting $\S$ and $\P.$ Moreover$,$ $f\in M_{\RR}^{-\infty}(X),$ $h\in M_{\Q}^{\mu}(X)$ implies $fh\in M^{-\infty}_{\P}(X)$ for every $\Q,\RR$ and some resulting $\P.$
\end{rem}

\begin{thm}\label{T1.2.28}
$f\in M_{\RR}^{-\infty}(X)$ implies $(1+f(z))^{-1}=1+l(z)$ for an element $l\in M_{\S}^{-\infty}(X)$ for every $\RR$ and some resulting asymptotic type $\S.$
\end{thm}

\begin{thm}
Given $f\in M_{\RR}^{\mu}(X),$ $\pi_\C\RR\cap\Gamma_{\frac{n+1}2-\gamma}=\emptyset,$ and cut-off functions $\omega,\omega'$ the operator $\omega\op_M^{\gamma-\frac n2}(f)\omega'$ induces continuous operators
$$\omega\op_M^{\gamma-\frac n2}(f)\omega':\KsgX\rightarrow\K^{s-\mu,\gamma}(\X),$$
and
$$\omega\op_M^{\gamma-\frac n2}(f)\omega':\K^{s,\gamma}_\P(\X)\rightarrow\K^{s-\mu,\gamma}_\Q(\X),$$
for every $($discrete$)$ asymptotic type $\P$ with some resulting $\Q,$ associated with weight data $(\gamma,\Theta)$ $($for any weight interval $\Theta)$.
\end{thm}

\begin{thm}
For every $f\in L_{\clas}^{\mu}(X;\Gamma_\beta)$ there exists an $h\in M_{\O}^{\mu}(X)$ such that
$$h|_{\Gamma_\beta}=f\quad\mod L^{-\infty}(X;\Gamma_\beta).$$
\end{thm}

%1.2.5
%
%
%
%
%
\subsection{The asymptotic part of the cone algebra}\label{1.2.5}

The pseudo-differential calculus on the (open stretched) cone $\X$ contains two substructures, consisting of Green operators\index{Green!operators}\index{operators!Green $-$} and smoothing Mellin plus Green operators\index{operators!Mellin plus Green $-$}, both with asymptotics. Those operators are automatically produced when we compose operators of simpler structure or construct parametrices in the elliptic case.\\

\nt Let us fix weight data $\g=(\gamma,\delta,\Theta),$ $\gamma,\delta\in\R,$ $\Theta=(\vartheta,0],$ $-\infty\le\vartheta<0.$

\begin{defn}
An operator $G\in\bigcap_{s,g}\L(\K^{s,\gamma;g}(\X),\K^{\infty,\delta;\infty}(\X)),$ $s,g\in\R,$ is called a Green operator on $\X,$ if there are $(G$-dependent$)$ asymptotic types $\P$ and $\Q,$ associated with the weight data $(\delta,\Theta)$ and $(-\gamma,\Theta),$ respectively$,$ such that
$$
\begin{array}{lcl}
G&\!\!:\!\!&\K^{s,\gamma;g}(\X)\rightarrow\S^\delta_\P(\X),\\[4mm]
G^*&\!\!:\!\!&\K^{s,-\delta;g}(\X)\rightarrow\S^\gamma_\Q(\X)
\end{array}
$$
are continuous for all $s,g\in\R.$ Here $G^*$ is the formal adjoint of $G$ in the sense
$$(Gu,v)_{\K^{0,0}(\X)}=(u,G^*v)_{\K^{0,0}(\X)}$$
for all $u,v\in C_0^\infty(\X)$ $($cf. the sesquilinear pairing $(\ref{Pair})).$\\

\nt Let ${L}_G(\X,\g)_{\P,\Q}$\label{AAGPQ} denote the space of all those operators$,$ and set
$${L}_G(\X,\g):=\bigcup{L}_G(\X,\g)_{\P,\Q}\label{AAG}$$
$($the union taken over all $\P,\Q).$
\end{defn}

\begin{rem}
We have
$$G\in{L}_G(\X,\g)\Leftrightarrow r^{\beta}Gr^{-\alpha}\in{L}_G(\X,\mathbf{h})$$
for arbitrary $\alpha,\beta\in\R$ and $\mathbf{h}=(\gamma+\alpha,\delta+\beta,\Theta).$
\end{rem}

\begin{rem}
The spaces ${L}_G(\X,\g)_{\P,\Q}$ are \Fr spaces in a natural way. Moreover$,$ the composition $G\til G$ for $G\in{L}_G(\X,\mathbf{c})_{\S,\Q},$ $\til G\in{L}_G(\X,\mathbf{b})_{\P,\S}$ and $\mathbf{c}=(\sigma,\delta,\Theta),$ $\mathbf{b}=(\gamma,\sigma,\Theta)$ gives rise to a bilinear continuous map
$${L}_G(\X,\mathbf{c})_{\S,\Q}\times{L}_G(\X,\mathbf{b})_{\P,\S}\rightarrow {L}_G(\X,\g)_{\P,\Q}$$
for $\g=(\gamma,\delta,\Theta).$
\end{rem}

\nt We now turn to so-called smoothing Mellin operators on $\X$ with asymptotics. Those operators are (although smoothing on the open manifold $\X)$ of crucial importance for the cone algebra in general. They are in general not compact, and may change the index of elliptic operators.\\

\nt Let us first observe some simple properties of operators of the form
$$A:=r^{-\mu+j}\omega\op_M^{\gamma_j-\frac n2}(f)\omega'$$
for $f\in M_\RR^\infty(X),$ $\mu\in\R,$ $j\in\N,$ $\gamma_j\in\R,$ and cut-off functions $\omega(r),$ $\omega'(r).$ In order that $A$ induces a continuos operator
$$A:\K^{s,\gamma}(\X)\rightarrow\K^{\infty,\gamma-\mu}(\X)$$
we have to make some assumptions on $\RR$ and $\gamma_j.$ In fact, by definition we have
$$\op_M^{\gamma_j-\frac n2}(f)=r^{\gamma_j-\frac n2}\op_M(T^{-(\gamma_j-\frac n2)}f)r^{-(\gamma_j-\frac n2)},$$
$(T^\beta f)(z)=f(z+\beta),$\label{Tbeta} cf. (\ref{new8}). Thus, to apply $\op_M(T^{-(\gamma_j-\frac n2)}f)$ on $\omega'r^{-(\gamma_j-\frac n2)}u$ for $u\in\KsgX$ we need that $r^{-(\gamma_j-\frac n2)}u\in\K^{s,0}(\X),$ i.e., $-\gamma_j+\gamma\ge0,$ and that $T^{-(\gamma_j-\frac n2)}f$ has no poles on the weight line $\Gamma_{\frac12},$ i.e., $\pi_\C\RR\cap\Gamma_{\frac{n+1}2-\gamma_j}=\emptyset.$ Moreover, to reach the space with weight $\gamma-\mu$ we need the condition
$$j+\gamma_j\ge\gamma.$$
In other words we ask
\begin{equation}\label{Me20}
\gamma-j\le\gamma_j\le\gamma\qquad\mathrm{and}\qquad\pi_\C\RR\cap\Gamma_{\frac{n+1}2-\gamma_j}=\emptyset.
\end{equation}

\nt In the following definition we fix weight data\index{weight!data} $\g=(\gamma,\gamma-\mu,\Theta)$\label{boldg} for $\Theta=(-(k+1),0],$ $k\in\N.$

\begin{defn}\label{1.2.34}
An operator $A\in\bigcap_{s,g\in\R}\L(\K^{s,\gamma;g}(\X),\K^{\infty,\gamma-\mu;\infty}(\X))$ is called a smoothing Mellin plus Green operator on $\X,$ associated with the weight data $\g,$ if it has the form
$$A=r^{-\mu}\omega\big\{\sum_{j=0}^kr^j\op_M^{\gamma_j-\frac n2}(f_j)\big\}\omega'+G$$
for some $G\in{L}_G^\mu(\X,\g),$ cut-off functions $\omega,\omega',$ smoothing Mellin symbols $f_j\in M_{\RR_j}^{-\infty}(X)$ and weights $\gamma_j$ such that $(\ref{Me20})$ holds for all $j.$ Let ${L}_{M+G}^\mu(\X,\g)$\label{AAMGm} denote the space of all those operators. In the case $\Theta=(-\infty,0]$ we define ${L}_{M+G}^\mu(\X,\g)$ as the intersection of all ${L}_{M+G}^\mu(\X,(\gamma,\gamma-\mu,(-(k+1),0]))$ over $k\in\N.$
\end{defn}

\nt Let us set
$$\sigma_{\mathrm c}^{\mu-j}(A)(z):=f_j(z),\qquad j=0,\dots,k,\label{consym}$$
called the conormal symbol\index{conormal!symbol}\index{symbol!conormal $-$} of $A$ of (conormal)\index{conormal!order} order $\mu-j.$

\begin{rem}
Given $A,\til A\in{L}_{M+G}^\mu(\X,\g)$ we have $A-\til A\in{L}_G(\X,\g)$ if and only if
$$\sigma_\c^{\mu-j}(A)=\sigma_\c^{\mu-j}(\til A)$$
for all $j=0,\dots,k.$
\end{rem}

\begin{prop}
$A\in{L}_{M+G}^\mu(\X,\g)$ induces continuous operators
\begin{equation}\label{AB}
A:\K^{s,\gamma}(\X)\rightarrow\K^{s-\mu,\gamma-\mu}(\X),\;\K^{s,\gamma}_\P(\X)\rightarrow\K^{s-\mu,\gamma-\mu}_\Q(\X).
\end{equation}
\end{prop}

\begin{thm}\label{comp25}
$A\in{L}_{M+G}^\mu(\X,\mathbf c)$ for $\mathbf c=(\gamma-\nu,\gamma-(\mu+\nu),\Theta)$ and $B\in{L}_{M+G}^\nu(\X,\mathbf b)$ for $\mathbf b=(\gamma,\gamma-\nu,\Theta)$ imply $AB\in{L}_{M+G}^{\mu+\nu}(\X,\mathbf d)$ for $\mathbf d=(\gamma,\gamma-(\mu+\nu),\Theta),$ and we have
\begin{equation}\label{M25}
\sigma_\c^{\mu+\nu-l}(AB)(z)=\sum_{i+j=l}(T^{\nu-i}\sigma_\c^{\mu-j}(A))(z)\sigma_\c^{\nu-i}(B)(z)
\end{equation}
for every $0\le l\le k.$
\end{thm}

\nt The formula (\ref{M25}) is also called the Mellin translation product\index{Mellin!translation product}\index{product!Mellin translation $-$} between the sequences
\begin{equation}\label{Me25}
(\sigma_\c^{\mu-j}(A))_{j=0,\dots,k}\;\textup{ and }\;(\sigma_\c^{\nu-i}(B))_{i=0,\dots,k}.
\end{equation}

\nt Operators of the form 
$$1+A:\KsgX\rightarrow\KsgX$$
for $A\in{L}_{M+G}^0(\X,\g),$ $\g=(\gamma,\gamma,(-(k+1),0]),$ $\gamma\in\R,$ $\Gamma_{\frac{n+1}2-\gamma}\cap\pi_\C\RR=\emptyset,$ for $\sigma_\c^0(A)(z)\in M_\RR^{-\infty}(X),$ form an algebra, as we easily see from Theorem \ref{comp25}. Such operators are special elliptic elements of the cone algebra\index{elliptic!elements of the cone algebra}\index{cone algebra!elliptic elements of the $-$}.

\begin{defn}\label{ell25}
\begin{itemize}
\item[\textup{(i)}] The operator $1+A$ for $A\in{L}_{M+G}^0(\X,\g),$ is called elliptic\index{elliptic!operator}\index{operators!elliptic $-$} if
\begin{equation}\label{inv25}
\sigma_\c^0(1+A)(z)=1+\sigma_\c^0(A)(z):H^s(X)\rightarrow H^s(X)
\end{equation}
is a family of bijective operators for all $z\in\Gamma_{\frac{n+1}2-\gamma},$ and some $s\in\R.$
\item[\textup{(ii)}] An operator $1+B$ for $B\in{L}_{M+G}^0(\X,\g),$ is called a parametrix of $1+A$ if
$$(1+A)(1+B)=1\;\textup{ and }\;(1+B)(1+A)=1$$
modulo ${L}_G(\X,\g).$
\end{itemize}
\end{defn}

\begin{thm}
An elliptic operator $1+A$ has a parametrix $1+B,$ and the sequence of conormal symbols $(\sigma_\c^{-i}(1+B))_{i=0,\dots,k}$ follows by inverting $(\sigma_\c^{-j}(1+A))_{j=0,\dots,k}$ through the Mellin translation product.
\end{thm}

\nt In fact, the ellipticity of $1+A$ allows us to invert the operators of the family (\ref{inv25}). This gives us 
$$(\sigma_\c^0(1+A))^{-1}(z)=1+l(z)$$
for an $l\in M_\S^{-\infty}(X).$ Then, using
$$\sigma_\c^{-l}((1+A)(1+B))=\sum_{i+j=l}(T^{-i}\sigma_\c^{-j}(1+A))\sigma_\c^{-i}(1+B)=\left\{
\begin{array}{rl}
1 & \textup{for }l=0\\[2mm]
0 & \textup{for }l>0
\end{array}\right.
$$
together with $\sigma_\c^{-j}(1+A)=1+\sigma_\c^{-j}(A)$ for $j=0$ and $\sigma_\c^{-j}(1+A)=\sigma_\c^{-j}(A)$ for $j>0$ we can successively determine $\sigma_\c^{-i}(B)$ for all $i=1,2,\dots$ In this process we also apply Remark \ref{2.27}.

%1.2.6
%
%
%
%
%
\subsection{The cone algebra with discrete asymptotics}\label{S1.2.6}

Given a manifold $M$ with conical singularities $S$ by a cone algebra\index{cone algebra} on $M$ we understand a subalgebra of $L_\clas^\mu(M\setminus S)$ that contains close to $S$ in the splitting of variables $(r,x)\in\R_+\times X$ all operators of the form
\begin{equation}\label{Rem}
r^{-\mu}\Op_r(p)\quad\mod\;L^{-\infty}(M\setminus S),
\end{equation}
for all $p(r,\rho)=\til p(r,r\rho),$ $\til p(r,\til\rho)\in C^\infty(\Rr,L^\mu_\clas(X;\R_{\til\rho})),$ together with the parametrices of elliptic elements. The remainders in (\ref{Rem}) will be specified in such a way that ellipticity entails the Fredholm property of operators in adequate distribution spaces and that elliptic regularity of solutions includes asymptotics for $r\rightarrow0.$ There are many variants of cone algebras. We may have compact manifolds $M$ with conical singularities as well as infinite cones $\Xw$ with some control for $r\rightarrow\infty,$ and there are different notions of asymptotics. In this chapter we focus on discrete asymptotics. In the following definition we consider a compact manifold $M$ with conical singularity $v,$ for a mapping
$\chi_\reg:V\setminus\{v\}\rightarrow\R_+\times X$
as in (\ref{Chireg}), and choose cut-off functions $\omega,\omega',\omega''$ on $\R_+$ such that $\omega''\prec\omega\prec\omega'.$ Moreover, we form corresponding cut-off functions $\epsilon:=\chi_\reg^{*}\omega,$ $\epsilon':=\chi_\reg^{*}\omega',$ $\epsilon'':=\chi_\reg^{*}\omega''$ on $M$ close to $v.$\\

\nt Recall that from Section \ref{1.2.5} we have the space ${L}^{\mu}_{M+G}(\X,\g)$ of smoothing Mellin plus Green operators with discrete asymptotics, for weight data $\g=(\gamma,\gamma-\mu,\Theta),$ and, moreover, we introduced the space ${L}_{G}(M,\g)$ of global Green operators on $M$ with discrete asymptotics. Since the smoothing Mellin operators are supported by a neighbourhood of the conical singularity we also obtain the space ${L}_{M+G}^\mu(M,\g)$ of global smoothing Mellin plus Green operators.

\begin{defn}\label{cone_def}
Let $M$ be a manifold with $($for simplicity one$)$ conical singularity $v,$ and assume that $M$ is compact. Fix $\g=(\gamma,\gamma-\mu,\Theta)$ for a weight $\gamma\in\R,$ an order $\mu\in\R,$ and a weight interval $\Theta=(-(k+1),0],$ for some $k\in\N\cup\{+\infty\}.$ Then ${L}^{\mu}(M,\g)$\label{AAmMg} is defined to be the set of all $A\in L^\mu_\clas(M\setminus\{v\})$ of the form
$$A=A_\sing+A_\reg+G$$
where $A_\sing=(\chi_\reg^{-1})_*\{\omega r^{-\mu}\op_M^{\gamma-\frac n2}(h)\omega'+\omega R\omega'\}$ and  $A_\reg\in(1-\epsilon)L^\mu_\clas(M\setminus\{v\})(1-\epsilon'')$ for arbitrary $h(r,z)\in C^\infty(\Rr,M_\O^\mu(X))$ and $R\in{L}^\mu_{M+G}(\X,\g),$ and $G\in{L}_G(M,\g).$
\end{defn}

\nt For $A\in{L}^\mu(M,\g)$ we set
$$\sigma(A):=(\sigma_\psi(A),\sigma_\mathrm{c}(A))\label{symbs}$$
where $\sigma_\psi(A)$ is the standard homogeneous principal symbol of $A$ as an operator in $L^\mu_\clas(M\setminus\{v\}).$ Observe that close to $v$ in the variables $(r,x)\in\R_+\times\Sigma,$ $\Sigma\subseteq\R^n,$ corresponding to a chart on $X$ we have
$$\sigma_\psi(A)(r,x,\rho,\xi)=r^{-\mu}\til\sigma_\psi(A)(r,x,r\rho,\xi)$$
for some function $\til\sigma_\psi(A)(r,x,\til\rho,\xi)$\label{sig28} (also referred to as the reduced homogeneous principal symbol of $A)$ that is homogeneous in $(\til\rho,\xi)\ne0$ of order $\mu$ and smooth in $r$ up to $0.$ Moreover, we set
$$\sigma_\mathrm{c}(A)(z)=h(0,z)+f_0(z)$$
for $f_0(z)=\sigma_\mathrm{c}(R)(z),$ $z\in\Gamma_{\frac{n+1}2-\gamma},$ $n=\dim X,$ called the (principal) conormal symbol\index{conormal!symbol}\index{symbol!conormal $-$} of $A.$ It is $L^\mu_\clas(X)$-valued and as such defines a family of continuous operators
$$\sigma_\mathrm{c}(A)(z):H^s(X)\rightarrow H^{s-\mu}(X),$$
$s\in\R.$\\

\nt Similarly as in Section \ref{1.2.5} we may observe also the lower order conormal symbols of $A.$ In this case we write $\sigma_\mathrm{c}^\mu(A)(z)$ rather than $\sigma_\mathrm{c}(A)(z),$ and the lower order terms are defined as
$$\sigma_\mathrm{c}^{\mu-j}(A)(z)=\frac1{j!}(\partial_r^jh)(0,z)+\sigma_\mathrm{c}^{\mu-j}(R)(z)\label{consym2}$$
for $j=0,\dots,k.$\\

\nt Let us now give an analogous definition in the case of the infinite (stretched) cone $\X$ rather than $M.$

\begin{defn}\label{cone_def2}
Let $X$ be a closed compact $C^\infty$ manifold$,$ fix weight data $\g=(\gamma,\gamma-\mu,\Theta)$ as in \textup{Definition \ref{cone_def},} and choose cut-off functions $\omega''\prec\omega\prec\omega'$ on $\R_+.$ Then ${L}^\mu(\X,\g)$ is defined to be the set of all $A\in L^\mu_\clas(\X)$ of the form
$$A=A_\sing+A_\reg$$
where $A_\sing=\omega r^{-\mu}\op_M^{\gamma-\frac n2}(h)\omega'+R,$
for an $h(r,z)\in C^\infty(\Rr,M_\O^\mu(X)),$ $R\in{L}^\mu_{M+G}(\X,\g),$ and $A_\reg\in(1-\omega)L^{\mu;0}_\clas(X^\asymp)(1-\omega'').$
\end{defn}

\nt Here $L^{\mu;0}_\clas(X^\asymp)$ is the space of classical operators of order $\mu$ on the manifold $X^\asymp$ with conical exits to infinity and exit order zero, cf. \cite[Section 1.4]{Schu20}.\\

\nt We have again a principal symbolic structure of operators $A\in{L}^\mu(\X,\g),$ namely,
\begin{equation}\label{new22}
\sigma(A)=(\sigma_\psi(A),\sigma_{\mathrm c}(A),\sigma_{\mathrm E}(A))
\end{equation}
where $\sigma_\psi(A),$ and $\sigma_{\mathrm c}(A)$ are of analogous meaning as before, while $\sigma_{\mathrm E}(A)$ is the pair of exit symbols\index{exit symbol}\index{symbol!exit $-$}
$$\sigma_{\mathrm E}(A)=(\sigma_{\mathrm e}(A),\sigma_{\psi,\mathrm e}(A))\label{cs3}$$
belonging to $A$ for $r\rightarrow\infty.$ Concerning the definition of $\sigma_{\mathrm e}$ and $\sigma_{\psi,\mathrm e},$ cf. \cite[Section 1.4]{Schu20}.

\begin{thm}\label{cont20}
\begin{itemize}
\item[\textup{(i)}] Let $M$ be a compact manifold with conical singularity$,$ and $A\in{L}^\mu(M,\g),$ $\g=(\gamma,\gamma-\mu,\Theta).$ Then $A$ induces continuous operators
$$A:H^{s,\gamma}(M)\rightarrow H^{s-\mu,\gamma-\mu}(M),H^{s,\gamma}_\P(M)\rightarrow H^{s-\mu,\gamma-\mu}_\Q(M)$$
for every $s\in\R$ and every asymptotic type $\P$ with some resulting asymptotic type $\Q,$ associated with the weight data $(\gamma,\Theta)$ and $(\gamma-\mu,\Theta),$ respectively.
\item[\textup{(ii)}] Let $X$ be a closed compact $C^\infty$ manifold$,$ and $A\in{L}^\mu(\X,\g),$ with $\g$ as in \textup{(i)}. Then $A$ induces continuous operators
$$A:\K^{s,\gamma}(\X)\rightarrow\K^{s-\mu,\gamma-\mu}(\X),\K^{s,\gamma}_\P(M)\rightarrow \K^{s-\mu,\gamma-\mu}_\Q(M)$$
for every $s\in\R$ and $\P$ and $\Q$ as in \textup{(i)}.
\end{itemize}
\end{thm}

\begin{rem}
\begin{itemize}
\item[\textup{(i)}] Let $A$ be as in \textup{Theorem \ref{cont20} (i);} then $\sigma(A)=0$ $($i.e.$,$ $\sigma_\psi(A)=\sigma_\mathrm{c}(A)=0)$ implies that
$$A:H^{s,\gamma}(M)\rightarrow H^{s-\mu,\gamma-\mu}(M)$$
is a compact operator for every $s\in\R.$ 
\item[\textup{(ii)}] Let $A$ be as in \textup{Theorem \ref{cont20} (ii);} then $\sigma(A)=0$ $($i.e.$,$ $\sigma_\psi(A)=\sigma_\mathrm{c}(A)=\sigma_\mathrm{e}(A)=\sigma_{\psi,\mathrm{e}}(A)=0)$ implies that
$$A:\K^{s,\gamma}(\X)\rightarrow\K^{s-\mu,\gamma-\mu}(\X)$$
is a compact operator for every $s\in\R.$
\end{itemize}
\end{rem}

\nt In the following theorem we assume $M$ and $\X$ to be as in Theorem \ref{cont20}.

\begin{thm}
\begin{itemize}
\item[\textup{(i)}] Let $A\in{L}^\mu(M,\g)$ for $\g=(\gamma-\til\mu,\gamma-\til\mu-\mu,\Theta)$ and $B\in{L}^{\til\mu}(M,\til\g)$ for $\til\g=(\gamma,\gamma-\til\mu,\Theta);$ then we have $AB\in{L}^{\mu+\til\mu}(M,\g\circ\til\g)$ for $\g\circ\til\g=(\gamma,\gamma-\til\mu-\mu,\Theta),$\label{gcompg} and $\sigma(AB)=\sigma(A)\sigma(B)$ in the sense
$$\sigma_\psi(AB)=\sigma_\psi(A)\sigma_\psi(B),\quad\sigma_\mathrm{c}(AB)=(T^{\til\mu}\sigma_\mathrm{c}(A))\sigma_\mathrm{c}(B)$$
$(T^\beta f)(z):=f(z+\beta).$
\item[\textup{(ii)}] Let $A\in{L}^\mu(\X,\g)$ and $B\in{L}^{\til\mu}(\X,\til\g)$ with $\g,\til\g$ as in \textup{(i)}$;$ then we have $AB\in{L}^{\mu+\til\mu}(\X,\g\circ\til\g)$ and $\sigma(AB)=\sigma(A)\sigma(B)$ in the sense of the rules from \textup{(i)} together with
$$\sigma_\mathrm{e}(AB)=\sigma_\mathrm{e}(A)\sigma_\mathrm{e}(B),\quad\sigma_{\psi,\mathrm{e}}(AB)=\sigma_{\psi,\mathrm{e}}(A)\sigma_{\psi,\mathrm{e}}(B).$$
\end{itemize}
\end{thm}

\begin{rem}
Similarly as in \textup{Theorem \ref{comp25}} the sequence of conormal symbols $(\sigma_\mathrm{c}^{\mu+\til\mu-l}(AB))_{l=0,\dots,k}$ can be obtained as the Mellin translation product\index{Mellin!translation product}\index{product!Mellin translation $-$} between the sequences $(\sigma_{\mathrm c}^{\mu-j}(A))_{j=0,\dots,k}$ and $(\sigma_{\mathrm c}^{\til\mu-i}(B))_{i=0,\dots,k}.$
\end{rem}

%1.2.7
%
%
%
%
%
\subsection{Ellipticity in the cone algebra and asymptotics of solutions}

Based on the principal symbolic structure of operators in the cone algebra we now study ellipticity and asymptotics of solutions.

\begin{defn}\label{Ell_def}
\begin{itemize}
\item[\textup{(i)}] Let $M$ be a compact manifold with conical singularity $v,$ and let $A\in{L}^\mu(M,\g)$ as in \textup{Definition \ref{cone_def}}. Then $A$ is called elliptic\index{elliptic!operator}\index{operators!elliptic $-$} if $\sigma_\psi(A)\ne0$ on $T^*(M\setminus\{v\})\setminus0,$ and$,$ locally near $v,$ $\til\sigma_\psi(A)(r,x,\til\rho,\xi)\ne0$ for all $(\til\rho,\xi)\ne0,$ up to $r=0,$ moreover$,$
\begin{equation}%\label{co}
\sigma_{\mathrm c}(A)(z):H^s(X)\rightarrow H^{s-\mu}(X)
\end{equation}
is a family of isomorphisms for all $z\in\Gamma_{\frac{n+1}2-\gamma}$ and some $s=s_0\in\R.$
\item[\textup{(ii)}] Let $A\in{L}^\mu(\X,\g)$ be as in \textup{Definition \ref{cone_def2}}. Then $A$ is called elliptic if $\sigma_\psi(A)$ and $\sigma_{\mathrm c}(A)$ satisfy analogous condition as in \textup{(i)} and if the symbolic components of $\sigma_{\mathrm E}(A)$ are elliptic in the sense of the calculus of operators in $L^{\mu;\nu}_\clas(X^\asymp)$\label{Lclmn} at $r=\infty.$
\end{itemize}
\end{defn}

\begin{thm}\label{1.2.41}
\begin{itemize}
\item[\textup{(i)}] Let $A\in{L}^\mu(M,\g)$ be as in \textup{Definition \ref{cone_def}}. Then $A$ is elliptic in the sense of \textup{Definition \ref{Ell_def} (i)} if and only if
\begin{equation}\label{F_def}
A:H^{s,\gamma}(M)\rightarrow H^{s-\mu,\gamma-\mu}(M)
\end{equation}
is a Fredholm operator\index{Fredholm operator}\index{operators!Fredholm $-$} for some $s=s_0\in\R$ $($which is equivalent to the Fredholm property for all $s\in\R).$
\item[\textup{(ii)}] Let $A\in{L}^\mu(\X,\g)$ be as in \textup{Definition \ref{cone_def2}}. Then $A$ is elliptic in the sense of \textup{Definition \ref{Ell_def} (ii)} if and only if
\begin{equation}\label{F_def2}
A:\K^{s,\gamma;g}(\X)\rightarrow\K^{s-\mu,\gamma-\mu;g}(\X)
\end{equation}
is a Fredholm operator for some $s=s_0,g=g_0\in\R$ $($again equivalent to the Fredholm property for all $s\in\R).$
\end{itemize}
\end{thm}

\begin{thm}\label{T1.2.47}
\begin{itemize}
\item[\textup{(i)}] Let $A\in{L}^\mu(M,\g)$ be an operator as in \textup{Definition \ref{cone_def}}. Then the Fredholm property of $(\ref{F_def})$ for some $s=s_0,$ entails the Fredholm property of $(\ref{F_def})$ for all $s\in\R.$ Moreover$,$ $A$ has a parametrix\index{parametrix} $B\in{L}^{-\mu}(M,\g^{-1}),$ i.e.$,$ $I-BA\in{L}_G(M,(\gamma,\gamma,\Theta)),$ $I-AB\in{L}_G(M,(\gamma-\mu,\gamma-\mu,\Theta)).$ Finally$,$ if $(\ref{F_def})$ is invertible for an $s=s_0\in\R$ then so is for all $s\in\R,$ and we have $A^{-1}\in{L}^{-\mu}(M,\g^{-1})$ for $\g^{-1}:=(\gamma-\mu,\gamma,\Theta).$\label{g-1}
\item[\textup{(ii)}] Let $A\in{L}^\mu(\X,\g)$ be an operator as in \textup{Definition \ref{cone_def2}}. Then the Fredholm property of $(\ref{F_def2})$ for reals $s=s_0,$ $g=g_0,$ entails the Fredholm property of $(\ref{F_def2})$ for all $s,g\in\R.$ Moreover$,$ $A$ has a parametrix $B\in{L}^{-\mu}(\X,\g^{-1}),$ i.e.$,$ $I-BA\in{L}_G(\X,(\gamma,\gamma,\Theta)),$ $I-AB\in{L}_G(\X,(\gamma-\mu,\gamma-\mu,\Theta)).$ Finally$,$ if $(\ref{F_def2})$ is invertible for some $s=s_0,g=g_0\in\R$ then so is for all $s,g\in\R,$ and we have $A^{-1}\in{L}^{-\mu}(\X,\g^{-1}).$
\end{itemize}
\end{thm}

\nt In the following theorem we indicate by subscripts \virg{$(\Q)$}, etc., spaces with or without asymptotics of type $\Q.$ 

\begin{thm}\index{elliptic!regularity in the cone calculus}\index{elliptic!regularity in the cone calculus!with asymptotics}
\begin{itemize}
\item[\textup{(i)}] Let $A\in{L}^\mu(M,\g)$ be elliptic$,$ $u\in H^{-\infty,\gamma}(M),$ and $Au=f\in H^{s-\mu,\gamma-\mu}_{(\Q)}(M),$ $s\in\R$ $($for some discrete asymptotic type $\Q,$ associated with the weight data $(\gamma-\mu,\Theta)).$ Then $u\in H^{s,\gamma}_{(\P)}(M)$ $($for a resulting discrete asymptotic type $\P,$ associated with the weight data $(\gamma,\Theta)).$
\item[\textup{(ii)}] Let $A\in{L}^\mu(\X,\g)$ be elliptic$,$ $u\in\K^{-\infty,\gamma;-\infty}(\X),$ and $Au=f\in\K^{s-\mu,\gamma-\mu;g}_{(\Q)}(\X),$ $s,g\in\R,$ for some $\Q$ as in \textup{(i)}. Then $u\in\K^{s,\gamma;g}_{(\P)}(\X)$ for some $\P$ as in \textup{(i)}.
\end{itemize}
\end{thm}

%1.3
%
%
%
%
%
\section{Abstract edge operators}

{\footnotesize The calculus of edge pseudo-differential operators is formulated in terms of various quantisations of \virg{scalar} (edge-degenerate) symbols. One aspect of this approach is based on operator-valued amplitude functions with twisted homogeneity; those encode, in particular, various kinds of trace, potential and Green
operators. Moreover, there is the concept of \virg{abstract} edge spaces which are the model of the concrete weighted edge spaces where an anisotropic information from the wedge geometry is feeded in corresponding scales of spaces, modelled on a vector space with group action. Another important part is a vector/operator-valued version of Kumano-go's technique to oscillatory integrals and pseudo-differential operators globally in the Euclidean space, see \cite{Kuma1}.\par}

%1.3.1
%
%
%
%
%
\subsection{Abstract edge spaces}\label{1.3.2}

\begin{defn}
A Hilbert space $H$ $($throughout this consideration assumed to be
separable$)$ is said to be endowed with a group action\index{group action!Hilbert space with $-$} $\kappa=\Kll$
if
\begin{itemize}
    \item[$\mathrm{(i)}$] $\Kl:H\longrightarrow H,$
    $\lambda\in\R_+,$ is a family of isomorphisms$,$ where
    $\Kl\kappa_\nu=\kappa_{\lambda\nu}$ for all $\lambda,\nu\in\R_+,$
    \item[$\mathrm{(ii)}$] $\lambda\mapsto\Kl h,$ $\lambda\in\R_+,$
    defines a function in $C^\infty(\R_+,H)$ for every $h\in H$
    $($i$.$e$.,$ $\kappa$ is strongly continuous$).$
\end{itemize}
\end{defn}

\nt Let us give a few examples of Hilbert spaces with group action. By $H^s(\R^m),$\newnot{Hstand} $s\in\R,$ we denote the standard Sobolev space\index{Sobolev spaces!standard $-$}\index{spaces!Sobolev $-$!standard $-$} in $\R^m$ of smoothness $s.$ Recall that $H^s(\R^m)$ can be defined to be the set of all $u\in\S'(\R^m)$ such that $(Fu)(\xi)=\hat u(\xi)$ belongs to $\ang\xi^{-s}L^2(\R^m_\xi).$ Moreover, if $\Omega\subseteq\R^m$ is a domain with smooth boundary, we set
$$H^s(\Omega):=\set{u|_\Omega:u\in H^s(\R^m)}.$$

\begin{exm}\label{H}
\begin{itemize}
    \item[$\mathrm{(i)}$] $H:=L^2(\R_+)$ and $(\Kl u)(r):=\lambda^\frac12u(\lambda
    r),$ $\lambda\in\R_+;$
    \item[$\mathrm{(ii)}$] $H:=H^s(\R_+),$ $s\in\R,$ with $\Kl$ as
    in $\mathrm{(i)};$
    \item[$\mathrm{(iii)}$] $H:=H^s(\R^m),$ $s\in\R,$ and $(\Kl u)(x):=\lambda^{\frac m2}u(\lambda
    x),$ $\lambda\in\R_+.$
\end{itemize}
\end{exm}

\begin{defn}\label{1.3.3}
A Fr\'{e}chet space $E,$ written as a projective limit of Hilbert
spaces $E^j,$ $j\in\N,$ with continuous embeddings
$E^j\hookrightarrow E^0,$ $j\in\N,$ $($i.e.$,$ $E=\bigcap_{j\in\N} E^j,$ and the projective limit referring to the mappings $E\hookrightarrow E^j,$ $j\in\N)$ is said to be endowed
with a group action\index{group action!\Fr space with $-$} $\kappa=\Kll,$ if
\begin{itemize}
    \item[$\mathrm{(i)}$] $\Kl$ is a group action on $E^0,$
    \item[$\mathrm{(ii)}$] $\kappa|_{E^j}$ is a group action on $E^j$ for every $j\in\N.$
\end{itemize}
\end{defn}

\begin{exm}\label{F}
Let $E=\S(\Rr)$ be written as
$$E=\Projlim{j\in\N}\ang{r}^{-j}H^j(\R_+).$$
Then $\kappa$ from $\mathrm{Example}$ $\ref{H}$ $(\mathrm{i})$
defines a group action on $E.$
\end{exm}

\begin{defn}\label{1.3.5}		
Given a Hilbert space $H$ with group action $\kappa=\Kll$ the abstract edge space\index{edge!abstract $-$ spaces}\index{spaces!edge $-$!abstract $-$} $\WW^s(\R^q,H)$\newnot{symbol:Ws} of smoothness $s\in\R,$ modelled on $H$ and with edge $\R^q,$ is defined to be the completion of $\S(\R^q,H)$ with respect to the norm
\begin{equation}\label{N}
\norm{u}_{\WW^s(\R^q,H)}:=\set{\int\ang{\eta}^{2s}\norm{\kappa_{\ang\eta}^{-1}
\hat{u}(\eta)}^2_Hd\eta}^{1/2},
\end{equation}
where $\hat u(\eta)=(Fu)(\eta)$ is the Fourier transform\index{Fourier!transform}\index{transform!Fourier $-$} in $\R^q.$ If necessary$,$ to indicate the choice of the group action we write
$$\WW^s(\R^q,H)_\kappa.\label{Wskappa}$$
In particular$,$ we set $H^s(\R^q,H)=\WW^s(\R^q,H)_\id.$\newnot{HsRH}
\end{defn}

\nt If $H=\C$ and $\Kl=\id$ for all $\lambda$ we obtain the standard Sobolev space $H^s(\R^q)$ of smoothness $s.$

\begin{rem}
Recall that $\WW^s(\R^q,H)$ can equivalently be defined as the set of all $u\in\S'(\R^q,H)$\label{Sprime} $(=\L(\S(\R^q),H))$ such that $\kappa^{-1}_{\ang\eta}\hat u(\eta)\in\ang{\eta}^{-s}L^2(\R^q,H).$ A proof may be found in \em\cite{Hirs2}\em.
\end{rem}

\begin{rem}
Replacing $\norm{\cdot}_H$ in \textup{(\ref{N})} by an equivalent norm in $H$ and $\eta\mapsto\ang\eta$ by any strictly positive function $p(\eta)$ with $c\ang\eta\le p(\eta)\le C\ang\eta$ for all $\eta\in\R$ and certain $c,C>0,$ gives us an equivalent norm in the space $\WW^s(\R^q,H).$ In particular$,$ we can take $p(\eta)=[\eta]$  $($recall that $\eta\mapsto[\eta]$ is any strictly positive function in $C^\infty(\R^q)$ such that $[\eta]=|\eta|$ for $|\eta|\ge\textup{const}).$
\end{rem}

\nt If $\Omega\subseteq\R^q$ is an open set
\begin{equation}\label{comp}
\WW^s_\comp(\Omega,H)
\end{equation}
will denote the set of all $u\in\D'(\Omega,H)$ $(=\L(C_0^\infty(\Omega),H))$ such that $\varphi u\in\WW^s(\R^q,H)$ for every $\varphi\in C_0^\infty(\Omega).$ Moreover, 
\begin{equation}\label{loc}
\WW^s_\loc(\Omega,H)
\end{equation}
is defined to be the set of all $u\in\D'(\Omega,H)$ such that $\varphi u\in\WW^s_\comp(\Omega,H)$ for every $\varphi\in C_0^\infty(\Omega).$ As a consequence of the above definitions, we have
$$\WW^s_\comp(\Omega,H)\subset\WW^s_\loc(\Omega,H)$$
for every $s\in\R.$

\begin{rem}
The norm $(\ref{N})$ can also be written in the form 
$$\norm{u}_{\WW^s(\R^q,H)_\kappa}=\norm{\ang{\eta}^s F(F^{-1}\kappa_{\ang\eta}^{-1}F)u}_{L^2(\R^q,H)}=
\norm{F^{-1}\kappa_{\ang\eta}^{-1}Fu}_{H^s(\R^q,H)},$$
which shows that there is an isomorphism
\begin{equation}\label{Tiso}
K:=F^{-1}\kappa_{\ang\eta}F:H^s(\R^q,H)\rightarrow{\WW^s(\R^q,H)_\kappa}.
\end{equation}
\end{rem}

\nt In other words, for every $u(y)\in\WW^s(\R^q,H)_\kappa$ there exists a $v(y)\in H^s(\R^q,H)$ such that
\begin{equation}\label{Tt}
\hat u(\eta)=\kappa_{\ang\eta}\hat v(\eta).
\end{equation} 

\nt Later on in this context it will be convenient to replace $\ang\eta$ by $[\eta].$

\begin{exm}\label{3.9}
Let us consider $H^s(\R^n)$ with the group action $(\kappa_\lambda u)(x)=\lambda^{n/2}u(\lambda x),$ $\lambda\in\R_+.$ Then we obtain
\begin{equation}\label{re}
\WW^s(\R^q,H^s(\R^n))=H^s(\R^q\times\R^n)
\end{equation}
for every $s\in\R.$ Observe that $\Kll$ is unitary in $L^2(\R^n);$ thus for $s=0$ the relation $(\ref{re})$ just means $L^2(\R^q,L^2(\R^n))=L^2(\R^q\times\R^n).$
\end{exm}

\begin{rem}
Let $E^1,E^0$ be Hilbert spaces with a continuous embedding $E^1\hookrightarrow E^0,$ and let $\kappa$ be a group action on $E^0$ that restricts to a group action on $E^1.$ Then we have canonical continuous embeddings
$$\WW^{s'}(\R^q,E^1)\hookrightarrow\WW^s(\R^q,E^0)$$
for all $s,s'\in\R,$ $s'\ge s.$
\end{rem}

\begin{defn}\label{Fredge}
Let $E$ be a \Fr space with group action as in \textup{Definition \ref{1.3.3}.} We set
\begin{equation}\label{new411}
\WW^s(\R^q,E):=\Projlim{j\in\N}\WW^s(\R^q,E^j),\qquad s\in\R,
\end{equation}
taken in the \Fr topology of the projective limit.
\end{defn}

\nt Similarly as (\ref{comp}) and (\ref{loc}), we have the spaces
$$\WW^s_\comp(\Omega,E)\quad\mathrm{and}\quad\WW^s_\loc(\Omega,E)$$
in the case of a \Fr space $E=\Projlim{j\in\N}E^j$ with group action, namely,
$$\WW^s_\comp(\Omega,E)=\Projlim{j\in\N}\WW^s_\comp(\Omega,E^j),$$
and analogously for the \virg{loc} version.\\

\begin{prop}\label{1.3.10}
Let $E$ be a \Fr space with group action. Then we have
$$\WW^\infty(\R^q,E)=H^\infty(\R^q,E),$$
i.e.$,$ the space $\WW^\infty(\R^q,E)$ is independent of the choice of the group action $\Kll$ in $E.$
\end{prop}

\begin{proof}
Let us first note that if $E$ is a \Fr space with group action, $E=\Projlim{j\in\N}E^j,$ for every $j$ there are constants $c_j,M_j>0$ such that
$$
\norm{\Kl}_{\L(E^j)}\le c_j\left\{
\begin{array}{ll}
\lambda^{M_j} & \mathrm{for\;}\lambda\ge1\\[2mm]
\lambda^{-M_j} & \mathrm{for\;}\lambda\le1
\end{array}\right..
$$
We then have, for every $s\in\R,$ continuous embeddings
\begin{equation}\label{new30}
H^{s+M_j}(\R^q,E^j)\hookrightarrow\WW^s(\R^q,E^j)\hookrightarrow H^{s-M_j}(\R^q,E^j),
\end{equation}
i.e., we obtain $\WW^\infty(\R^q,E^j)=H^\infty(\R^q,E^j)$ for every $j,$ and hence
$$\WW^\infty(\R^q,E)=\Projlim{j\in\N}\WW^\infty(\R^q,E^j)=
\Projlim{j\in\N}H^\infty(\R^q,E^j)=H^\infty(\R^q,E).$$
\end{proof}

\begin{rem}\label{emb}
For $s'\ge s$ we have a canonical continuous embedding
$$\WW^{s'}(\R^q,E)\hookrightarrow\WW^s(\R^q,E).$$
\end{rem}

\begin{prop}\label{2.2.28}
Let $E=\Projlim{j\in\N}E^j$ be a \Fr space with group action. Then for every $s\in\N$ we have
$$\WW^s(\R^q,E)=\{u\in\WW^0(\R^q,E):D^\alpha_yu\in\WW^0(\R^q,E), \abs\alpha\le s\}.$$
\end{prop}

\begin{proof}
There are constants $c_1,c_2>0$ such that
$$c_1(\sum_{\abs\alpha\le s}\abs{\eta^\alpha}^2)\le\ang{\eta}^{2s}\le c_2(\sum_{\abs\alpha\le s}\abs{\eta^\alpha}^2).$$
Therefore, for all $\abs\alpha\le s,$
$$\int||\kappa^{-1}_{\ang\eta}\eta^\alpha\hat u(\eta)||^2_{E^j}d\eta<\infty\quad\Rightarrow
\quad\int\ang\eta^{2s}||\kappa^{-1}_{\ang\eta}\hat u(\eta)||^2_{E^j}d\eta<\infty,$$
and vice versa$,$ and this holds for all $j.$
\end{proof}

\begin{prop}
Let $E=\Projlim{j\in\N}E^j$ be a \Fr space with group action$,$ and set
$$L^2(\R^q,E):=\Projlim{j\in\N}L^2(\R^q,E^j).$$
Then we have
\begin{equation}\label{2.2.27}
\WW^\infty(\R^q,E)=\{u\in L^2(\R^q,E):D^\alpha_yu\in L^2(\R^q,E),\mathrm{\;for\,all\;}\alpha\in\N^q\}.
\end{equation}
\end{prop}

\begin{proof}
For every fixed $j$ we have
$$H^s(\R^q,E^j)=\{u\in\S'(\R^q,E^j):\int\ang\eta^{2s}||\hat u(\eta)||^2_{E^j}d\eta<\infty\}.$$
From \eqref{new30} it follows that $\WW^\infty(\R^q,E^j)=\bigcap_{s\in\R}\WW^s(\R^q,E^j)=\bigcap_{s\in\N}H^s(\R^q,E^j).$ By virtue of
$$H^s(\R^q,E^j)=\{u\in L^2(\R^q,E^j):D^\alpha_yu\in L^2(\R^q,E^j),\mathrm{\;for\,all\;}\alpha\in\N^q,\abs\alpha\le s\}$$
for every $s\in\N$ (see Proposition \ref{2.2.28}) we then obtain the relation (\ref{2.2.27}) for the space $E^j$ instead of $E.$ Since this is true for all $j$ it follows (\ref{2.2.27}) for $E$ itself.
\end{proof}

\begin{cor}\label{1.3.15}
Let $E=\varprojlim_{j\in\N}E^j$ be a \Fr space with group action$.$
Then we have
\begin{equation}\label{UE}
\WW^\infty_\comp(\Omega,E)=C_0^\infty(\Omega,E),\quad\WW^\infty_\loc(\Omega,E)=C^\infty(\Omega,E).
\end{equation}
\end{cor}

\begin{proof}
First, Proposition \ref{1.3.10} gives us
$$\WW^\infty(\R^q,E)=H^\infty(\R^q,E)$$
which is characterised as the space 
$$\set{u\in L^2(\R^q,E):D^\alpha_yu\in L^2(\R^q,E)\mathrm{\;for\,all\;}\alpha\in\N^q}.$$
For $u\in\WW^\infty_\loc(\Omega,E)$ and every fixed $\varphi\in C_0^\infty(\Omega)$ it follows that $D^\alpha_y(\varphi u)$ belongs to $L^2(\R^q,E)$ and has compact support for all $\alpha\in\N^q.$ This just characterises the space $C^\infty(\Omega,E).$ The first relation of (\ref{UE}) is then an immediate consequence.
\end{proof}

\nt Later on for $H$ we will insert weighted cone spaces $\KsgX$ on the infinite stretched cone $\X=\R_+\times X$ with the group action $(\kappa_\lambda u)(r,x)=\lambda^\frac{n+1}2u(\lambda r,x),$ $\lambda\in\R_+,$ for $n=\dim X.$ Then we obtain so-called weighted edge spaces\index{weighted!spaces!edge $-$}\index{spaces!edge $-$!weighted $-$}
$$\WW^s(\R^q,\KsgX).\label{WsKsg}$$
More generally, it may be interesting to admit weights for $r\rightarrow\infty,$ i.e., to take the spaces $\K^{s,\gamma;g}(\X)=\ang{r}^{-g}\KsgX$ (cf. Definition \ref{ksgg}) for any $g\in\R,$ with the group action \eqref{Klhom} and associated edge spaces
\begin{equation}\label{new311}
\WW^s(\R^q,\K^{s,\gamma;g}(\X))_{\kappa^g}.
\end{equation}

\nt In the case $g=0$ we simply write $\kappa$ and $\WW^s(\R^q,\K^{s,\gamma}(\X)).$

\begin{thm}
For every $s,\gamma,g\in\R$ we have 
$$H^s_\comp(\X\times\R^q)\subset\WW^s(\R^q,\K^{s,\gamma;g}(\X))_{\kappa^g}
\subset H^s_\loc(\X\times\R^q).$$
\end{thm}

\nt The edge pseudo-differential operators below will be independent of the specific choice of the group action when we refer to operators in spaces of that kind. A careful analysis of the functional analytic properties of weighted edge spaces (\ref{new311}) shows that $g:=s-\gamma$ may be a natural choice. In this case we set
\begin{equation}\label{new111}
K^{s,\gamma}(\X):=\K^{s,\gamma;s-\gamma}(\X)
\end{equation}
and
\begin{equation}\label{new211}
W^{s,\gamma}(\X\times\R^q):=\WW^s(\R^q,K^{s,\gamma}(\X))_{\kappa^{s-\gamma}}.
\end{equation}

\nt However, in the majority of our considerations we refer to the case $g=0$ which is natural, too (e.g., when standard Sobolev spaces are anisotropically decomposed, cf. Example \ref{3.9}).

\begin{thm}\label{sum}
Let $E$ be a \Fr space with group action $\kappa,$ and let $E_0,E_1$ be \Fr subspaces$,$ continuously embedded in $E,$ such that $E=E_0+E_1$ in the sense of the non-direct sum.
\begin{itemize}
\item[\textup{(i)}] If $\kappa$ restricts to group actions on $E_i,$ $i=0,1,$ then $\WW^s(\R^q,E_i)$ is continuously embedded in $\WW^s(\R^q,E),$ and we have
$$\WW^s(\R^q,E)=\WW^s(\R^q,E_0)+\WW^s(\R^q,E_1),$$
also as a non-direct sum$;$ here $\WW^s(\R^q,E_i)=KH^s(\R^q,E_i),$ see \textup{(\ref{Tiso})}.
\item[\textup{(ii)}] In the case that $E_0$ or $E_1$ are not invariant under $\kappa,$ we then have
$$\WW^s(\R^q,E)=KH^s(\R^q,E_0)+KH^s(\R^q,E_1),$$
again as a non-direct sum.
\end{itemize}
\end{thm}

%1.3.2
%
%
%
%
%
\subsection{Calculus with operator-valued symbols}\label{1.3.2a}

Parallel to the concept of abstract edge spaces modelled on Hilbert/\Fr spaces with group action, there are spaces of symbols with \textit{twisted homogeneity}\index{twisted!homogeneity}. Let $H$ and $\til H$ be Hilbert spaces with group action $\kappa$ and $\til\kappa,$ respectively. A function $f(y,\eta)\in C^\infty\big(U\times(\R^q\setminus\{0\}),\L(H,\til H)\big),$ $U\subseteq\R^p$ open, is said to be (twisted) homogeneous of order $\mu\in\R$ in the variable $\eta\in\R^q\setminus\{0\},$ if
\begin{equation}\label{hom35}
f(y,\lambda\eta)=\lambda^\mu\til\kappa_\lambda f(y,\eta)\kappa_\lambda^{-1}
\end{equation}
for all $\lambda\in\R_+$ and all $(y,\eta)\in U\times(\R^q\setminus\{0\}).$ Let $\chi(\eta)$ be an excision function; then 
$$a(y,\eta):=\chi(\eta)f(y,\eta)$$
belongs to $C^\infty\big(U\times\R^q,\L(H,\til H)\big)$ and satisfies the symbolic estimates\index{symbolic estimates}
\begin{equation}\label{symb}
\sup_{(y,\eta)\in K\times\R^q}\ang{\eta}^{-\mu+|\beta|}\norm{\til\kappa_{\ang\eta}^{-1}\{D_y^\alpha D_\eta^\beta a(y,\eta)\}\kappa_{\ang\eta}}_{\L(H,\til H)}<\infty
\end{equation}
for all $K\Subset U,$ $\alpha\in\N^p,$ $\beta\in\N^q.$

\begin{defn}\label{1.3.14}
Let $H$ and $\til H$ be Hilbert spaces with group action $\kappa$ and $\til\kappa,$ respectively. Then the space of $($operator-valued$)$ symbols\index{symbols!operator-valued $-$}\index{operator-valued!symbols}\index{symbol!space!(classical) operator-valued $-$}
\begin{equation}\label{symbsp}
S^\mu(U\times\R^q;H,\til H)
\end{equation}
for open $U\subseteq\R^p,$ $\mu\in\R,$ is defined to be the set of all $a(y,\eta)\in C^\infty\big(U\times\R^q,\L(H,\til H)\big)$ such that the symbolic estimates \textup{(\ref{symb})} hold for all $K\Subset U,$ $\alpha\in\N^p,$ $\beta\in\N^q.$ Let 
\begin{equation}\label{new112}
S^\mu_\clas( U\times\R^q;H,\til H)
\end{equation}
denote the subspace of all classical $($operator-valued$)$ symbols\index{classical!symbols!operator-valued $-$}\index{symbols!operator-valued $-$!classical $-$}\index{symbol!space!(classical) operator-valued $-$} $a(y,\eta)$ belonging to $S^\mu(U\times\R^q;H,\til H)$ such that there are $($twisted$)$ homogeneous components\index{twisted!homogeneous $-$ components}\index{homogeneous!components!twisted $-$} $a_{(\mu-j)}(y,\eta)$ of order $\mu-j,$ $j\in\N,$ with
$$a(y,\eta)-\chi(\eta)\sum_{j=0}^Na_{(\mu-j)}(y,\eta)\in S^{\mu-(N+1)}(U\times\R^q;H,\til H)$$
for every $N\in\N.$ We write $S^\mu_{(\clas)}(U\times\R^q;H,\til H)$ if we mean \textup{(\ref{symbsp}) or (\ref{new112})}. In particular$,$ $S^\mu_{(\clas)}(\R^q;H,\til H)$ will denote the subspace of respective elements that are independent of $y.$ If necessary$,$ in order to indicate the dependence on $\kappa,\til\kappa$ we write
\begin{equation}\label{new212}
S^\mu_{(\clas)}(U\times\R^q;H,\til H)_{\kappa,\til\kappa}.
\end{equation}
\end{defn}

\begin{rem}
The space \textup{(\ref{symbsp})} is \Fr with the semi-norm system given by the left hand side of \textup{(\ref{symb}),} where $K$ varies over compact sets of $U$ $($countably many exhausting $U)$ and $\alpha\in\N^p,$ $\beta\in\N^q.$
\end{rem}

\begin{rem}\label{1.3.16}
In our calculus we will construct many concrete examples of operator-valued symbols. For future references we observe that when a function $a(y,\eta)\in C^\infty\big(U\times\R^q,\L(H,\til H)\big)$ has the property
\begin{equation}\label{new18}
a(y,\lambda\eta)=\lambda^\mu\til\kappa_\lambda a(y,\eta)\Kl^{-1}
\end{equation}
for all $\lambda\ge1,$ $|\eta|\ge C,$ for some $C>0,$ then
$$a(y,\eta)\in S^\mu_\clas(U\times\R^q;H,\til H).$$
\end{rem}

\begin{exm}\label{Me28}
Set $H:=\KsgX,$ $\til H:=\K^{s-\mu,\gamma-\mu}(\X),$ and let $h(y)\in C^\infty(U,M_\O^\mu(X)).$ Moreover$,$ consider the family of operators
$$a(y,\eta):=r^{-\mu+j}\omega(r[\eta])\op_M^{\gamma-\frac n2}(h)(y)\eta^\alpha\omega'(r[\eta])$$
for $j\in\N,$ $\alpha\in\N^q,$ $|\alpha|\le j,$ and for cut-off functions $\omega,\omega'.$ Then we have
$$a(y,\eta)\in C^\infty(U\times\R^q,\L(\KsgX,\K^{s-\mu,\gamma-\mu}(\X)))$$
for every $s\in\R,$ and $a(y,\lambda\eta)=\lambda^{\mu-(j-|\alpha|)}\Kl a(y,\eta)\Kl^{-1}$ for all $\lambda\ge1,$ $|\eta|\ge C,$ for some $C>0.$ This implies
$$a(y,\eta)\in S_\clas^{\mu-(j-|\alpha|)}(U\times\R^q;\KsgX,\K^{s-\mu,\gamma-\mu}(\X)).$$
\end{exm}

\begin{prop}\label{cl21}
For every $a(\eta)\in S^\mu(\R^q;H,\til H)$ and any fixed $\xi\in\R^q$ we have $a_\xi(\eta):=a(\eta+\xi)\in S^\mu(\R^q;H,\til H)$ and $a(\eta)-a_\xi(\eta)\in S^{\mu-1}(\R^q;H,\til H).$ In particular$,$ if $a(\eta)$ is classical$,$ then
$$a_{(\mu)}(\eta)=a_{\xi,(\mu)}(\eta)$$
for every $\xi.$
\end{prop}

\begin{rem}
There is an analogue of \textup{Definition \ref{Lmu2}} for \Fr spaces $E,\til E$ with group action rather than Hilbert spaces $H,\til H.$ In the case
$$E=\Projlim{j\in\N}E^j,\quad\til E=\Projlim{k\in\N}\til E^k$$
with group actions $\kappa$ and $\til\kappa$ in $E^0$ and $\til E^0,$ respectively$,$ restricting to group actions in $E^j$ and $\til E^k$ $($cf. \textup{Definition} $\ref{1.3.3})$ we have the spaces $S^\mu_{(\clas)}(\Omega\times\Omega\times\R^q,E^j,\til E^k)$ for every $j,k.$ For notational convenience we assume that there are continuous embeddings $E^{j+1}\hookrightarrow E^j,$ $\til E^{k+1}\hookrightarrow\til E^k$ for all $j,k\in\N.$ Then an element $a(y,y',\eta)\in C^\infty(\Omega\times\Omega\times\R^q,\L(E^0,\til E^0))$ belongs to
$$S^\mu_{(\clas)}(\Omega\times\Omega\times\R^q,E,\til E)$$
if for every $k\in\N$ there is a $j=j(k)\in\N$ such that $a(y,y',\eta)$ also belongs to $S^\mu_{(\clas)}(\Omega\times\Omega\times\R^q,E^{j(k)},\til E^k)$ for every $k\in\N.$
\end{rem}

\nt Let us give an example in the sense of \textup{Remark \ref{1.3.16}} now for \Fr spaces rather than Hilbert spaces.

\begin{exm}\label{b29}
Let $\varphi(r)\in C_0^\infty(\R_+),$ and let $\M_\varphi$\label{Mphi} be the operator of multiplication by $\varphi.$ Moreover$,$ set $\varphi_\eta(r):=\varphi(r[\eta]).$ Then $b(\eta):\eta\mapsto\M_{\varphi_\eta}$ defines a function
$$b(\eta)\in C^\infty(\R^q,\L(\K^{\infty,\gamma}(\X),\K^{\infty,\infty}(\X))),$$
cf. the formula \textup{(\ref{new10a}),} and we have
$$b(\lambda\eta)=\Kl b(\eta)\Kl^{-1}$$
for all $\lambda\ge1,$ $|\eta|\ge C$ for some $C>0.$ Thus
$$b(\eta)\in S^0_\clas(\R^q;\K^{\infty,\gamma}(\X),\K^{\infty,\infty}(\X)).$$
\end{exm}

\nt In the case $U=\Omega\times\Omega,$ $\Omega\subseteq\R^q$ open, we write $(y,y')$ instead of $y.$ Similarly as in the scalar case (see Section \ref{Sec1.2.1}) we form pseudo-differential operators with operator-valued symbols\index{pseudo-differential operators!with operator-valued symbols}\index{operator-valued!symbols!pseudo-differential operators with $-$}\index{symbols!operator-valued $-$!pseudo-differential operators with $-$}
\begin{equation}\newnot{new39}
\Op_y(a)u(y):=\dint e^{i(y-y')\eta}a(y,y',\eta)u(y')dy'\dslash\eta,
\end{equation}
$\dslash\eta=(2\pi)^{-q}d\eta,$ first for $u\in C_0^\infty(\Omega,H),$ and then extended to other function and distribution spaces. If the variable in the pseudo-differential action is clear, we simply write $\Op(a)$ rather than $\Op_y(a).$ The notation double, left and right symbol\index{symbol!left, right, double $-$}\index{left symbol}\index{right symbol}\index{double symbol} will be employed in an analogous meaning as in the scalar case. Let us set
\begin{equation}\newnot{Lmu2}
L^\mu_{(\clas)}(\Omega;H,\til H):=\set{\Op(a):a(y,y',\eta)\in S^\mu_{(\clas)}(\Omega\times\Omega\times\R^q;H,\til H)}.
\end{equation}

\begin{thm}\label{cont}
An $A\in L^\mu(\Omega;H,\til H)$ $($with $H,\til H$ being \Fr spaces with group action$)$ induces continuous operators
$$A:\WW^s_\comp(\Omega,H)\rightarrow\WW^{s-\mu}_\loc(\Omega,\til H)$$
for every $s\in\R.$
\end{thm}

\begin{thm}\label{1.3.22}
Let $H$ be a Hilbert space with group action $\kappa$ and let $\varphi\in\S(\R^q).$
\begin{itemize}
\item[\textup{(i)}] The operator of multiplication by $\varphi$ represents a continuous operator
$$\M_\varphi:\WW^s(\R^q,H)\rightarrow\WW^s(\R^q,H)$$
and $\varphi\mapsto\M_\varphi$ defines a continuous operator
$$\S(\R^q)\rightarrow\L(\WW^s(\R^q,H))$$
for every $s\in\R.$
\item[\textup{(ii)}] The operator $\M_\varphi$ gives rise to an element $\varphi\cdot\id_H\in S^0(\R^q\times\R^q;H,H),$ the mapping $\varphi\mapsto\varphi\cdot\id_H=:a(y,\eta)$ defines a continuous operator
$$\S(\R^q)\rightarrow S^0(\R^q\times\R^q;H,H),$$
and we have $\Op(a)=\M_\varphi.$
\end{itemize}
\end{thm}

\begin{rem}%\label{E21}
Let $E,\til E$ be Hilbert spaces$,$ and $a\in\L(E,\til E)$ a given element. Then$,$ applying the operator $a$ pointwise to an element $u(y)\in H^s(\R^q,E),$ $s\in\R,$ and denoting the resulting function by $\Set{\M_au}(y),$ we have
$$\Set{\M_au}(y)\in H^s(\R^q,\til E),$$
$\M_a=\Op_y(a),$ where $a$ is interpreted as an element of $S_\clas^0(\R^q;E,\til E)$ where the involved group actions are $\id_E$ and $\id_{\til E},$ respectively.
\end{rem}

\nt In fact, the mapping $e\mapsto ae,$ $e\in E,$ can be interpreted as an operator-valued symbol in $S^0(\R^q;E,\til E)$ independent of $\eta.$ As a special case of the symbolic estimates (\ref{symb}), the norm $\norm{a}_{\L(E,\til E)}=\sup\frac{\norm{ae}_{\til E}}{\norm{e}_E}$ is finite. Thus
$$\begin{array}{rll}
\displaystyle\norm{\M_au}^2_{H^s(\R^q,\til E)}&\hspace{-2mm}=&\displaystyle\int\ang{\eta}^{2s}\norm{\Set{F\!\M_au}(\eta)}^2_{\til E}d\eta\;\;=\;\int\ang{\eta}^{2s}\norm{a\Set{Fu}(\eta)}^2_{\til E}d\eta\\[4mm]
&\hspace{-2mm}\le&\norm{a}^2\displaystyle\int\ang{\eta}^{2s}\norm{Fu(\eta)}^2_Ed\eta\;\;=\;\;\norm{a}^2\norm{u}^2_{H^s(\R^q,E)}.
\end{array}$$

\nt For future references we now give more examples of operator-valued symbols. In the following $\Omega\subseteq\R^q$ will be an open set.

\begin{exm}\label{exsy}
Let $\psi(y)\in C_0^\infty(\Omega),$ $\omega(r)$ a cut-off function$,$ and $d\in\R,$ $\alpha\in\N^q;$ then
$$a(y,\eta):=\psi(y)\omega(r[\eta])r^d\eta^\alpha$$
defines a symbol
$$a(y,\eta)\in S^{|\alpha|-d}_\clas(\Omega\times\R^q;\K^{s,\gamma;g}(\X),\K^{s,\gamma+d;g}(\X))$$
for every $s,\gamma,g\in\R.$
\end{exm}

\nt Example \ref{exsy} is a special case of Remark \ref{1.3.16}, since
$$a(y,\eta)\in C^\infty(\Omega\times\R^q,\L(\K^{s,\gamma;g}(\X),\K^{s,\gamma+d;g}(\X)))$$
and the relation (\ref{new18}) holds for $\mu=|\alpha|-d.$

\begin{exm}\label{exsy2}
The operator of multiplication by $\omega(r)-\omega(r[\eta])$ for a cut-off function $\omega$ defines an element of $S^0(\R^q;\K^{s,\gamma;g}(\X),\K^{s,\infty;g}(\X))$ for every $s,\gamma,g\in\R.$
\end{exm}

\begin{prop}\label{diff1}
Let $E=\Projlim{j\in\N}E^j$ be as in \textup{Corollary \ref{1.3.15}.} The operator $D_y^\alpha,$ $\alpha\in\N^q,$ induces continuous operators

\begin{equation}\label{DW1}
D^\alpha_y:\WW^s(\R^q,E)\rightarrow\WW^{s-|\alpha|}(\R^q,E)
\end{equation}
and
\begin{equation}\label{DW2}
D^\alpha_y:\WW^s_{\textup{comp/loc}}(\Omega,E)\rightarrow\WW^{s-|\alpha|}_{\textup{comp/loc}}(\Omega,E)
\end{equation}
for every $s\in\R,$ $\Omega\subseteq\R^q$ open.
\end{prop}

\begin{proof}
To show (\ref{DW1}) we write $D^\alpha_y=F^{-1}\eta^\alpha F$ and obtain for every $j\in\N$
\begin{eqnarray*}
\norm{D^\alpha_yu}^2_{\WW^{s-|\alpha|}(\R^q,E^j)} & = &\displaystyle\int\ang{\eta}^{2(s-|\alpha|)}\norm{\kappa_{\ang{\eta}}^{-1}(D^\alpha_yu)^\wedge(\eta)}^2_{E^j}d\eta\\
& = & \displaystyle\int\ang{\eta}^{2s}\ang{\eta}^{-2|\alpha|}|\eta^\alpha|^2 \norm{\kappa_{\ang{\eta}}^{-1}\hat{u}(\eta)}^2_{E^j}d\eta\\
& \le & \norm{u}^2_{\WW^s(\R^q,E^j)}.
\end{eqnarray*}
For (\ref{DW2}) we choose any $j\in\N,$ $\varphi\in C_0^\infty(\Omega),$ and show that there are elements $\varphi_\beta\in C_0^\infty(\Omega),$ such that
$$\norm{\varphi(D^\alpha_yu)}_{\WW^{s-|\alpha|}(\R^q,E^j)}\le\sum_{\beta\le\alpha}\norm{\varphi_\beta u}_{\WW^s(\R^q,E^j)}.$$

\nt In fact, by the Leibniz rule we have
$$D^\alpha_y(\varphi u)=\sum_{\beta\le\alpha}\psi_\beta D^\beta_yu$$
for suitable $\psi_\beta\in C_0^\infty(\Omega),$ where $\psi_\alpha=\varphi,$ i.e.,
$$\varphi(D^\alpha_yu)=D^\alpha_y(\varphi u)-\sum_{\beta<\alpha}\psi_\beta D^\beta_yu.$$

\nt By repeatedly forming commutators between differentiations and multiplications by functions in $C_0^\infty(\Omega)$ we obtain an expression of the form
$$\varphi(D^\alpha_yu)=\sum_{\beta\le\alpha}D^\beta_y(\varphi_\beta u)$$
for suitable $\varphi_\beta\in C_0^\infty(\Omega).$ Then, using Theorem \ref{1.3.22} together with the first part of the proof, it follows that
\begin{eqnarray*}
\norm{\varphi(D^\alpha_yu)}_{\WW^{s-|\alpha|}(\R^q,E^j)} & \le & \Big\Vert\sum_{\beta\le\alpha}D^\beta_y(\varphi_\beta u)\Big\Vert_{\WW^{s-|\alpha|}(\R^q,E^j)}\\
& \le & \sum_{\beta\le\alpha}\norm{\varphi_\beta u}_{\WW^{s-|\alpha|+|\beta|}(\R^q,E^j)}\\
& \le & C\sum_{\beta\le\alpha}\norm{\varphi_\beta u}_{\WW^s(\R^q,E^j)}.
\end{eqnarray*}

\nt In the latter step we employed Remark \ref{emb}.
\end{proof} 

\nt The elements of the pseudo-differential calculus are particularly transparent when we refer to operators with symbols with uniform symbolic estimates, according to a generalisation of Kumano-go's technique to the case of operator-valued symbols.\\

\nt Let us first recall some variants of vector-valued amplitude functions\index{vector-valued amplitude functions}\index{amplitude functions!vector-valued $-$}, with values in a \Fr space $V,$ that are involved in oscillatory integral constructions.

\begin{defn}\label{sy23}
Let $V$ be a \Fr space with a fixed countable system of semi-norms $(\pi_j)_{j\in\N}$ that defines the topology of $V.$ Moreover$,$ let $\bs{\mu}:=(\mu_j)_{j\in\N}$ and $\bs{\nu}:=(\nu_j)_{j\in\N}$ be sequences of reals. Then the space
\begin{equation}\label{symb23}
S^{\bs{\mu};\bs{\nu}}(\R^{2q};V)
\end{equation}
is defined to be the set of all $a(x,\xi)\in C^\infty(\R^{2q},V)$ such that
\begin{equation}\label{sm23}
\sup_{(x,\xi)\in\R^{2q}}\ang\xi^{-\mu_j}\ang x^{-\nu_j}\pi_j(D^\alpha_xD^\beta_\xi a(x,\xi))
\end{equation}
is finite for all $\alpha,\beta\in\N^q$ and all $j\in\N.$ Observe that the space $S^{\bs{\mu};\bs{\nu}}(\R^{2q};V)$ is \Fr with the system of semi-norms \textup{(\ref{sm23})}. Moreover$,$ let
$$S^{\bs{\infty};\bs{\infty}}(\R^{2q};V)=\bigcup_{\bs{\mu},\bs{\nu}}S^{\bs{\mu};\bs{\nu}}(\R^{2q};V)$$
$($the union is taken over all $\bs{\mu},\bs{\nu}).$
\end{defn}

\nt The spaces in Definition \ref{sy23} for $V=\C$ are employed in \cite{Kuma1}; the variant with values in a \Fr space is studied in \cite{Seil3}.\\

\nt Given a function $\chi(x,\xi)\in\S(\R^{2q})$ with $\chi(0,0)=1$ we form the oscillatory integral\index{oscillatory integral}\index{integral!oscillatory $-$}
\begin{equation}\label{osc23}
\textup{Os}[a]=\dint e^{-ix\xi}a(x,\xi)dx\dslash\xi:=\lim_{\varepsilon\rightarrow0}\dint e^{-ix\xi}\chi(\varepsilon x,\varepsilon\xi)a(x,\xi)dx\dslash\xi.
\end{equation}
Recall that the right hand side of (\ref{osc23}) represents a regularisation of a corresponding divergent double integral and that this process is justified by certain integrations by parts that lead to convergent integrals.\\

\nt The same process works also for vector-valued amplitude functions $a(x,\xi)\in S^{{\bs{\mu};\bs{\nu}}}(\R^{2q};V).$ As a result we obtain the following theorem.
\begin{thm}\label{Pos23}
The oscillatory integral construction $(\ref{osc23})$ defines a continuous mapping
$$\textup{Os}[\cdot]:S^{\bs{\mu};\bs{\nu}}(\R^{2q};V)\rightarrow V.$$
\end{thm}

\nt Let $H$ and $\til H$ be Hilbert spaces with group actions $\kappa$ and $\til\kappa,$ respectively (the following considerations easily extend to the case of \Fr spaces $E$ and $\til E$ with group actions).

\begin{defn}\label{symbb}
The space $S^\mu(\R^d\times\R^q;H,\til H)_\mathrm{b}$ is defined to be the set of all $a(y,\eta)\in C^\infty(\R^d\times\R^q,\L(H,\til H))$ such that
$$\sup_{(y,\eta)\in\R^d\times\R^q}\ang\eta^{-\mu+|\beta|}\norm{\til\kappa^{-1}_{\ang\eta}\{D_y^\alpha D_\eta^\beta a(y,\eta)\}\kappa_{\ang\eta}}_{\L(H,\til H)}<\infty$$
for all $\alpha\in\N^d,$ $\beta\in\N^q.$ Moreover$,$ let
$$S^\mu_{\clas}(\R^d\times\R^q;H,\til H)_\mathrm{b}$$
denote the subspace of all $a(y,\eta)\in S^\mu(\R^d\times\R^q;H,\til H)_\mathrm{b}$ such that there are elements $a_{(\mu-j)}(y,\eta)$ which are twisted homogeneous of order $\mu-j,$ $j\in\N$ $($more precisely$,$ $a_{(\mu-j)}(y,\eta)\in C^\infty(\R^d\times(\R^q\setminus\{0\}),\L(H,\til H))_\mathrm{b}$ with subscript \textup{`b'} referring to the $y$-variable$,$ and twisted homogeneous in the sense of \textup{(\ref{hom35}))} with
$$a(y,\eta)-\chi(\eta)\sum_{j=0}^Na_{(\mu-j)}(y,\eta)\in S^{\mu-(N+1)}(\R^d\times\R^q;H,\til H)_\mathrm{b}$$
for every $N\in\N.$
\end{defn}

\nt As usual, we employ subscripts `(cl)' when a consideration is valid in the classical as well in the general case. Moreover, if necessary we write $S^\mu_{(\clas)}(\,\dots)_{\mathrm{b},\kappa,\til\kappa}$ similarly as (\ref{new212}). In the case $d=2q$ we replace $y$ by $(y,y').$ Let us set
$$L^\mu_{(\clas)}(\R^q;H,\til H)_{\mathrm{b}}:=\{\Op(a):a(y,y',\eta)\in S^\mu_{(\clas)}(\R^{2q}\times\R^q;H,\til H)_{\mathrm{b}}\}\label{Lmub}$$
for every $\mu\in\R$ (cf. the notation (\ref{new39})).

\begin{thm}\label{left}
Let $a(y,y',\eta)\in S^\mu_{(\clas)}(\R^q\times\R^q\times\R^q;H,\til H)_\mathrm{b},$ $\mu\in\R;$ then there are unique left and right symbols\index{symbol!left, right, double $-$}
$$a_{\mathrm L}(y,\eta)\;\textup{and}\;a_{\mathrm R}(y',\eta)\in S^\mu_{(\clas)}(\R^q\times\R^q\times\R^q;H,\til H)_\mathrm{b},$$
respectively$,$ such that
$$\Op(a)=\Op(a_{\mathrm L})=\Op(a_{\mathrm R}).$$
Those can be expressed by oscillatory integrals
$$a_{\mathrm L}(y,\eta)=\dint e^{-ix\xi}a(y,y+x,\eta+\xi)dx\dslash\xi$$
and
$$a_{\mathrm R}(y',\eta)=\dint e^{-ix\xi}a(y'+x,y',\eta-\xi)dx\dslash\xi.$$
We have asymptotic expansions
$$a_{\mathrm L}(y,\eta)=\sum_{\alpha\in\N^q}\frac1{\alpha!}\partial_\eta^\alpha D_{y'}^\alpha a(y,y',\eta)|_{y'=y}$$
and
$$a_{\mathrm R}(y',\eta)=\sum_{\alpha\in\N^q}\frac1{\alpha!}(-1)^{|\alpha|}\partial_\eta^\alpha D_y^\alpha a(y,y',\eta)|_{y=y'}.$$

\nt Moreover$,$ the mappings $a\mapsto a_{\mathrm L}$ and $a\mapsto a_{\mathrm R}$ are continuous. There are remainders
$$r_{\mathrm L,N}:=a_{\mathrm L}-\sum_{|\alpha|\le N}\frac1{\alpha!}\partial_\eta^\alpha D_{y'}^\alpha a(y,y',\eta)|_{y'=y}\in S^{\mu-(N+1)}(\R^q\times\R^q;H,\til H)_\mathrm{b},$$
$$r_{\mathrm L,N}=(N+1)\sum_{|\alpha|=N+1}\int_0^1\frac{(1-t)^N}{\alpha!}\dint e^{-ix\xi}(\partial_\eta^\alpha D_{y'}^\alpha a)(y,y+x,\eta+t\xi)dx\dslash\xi dt,$$
and
$$r_{\mathrm R,N}:=a_{\mathrm R}-\sum_{|\alpha|\le N}\frac1{\alpha!}(-1)^{|\alpha|}\partial_\eta^\alpha D_y^\alpha a(y,y',\eta)|_{y=y'}\in S^{\mu-(N+1)}(\R^q\times\R^q;H,\til H)_\mathrm{b},$$
$$r_{\mathrm R,N}=(N+1)\sum_{|\alpha|=N+1}\int_0^1\frac{(1-t)^N}{\alpha!}\dint e^{-ix\xi}(\partial_\eta^\alpha D_y^\alpha a)(y'+x,y',\eta-t\xi)dx\dslash\xi dt,$$
and the mappings $a\mapsto r_{\mathrm L,N}$ and $a\mapsto r_{\mathrm R,N}$ are also continuous between the respective symbol spaces.
\end{thm}

\begin{thm}\label{Multi} 
Let $H,\til H$ and $H_0$ be Hilbert spaces with group actions $\kappa,\til\kappa$ and $\kappa_0,$ respectively. Then $A\in L^\mu_{(\clas)}(\R^q;H_0,\til H)_\mathrm{b},$ $B\in L^\nu_{(\clas)}(\R^d;H,H_0)_\mathrm{b}$ entails $AB\in L^{\mu+\nu}_{(\clas)}(\R^q;H,\til H)_\mathrm{b}.$ If $A=\Op(a),$ $B=\Op(b)$ for left symbols\index{composition between pseudo-differential\newline operators}\index{pseudo-differential operators!composition between $-$}
$$a(y,\eta)\in S^\mu_{(\clas)}(\R^q\times\R^q;H_0,\til H)_\mathrm{b},\quad b(y,\eta)\in S^\nu_{(\clas)}(\R^q\times\R^q;H,H_0)_\mathrm{b}$$
then we have $AB=\Op(a\sharp b)$ for
$$(a\sharp b)(y,\eta)=\dint e^{-ix\xi}a(y,\eta+\xi)b(y+x,\eta)dx\dslash\xi$$
also referred to as the Leibniz product\index{Leibniz product} between $A$ and $b,$ where $(a\sharp b)(y,\eta)\in S^{\mu+\nu}_{(\clas)}(\R^q\times\R^q;H,\til H)_\mathrm{b},$ and $a\sharp b$ has an expansion
$$(a\sharp b)(y,\eta)=\sum_{|\alpha|\le N}\frac1{\alpha!}\partial_\eta^\alpha a(y,\eta)D_y^\alpha b(y,\eta)+r_N(y,\eta)$$
for a symbol $r_N(y,\eta)\in S^{\mu+\nu-(N+1)}_{(\clas)}(\R^q\times\R^q;H,\til H)_\mathrm{b}$ which has the form 
$$r_N=(N+1)\hspace{-2mm}\sum_{|\alpha|=N+1}\int_0^1\hspace{-1mm}\frac{(1-t)^N}{\alpha!}\hspace{-1mm}\dint e^{-ix\xi}(\partial_\eta^\alpha a)(y,\eta+t\xi)(D_y^\alpha b)(y+x,\eta)dx\dslash\xi dt$$
for every $N\in\N.$ The mappings $(a,b)\mapsto a\sharp b$ and $(a,b)\mapsto r_N$ are bilinear continuous in the respective spaces of symbols.
\end{thm}

%1.4
%
%
%
%
%
\section{The edge algebra}

{\footnotesize If $M$ is a manifold with edge $Y,$ in the simplest case a wedge $\Xw\times\R^q$ for a $C^\infty$ manifold $X$ as the base of the model cone and $\R^q$ as the edge, by the edge algebra we understand a subalgebra of $L_\clas^\mu(M\setminus Y)$ that contains all edge-degenerate differential operators together with the parametrices of elliptic elements. Asymptotics of solutions to elliptic equations will be obtained by using parametrices within the calculus. Here we study the case of constant discrete asymptotics. The ingredients of the edge algebra are organised in such a way that weighted edge spaces and subspaces with such asymptotics are expected under the action of operators.\\

\nt In this Section we give an overview on essential elements of the edge calculus in a new accessible form, later on used as a reference in \cite{Schu62}, \cite{Schu63}. More details and complete proofs, as far as they are not given here, may be found, for instance, in \cite{Schu20}.\par}

%1.4.1
%
%
%
%
%
\subsection{Weighted edge spaces and discrete asymptotics}\label{1.4.1}
In Section \ref{1.3.2} we introduced weighted edge spaces, based on Definitions \ref{1.3.5} and \ref{Fredge}. We now deepen this material and study subspaces with (constant) discrete asymptotics.\\

\nt Given a discrete asymptotic type\index{discrete asymptotic type}\index{asymptotic!discrete $-$ type} $\P,$ written in the form (\ref{discras}), associated with weight data\index{weight!data} $(\gamma,\Theta),$ $\Theta=(\vartheta,0],$ $-\infty\le\vartheta<0,$ according to the consideration in Section \ref{1.2.3} we have the spaces $\K_\P^{s,\gamma;g}(\X)$ as subspaces of $\K^{s,\gamma;g}(\X),$ $s,g\in\R.$ Using the fact that $\K_\P^{s,\gamma}(\X)$ is a \Fr space with group action $\kappa,$ see Definition \ref{1.3.3} and the formula (\ref{Klhom}) for $g=0,$ we can form the associated wedge space\index{weighted!spaces!edge!with constant discrete asymptotics}\index{edge!spaces!weighted $-$ with constant discrete\newline asymptotics}\index{constant discrete asymptotics!weighted edge spaces with $-$}
\begin{equation*}\newnot{WsP}
\WW^s(\R^q,\K_\P^{s,\gamma}(\X))=:\WW^{s,\gamma}_\P(\X\times\R^q)
\end{equation*}
which is again a \Fr space, cf. Definition \ref{Fredge}.

\begin{thm}
We have
\begin{equation}\label{new19}
\WW^s(\R^q,\K_\P^{s,\gamma}(\X))=\WW^s(\R^q,\K_\P^{\infty,\gamma}(\X))+\WW^s(\R^q,\K_\Theta^{s,\gamma}(\X))
\end{equation}
for every $s\in\R.$
\end{thm}

\begin{proof}
Because of the relation
$$\K_\P^{s,\gamma}(\X)=\K_\P^{\infty,\gamma}(\X)+\K_\Theta^{s,\gamma}(\X)$$
it suffices to apply \textup{Theorem \ref{sum} (i)}.
\end{proof}

\nt Let us characterise the spaces (\ref{new19}) in more detail where we assume $\Theta=(\vartheta,0]$ to be finite. In this case we have a representation
$$\K^{s,\gamma}_\P(\X)=\E_\P(\X)+\K^{s,\gamma}_\Theta(\X),$$
cf. the formula (\ref{new14}), and hence, by virtue of Theorem \ref{sum}, using the fact that $\K^{s,\gamma}_\Theta(\X)$ is preserved under the action of $\kappa=\{\kappa_\lambda\}_{\lambda\in\R_+},$
\begin{equation}\label{asnew24}
\WW^s(\R^q,\K^{s,\gamma}_\P(\X))=KH^s(\R^q,\E_\P(\X))+\WW^s(\R^q,\K^{s,\gamma}_\Theta(\X)).
\end{equation}

\nt The space $\WW^s(\R^q,\K^{s,\gamma}_\Theta(\X))$ represents distributions of edge flatness\index{edge!flatness relative to a weight}\index{flatness!relative to a weight!edge $-$ }\index{weight!flatness relative to a $-$!edge $-$} $-\vartheta-0$ relative to the weight $\gamma,$ while the elements of $KH^s(\R^q,\E_\P(\X))$ are the singular functions of edge asymptotics of type $\P.$ The cut-off function $\omega$ as well as the function $\eta\mapsto[\eta]$ are fixed. The influence under changing those data will be discussed below.\\

\nt Using (\ref{singdiscr}) and setting $E=H^s(\R^q,\E_\P(\X))$ we obtain the singular functions of (constant) discrete asymptotics\index{singular functions!of constant discrete asymptotics}\index{constant discrete asymptotics!singular functions of $-$}
\begin{eqnarray}\label{discsing}
\notag KE & = &F^{-1}_{\eta\rightarrow y}\Big\{[\eta]^{\frac{n+1}2}\omega(r[\eta])\displaystyle\sum_{j=0}^N\displaystyle\sum_{l=0}^{m_j}\hat b_{jl}(x,\eta)(r[\eta])^{-p_j}\log^l(r[\eta]):\\[-5mm]
\\[-1mm]
\notag & & \hspace{1cm}b_{jl}(x,y')\in C^\infty(X,H^s(\R^q_{y'})),\,0\le l\le m_j,\,j=0,\dots,N\Big\},
\end{eqnarray}
$\hat b_{jl}(x,\eta):=F_{y'\rightarrow\eta}b_{jl}(x,y').$ In other words, every $u\in\WW^s(\R^q,\K^{s,\gamma}_\P(\X))$ can be written in the form
\begin{equation}\label{decomp}
u=u_\mathrm{sing}+u_\mathrm{flat}
\end{equation}
with $u_\mathrm{sing}$ being an element of  (\ref{discsing}) and $u_\mathrm{flat}\in\WW^s(\R^q,\K^{s,\gamma}_\Theta(\X)).$

\begin{rem}
Let $E$ be a \Fr space $E$ with group action $\Kll\!;$ according to \textup{Proposition \ref{1.3.10}} we have $\WW^\infty(\R^q,E)=H^\infty(\R^q,E),$ i.e.$,$ we may ignore $\kappa$ in this case. Thus$,$ analogously as \textup{(\ref{asnew24})} we have
\begin{equation}\label{asnew25}
\WW^\infty(\R^q,\K^{\infty,\gamma}_\P(\X))=H^\infty(\R^q,\E_\P(\X))+H^\infty(\R^q,\K^{\infty,\gamma}_\Theta(\X)).
\end{equation}
In particular$,$ for the space of singular functions it follows that
$$H^\infty(\R^q,\E_\P(\X))=C^\infty\big(X,H^\infty(\R^q,\E_\P(\R_+))\big),$$
see the formula \textup{(\ref{asnew11}),} i.e.$,$ the singular functions of the discrete asymptotics for $s=0$ are of the form $c_{jl}(x,y)\omega(r)r^{-p_j}\log^lr,$ $c_{jl}\in C^\infty(X,H^\infty(\R^q)).$
\end{rem}

\begin{rem}
Let us fix $p\in\C,$ $\Re p<\frac{n+1}2-\gamma,$ $l\in\N,$ $b\in C^\infty(X);$ then
\begin{equation}\label{45a}
c\mapsto[\eta]^{\frac{n+1}2}\omega(r[\eta])b(x)(r[\eta])^{-p}\log^l(r[\eta])c
\end{equation}
represents an element
\begin{equation}\label{45b}
k(\eta)\in S^0_\clas(\R^q;\C,\K^{\infty,\gamma}(\X))_{\id,\kappa},
\end{equation}
see \textup{Definition \ref{1.3.14}}.
\end{rem}

\nt In fact, we have
$$k(\lambda\eta)c=\kappa_\lambda k(\eta)c$$
for all $c\in\C,$ $\lambda\ge1,$ $|\eta|$ sufficiently large, and 
\begin{equation}\label{keta}
k(\eta)\in C^\infty(\R^q,\L(\C,\K^{\infty,\gamma}(\X)))
\end{equation}
(see Remark \ref{1.3.16}).\\

\nt Denote for the moment the potential symbol\index{potential!symbol}\index{symbol!potential $-$} (\ref{keta}) associated with $p,l,b$ by $k_{p,l,b}(\eta),$ then the associated singular functions have the form $\Op_y(k)v$ for $v\in H^s(\R^q).$ In general, a function $b_{jl}(x,y')\in C^\infty(X,H^s(\R^q_{y'}))$ has a convergent expansion
$$b_{jl}(x,y')=\sum_{\nu=0}^\infty\lambda_\nu b_{jl\nu}(x)v_{jl\nu}(y'),$$
$\lambda_\nu\in\C,$ $\sum_\nu|\lambda_\nu|<\infty,$ with sequences $(b_{jl\nu})_{\nu\in\N}\subset C^\infty(X),$ $(v_{jl\nu})_{\nu\in\N}\subset H^s(\R^q),$ tending to zero in the respective spaces, and then our singular functions (\ref{discsing}) have the form
\begin{equation}\label{45c}
u_\mathrm{sing}=\sum_{\nu=0}^\infty\sum_{j=0}^N\sum_{l=0}^{m_j}\lambda_\nu\Op(k_{p_j,l,b_{jl\nu}})v_{jl\nu}.
\end{equation}

\nt In the definition of singular terms of the asymptotics we fixed the function $\eta\mapsto[\eta].$ Let us compare the results between different such functions $[\eta]$ and $[\eta]_1.$ On the level of potential symbols (\ref{45b}) coming from (\ref{45a}) and their analogues $k_1(\eta)$ based on $[\eta]_1,$ we have
$$k(\eta)-k_1(\eta)\in S^{-\infty}(\R^q;\C,\K^{\infty,\gamma}(\X)),$$
since this difference has compact support in $\eta.$ Thus
$$\Op(k-k_1):H^s(\R^q)\rightarrow\WW^\infty(\R^q,\K^{\infty,\gamma}_\P(\X)),$$
i.e., we obtain a difference of the representation of singular functions belonging to the space (\ref{asnew25}). In the general case the expansion (\ref{45c}) shows a similar effect.

\begin{defn}\label{glatt}
The space of smoothing operators on the wedge $\Xw\times\Omega$ with $($constant$)$ discrete asymptotics\index{smoothing operators!on the wedge!with constant discrete asymptotics}\index{operators!smoothing $-$ on the wedge!with constant discrete asymptotics}\index{constant discrete asymptotics!smoothing operators on the wedge with $-$} ${L}^{-\infty}(\Xw\times\Omega,\g),$ $\g=(\gamma,\sigma,\Theta),$ is defined as the set of all
$$C:\WW^s_\comp(\Omega,\KsgX)\rightarrow\WW^\infty_\loc(\Omega,\K^{\infty,\sigma}(\X))$$
that are continuous for all $s\in\R$ and such that $C$ and $C^*$ induce continuous operators
$$C:\WW^s_\comp(\Omega,\KsgX)\rightarrow\WW^\infty_\loc(\Omega,\K^{\infty,\sigma}_\P(\X)),$$
and
$$C^*:\WW^s_\comp(\Omega,\K^{s,-\sigma}(\X))\rightarrow
\WW^\infty_\loc(\Omega,\K^{\infty,-\gamma}_\Q(\X))$$
for all $s\in\R,$ for some asymptotic types $\P$ and $\Q,$ associated with the weight data $(\sigma,\Theta)$ and $(-\gamma,\Theta),$ respectively. More generally$,$ given $j_-,j_+\in\N,$ by ${{L}}^{-\infty}(\Xw\times\Omega,\g;j_-,j_+)$ we denote the space of all
$$C:\WW^s_\comp(\Omega,\K^{s,\gamma}(\X))\oplus H_\comp^s(\Omega,\C^{j_-})\rightarrow
\WW^\infty_\loc(\Omega,\K^{\infty,\sigma}_\Q(\X))\oplus H_\loc^\infty(\Omega,\C^{j_+})$$
that are continuous for all $s\in\R$ and such that $C$ and $C^*$ induce continuous operators analogously as in the case of upper left corners$,$ here including the extra spaces over $\Omega.$
\end{defn}

\begin{rem}
Observe that we have
$$\Phi L^{-\infty}(\X\times\Omega)\Phi'\subset{{L}}^{-\infty}(\Xw\times\Omega,\g)$$
for every $\Phi,\Phi'\in C^\infty(\R_+)$ vanishing near zero.
\end{rem}

%1.4.2
%
%
%
%
%
\subsection{Edge quantisation}\label{S1.4.2}

The edge quantisation\index{edge!quantisation}\index{quantisation!edge $-$} is a rule to pass from edge-degenerate operators (or symbols)\index{edge-degenerate!operator}\index{edge-degenerate!symbol}\index{operators!edge-degenerate $-$}\index{symbol!edge-degenerate $-$} to a representation modulo elements of order $-\infty$ such that the quantised objects are continuous between the chosen weighted edge spaces. Analogously as in Section \ref{S1.1.1} on a Manifold $M$ with edge $Y$ we proceed as follows. We define the space $L_\deg^\mu(M)\subset L^\mu(M\setminus Y)$ consisting of all operators $A$ that are modulo $L^{-\infty}(M\setminus Y)$ in the local splitting of variables $(r,x,y)$ close to $Y,$ cf. the formula \eqref{regchart}, of the form
\begin{equation}\label{AA1}
A=r^{-\mu}\Op_{r,x,y}(a)
\end{equation}
for a symbol $a(r,x,y,\rho,\xi,\eta)=\til a(r,x,y,r\rho,\xi,r\eta),$ $\til a(r,x,\til\rho,\xi,\til\eta)\in S^\mu_\clas(\Rr\times\Sigma\times\Omega\times\R_{\til\rho,\xi,\til\eta}^{1+n+q})$ (where $\Sigma\subseteq\R^n$ corresponds to a chart on $X).$ An alternative is to write (again, modulo $L^{-\infty}(M\setminus Y)$)
\begin{equation}\label{AA2}
A=r^{-\mu}\Op_{r,y}(p)
\end{equation}
locally near $Y$ for an operator family $p(r,y,\rho,\eta)=\til p(r,y,r\rho,r\eta),$ $\til p(r,y,\til\rho,\til\eta)\in C^\infty(\Rr\times\Omega,L^\mu_\clas(X;\R_{\til\rho,\til\eta}^{1+q})).$\\

\nt Similarly in the conical case the latter representation will be dominating in the considerations below. However, from \eqref{AA1} we see that the homogeneous principal symbol $\sigma_\psi(A)$ of order $\mu,$ invariantly defined as a function in $C^\infty(T^*(M\setminus Y)\setminus 0),$ has a representation locally near $Y$ in the form
$$\sigma_\psi(A)(r,x,y,\rho,\xi,\eta)=r^{-\mu}\til\sigma_\psi(A)(r,x,y,r\rho,\xi,r\eta)$$
for a $\til\sigma_\psi(A)(r,x,y,\til\rho,\xi,\til\eta)$ which is smooth up to $r=0.$ We will call $\til\sigma_\psi(A)$ the reduced principal symbol of $A.$\\

\nt It is straightforward that every $A\in L_\deg^\mu(M)$ admits a properly supported representative modulo $L^{-\infty}(M\setminus Y).$

\begin{thm}\label{comp8A}
$A\in L_\deg^\mu(M),$ $B\in L_\deg^\nu(M)$ and $A$ or $B$ properly supported entails $AB\in L_\deg^{\mu+\nu}(M)$ and we have
$$\sigma_\psi(AB)=\sigma_\psi(A)\sigma_\psi(B),\quad\til\sigma_\psi(AB)=\til\sigma_\psi(A)\til\sigma_\psi(B).$$
\end{thm}

\begin{defn}
An $A\in L_\deg^\mu(M)$ is called $\sigma_\psi$-elliptic $($of order $\mu)$ if $\sigma_\psi(A)\ne0$ on $T^*(M\setminus Y)\setminus0$ as usual$,$ and locally close to $Y,$ $\til\sigma_\psi(A)\ne0$ up to $r=0.$
\end{defn}

\begin{thm}
$A\in L_\deg^\mu(M)$ $\sigma_\psi$-elliptic has a $($properly supported$)$ parametrix $P\in L_\deg^{-\mu}(M)$ in the sense
$$1-PA,1-AP\in L^{-\infty}(M\setminus Y),$$
and $\sigma_\psi(P)=\sigma_\psi(A)^{-1},$ $\til\sigma_\psi(P)=\til\sigma_\psi(A)^{-1}.$
\end{thm}

\nt By edge quantisation we understand a rule to pass from an edge-degenerate symbol
$$p(r,x,y,\rho,\xi,\eta)=\til p(r,x,y,r\rho,\xi,r\eta)$$
for $\til p(r,x,y,\til\rho,\xi,\til\eta)\in S^\mu_\clas(\Rr\times\Sigma\times\Omega\times\R^{1+n+q}),$
$\Sigma\subseteq\R^n,$ $\Omega\subseteq\R^q$ open, to a Mellin symbol\index{symbol!Mellin $-$}\index{Mellin!symbols}
$$f(r,x,y,z,\xi,\eta)=\til f(r,x,y,z,\xi,r\eta),$$
$\til f(r,x,y,z,\xi,\til\eta)\in S^\mu_\clas(\Rr\times\Sigma\times\Omega\times\Gamma_{\frac{n+1}2-\gamma}\times\R^{n+q}_{\xi,\til\eta}),$ with $z$ running over a weight line $\Gamma_{\frac{n+1}2-\gamma},$ such that
$$\Op_{r,x,y}(p)=\op_M^{\gamma-\frac n2}\Op_{x,y}(f)$$
modulo $L^{-\infty}(\R_+\times\Sigma\times\Omega),$ as operators $C_0^\infty(\R_+\times\Sigma\times\Omega)\rightarrow C^\infty(\R_+\times\Sigma\times\Omega).$ In order to have a more concise description we mainly talk about operator-valöued amplitude functions where the action in $x$ is already carried out, globally on a $C^\infty$ manifold $X.$ Then the task to construct a correspondence $p\mapsto f$ will be formulated on the level of $L_\clas^\mu(X)$-valued amplitude functions. It turns out that we find the Mellin symbol in a very specific way, namely as a holomorphic operator function, now denoted by $h.$

\begin{thm}\label{qua}\index{Mellin!quantisation}\index{quantisation!Mellin $-$}
For every
\begin{equation}\label{new37a}
p(r,y,\rho,\eta)=\til p(r,y,r\rho,r\eta),\quad\til p(r,y,\til\rho,\til\eta)\in C^\infty\big(\Rr\times\Omega,L^\mu_\clas(X;\R^{1+q}_{\til\rho,\til\eta})\big)
\end{equation}
there exists an operator function\index{symbol!Mellin $-$!holomorphic $-$}\index{Mellin!symbol!holomorphic $-$}\index{holomorphic!Mellin symbol}
$$h(r,y,z,\eta)=\til h(r,y,z,r\eta),\quad\til h(r,y,z,\til\eta)\in C^\infty\big(\Rr\times\Omega,M^\mu_\O(X;\R^q_{\til\eta})\big),$$
such that for every fixed $\gamma\in\R$
\begin{equation}\label{new37}
\Op_y(p)=\Op_y(\op_M^{\gamma-\frac n2}(h))\;\mod\,C^\infty\Set{\R_+\times\Omega,L^{-\infty}(\R_+\times X\times\Omega)},
\end{equation}
as operators $C_0^\infty(\R_+\times X)\rightarrow C^\infty(\R_+\times X)$ for all $(y,\eta)\in\Omega\times\R^q.$
\end{thm}

\begin{thm}\label{symb26}
\begin{itemize}
\item[\textup{(i)}]Let $\til h(r,y,z,\til\eta)\in C^\infty(\Rr\times\Omega,M_\O^\mu(X;\R_{\til\eta}^q)),$ and set $$h(r,y,z,\eta):=\til h(r,y,z,r\eta),$$
and
$$a(y,\eta):=\omega(r[\eta])r^{-\mu}\op_M^{\gamma-\frac n2}(h)(y,\eta)\omega'(r[\eta]),$$
for some cut-off functions $\omega,\omega'.$ Then 
$$a(y,\eta)\in S^\mu(\Omega\times\R^q;\K^{s,\gamma}(\X),\K^{s-\mu,\gamma-\mu}(\X))$$
for every $s\in\R.$
\item [\textup{(ii)}]Let $\til p(r,y,\til\rho,\til\eta)\in C^\infty(\Rr\times\Omega,L^\mu_\clas(X;\R_{\til\rho,\til\eta}^{1+q})),$ and set
$$p(r,y,\rho,\eta):=\til p(r,y,r\rho,r\eta)$$ and
$$b(y,\eta):=\epsilon(r)(1-\omega(r[\eta]))r^{-\mu}\Op_r(p)(y,\eta)(1-\omega''(r[\eta]))\epsilon'(r)$$
for arbitrary cut-off functions $\omega,\omega'',\epsilon,\epsilon'.$ Then
\begin{equation}\label{K28}
b(y,\eta)\in S^\mu(\Omega\times\R^q;\K^{s,\gamma}(\X),\K^{s-\mu,\infty}(\X))
\end{equation}
for every $s,\gamma\in\R.$
\item [\textup{(iii)}]Let $p_{\mathrm{int}}(r,y,\rho,\eta)\in C^\infty(\R_+\times\Omega,L^\mu_\clas(X;\R_{\rho,\eta}^{1+q}))$ and $\varphi,\varphi'\in C^\infty_0(\R_+).$ Then for
$$c(y,\eta):=\varphi(r)\Op_r(p_{\mathrm{int}})(y,\eta)\varphi'(r)$$ we have
$$c(y,\eta)\in S^\mu(\Omega\times\R^q;\K^{s,\gamma}(\X),\K^{s-\mu,\infty}(\X))$$
for every $s,\gamma\in\R.$
\end{itemize}
\end{thm}

\begin{thm}\label{sm29}
For every $\til p(r,y,\til\rho,\til\eta)$ in the space $C^\infty(\Rr\times\Omega,L^\mu_\clas(X;\R_{\til\rho,\til\eta}^{1+q})),$ $p(r,y,\rho,\eta):=\til p(r,y,r\rho,r\eta)$ and $h(r,y,z,\eta)$ as in \textup{Theorem \ref{qua}} for arbitrary cut-off functions $\omega,\omega',\omega'',\epsilon,\epsilon',$ $\omega''\prec\omega\prec\omega',$ we have
\begin{eqnarray*}
\epsilon r^{-\mu}\Op_r(p)(y,\eta)\epsilon' & = & \epsilon\{\omega(r[\eta])r^{-\mu}\op_M^{\gamma-\frac n2}(h)(y,\eta)\omega'(r[\eta])\\[2mm]
& + & (1-\omega(r[\eta]))r^{-\mu}\Op_r(p)(y,\eta)(1-\omega''(r[\eta]))\}\epsilon' 
\end{eqnarray*}
modulo $C^\infty(\Omega,L^{-\infty}(\X;\R^q)).$
\end{thm}

\nt The operator-valued symbols
\begin{equation}\label{new38}
\begin{array}{rcl}
a(y,\eta)&=&\epsilon\big\{\omega(r[\eta])r^{-\mu}\op_M^{\gamma-\frac n2}(h)(y,\eta)\omega'(r[\eta])\\[4mm]
&+&(1-\omega(r[\eta]))r^{-\mu}\Op_r(p)(y,\eta)(1-\omega''(r[\eta]))\big\}\epsilon'\\[4mm]
&+&\varphi\Op_r(p_{\mathrm{int}})(y,\eta)\varphi'
\end{array}
\end{equation}
with $p,h$ and $\omega,\omega',\omega'',\epsilon,\epsilon'$ as in Theorem \ref{sm29} and $p_{\mathrm{int}},\varphi,\varphi'$ as in Theorem \ref{symb26} (iii) will play the role of the non-smoothing amplitude functions of the edge operator calculus\index{amplitude functions!of the edge calculus}\index{edge!amplitude functions of the $-$ calculus}.

\begin{rem}
Given \textup{\eqref{new37a}} the rule to pass from $r^{-\mu}\Op_r(p)(y,\eta)$ to
\begin{eqnarray*}
\epsilon a_{\mathrm{edge}}(y,\eta)\epsilon' & := & \epsilon\{\omega(r[\eta])r^{-\mu}\op_M^{\gamma-\frac n2}(h)(y,\eta)\omega'(r[\eta])\\[2mm]
& + & (1-\omega(r[\eta]))r^{-\mu}\Op_r(p)(y,\eta)(1-\omega''(r[\eta]))\}\epsilon'
\end{eqnarray*}
is what we understand by edge quantisation\index{edge!quantisation!with respect to a weight}\index{quantisation!edge $-$!with respect to a weight}\index{weight!edge quantisation with respect to a $-$} $($with respect to a chosen weight $\gamma)$ close to the edge. It can be proved that
$$\epsilon r^{-\mu}\Op_r(p)(y,\eta)\epsilon'=\epsilon a_{\mathrm{edge}}(y,\eta)\epsilon'\qquad\mod C^\infty(\Omega,L^{-\infty}(\X;\R_\eta^q))$$
in the sense of pseudo-differential operators on $\X,$ i.e.$,$ operators
$$C_0^\infty(\X)\rightarrow C^\infty(\X).$$
This entails
\begin{equation}\label{g9}
\epsilon r^{-\mu}\Op_{r,y}(p)\epsilon'=\epsilon\Op_y(a_{\mathrm{edge}})\epsilon'\qquad\mod L^{-\infty}(\X\times\Omega).
\end{equation}
\end{rem}

\nt Associated to the operator function (\ref{new38}), we have a (edge-degenerate) homogeneous principal symbol\index{edge-degenerate!symbol!principal $-$}\index{symbol!edge-degenerate $-$!principal $-$}\index{principal symbol!edge-degenerate $-$} $\sigma_\psi(a)(r,x,y,\rho,\xi,\eta)$ which has the form
\begin{equation}\label{sym38}
\sigma_\psi(a)(r,x,y,\rho,\xi,\eta)=r^{-\mu}\epsilon(r)\til\sigma_\psi(a)(r,x,y,r\rho,\xi,r\eta)\epsilon'(r)+\sigma_\psi(\varphi\Op_r(p_{\mathrm{int}})\varphi')
\end{equation}
where $\til\sigma_\psi(a)(r,x,y,\til\rho,\xi,\til\eta)\in C^\infty(T^*(\X\times\Omega)\setminus0)$ is the parameter-dependent homogeneous principal symbol\index{homogeneous!principal symbol}\index{symbol!homogeneous principal $-$}\index{principal symbol!homogeneous $-$} (of order $\mu)$ of the function $\til p(r,x,y,\til\rho,\xi,\til\eta)$ occurring in Theorem \ref{sm29}, smooth up to $r=0,$ and the second summand in (\ref{sym38}) is the standard homogeneous principal symbol (of order $\mu)$ of the operator $\varphi\Op_r(p_{\mathrm{int}})\varphi'\in L_\clas^\mu(\X\times\Omega).$ Moreover, setting
\begin{equation}\label{new38a}
p_0(r,y,\rho,\eta):=\til p(0,y,r\rho,r\eta),\quad h_0(r,y,z,\eta):=\til h(0,y,z,r\eta)
\end{equation}
we also have a relation of the kind (\ref{new37}), namely,
$$\Op_y(p_0)=\Op_y(\op_M^{\gamma-\frac n2}(h_0))\quad\mod C^\infty(\R_+\times\Omega,L^{-\infty}(\X\times\Omega)),$$
and we define the homogeneous principal edge symbol
\begin{equation}\label{sig1}
\begin{array}{rcl}
\sigma_\wedge(a)(y,\eta)&:=&\omega(r|\eta|)r^{-\mu}\op_M^{\gamma-\frac n2}(h_0)(y,\eta)\omega'(r|\eta|)\\[4mm]
&+&(1-\omega(r|\eta|))r^{-\mu}\Op_r(p_0)(y,\eta)(1-\omega''(r|\eta|)).
\end{array}
\end{equation}

%1.4.3
%
%
%
%
%
\subsection{Smoothing Mellin, trace, potential and Green operators}\label{S4.3}

The amplitude functions of operators on a manifold with edge consist of symbols in the sense of Definition \ref{1.3.14} with values in the cone algebra of the infinite cone $\Xw$ (see Section \ref{1.2}). According to the structure of the cone algebra, we have, in particular, Green operator-valued symbols.

\begin{defn}\label{Green}
A $g(y,\eta)\in\bigcap_{s,g\in\R} S^\mu_\clas(\Omega\times\R^q;\K^{s,\gamma;g}(\X),\K^{\infty,\sigma;\infty}(\X))$ for an open set $\Omega\subset\R^q$ is called a Green symbol of order $\mu\in\R,$ associated with the weight data\index{Green symbols!associated with weight data}\index{symbols!Green $-$ associated with weight data}\index{weight!data!Green symbols associated with $-$} $\g:=(\gamma,\sigma,\Theta),$\label{g43} for weights $\gamma,\sigma\in\R$ and some interval\index{weight!interval}\index{interval!weight $-$} $\Theta,$ if
$$g(y,\eta)\in\bigcap_{s,g\in\R} S^\mu_\clas(\Omega\times\R^q;\K^{s,\gamma;g}(\X),\S^\sigma_\P(\X)),$$
and
$$g^*(y,\eta)\in\bigcap_{s,g\in\R} S^\mu_\clas(\Omega\times\R^q;\K^{s,-\sigma;g}(\X),\S^{-\gamma}_\Q(\X)),$$
where $\P$ and $\Q$ are $(g$-dependent$)$ asymptotic types associated with $(\sigma,\Theta)$ and $(-\gamma,\Theta),$ respectively. The $(y,\eta)$-wise formal adjoint refers to the non-degenerate sesquilinear pairing $(\ref{Pair}).$ By $\RR^\mu_G(\Omega\times\R^q,\g)$\label{RG} we denote the set of all such operator-valued symbols $g,$ called Green symbols $($with asymptotics$),$ for arbitrary $\P$ and $\Q.$ Moreover$,$ let $\RR^\mu_G(\Omega\times\R^q,\g)_{\P,\Q}$\label{RmuPQ} denote the subspace of Green symbols for fixed $\P$ and $\Q.$
\end{defn}

\nt In an analogous manner we define double symbols $g(y,y',\eta)$ of the class $\RR^\mu_G(\Omega\times\Omega\times\R^q,\g),$ $\RR^\mu_G(\Omega\times\Omega\times\R^q,\g)_{\P,\Q}$ $($cf. \textup{Definition} $\ref{1.3.14}$ where the dimensions of variables and covariables may be independent$).$\\

\nt For notational purpose we mainly consider Green symbols with $y$ and $\eta$ of the same dimension. Corresponding results will have a straightforward extension to open sets $\Omega$ of arbitrary dimension, such as the following theorem. Moreover, observe that $\RR^\mu_G(\Omega\times\R^q,\g)_{\P,\Q},$ for fixed $\P$ and $\Q,$ is \Fr in a natural way.

\begin{thm}
Let $g_j(y,\eta)\in\RR^{\mu-j}_G(\Omega\times\R^q,\g),$ $j\in\N,$ be an arbitrary sequence where the involved asymptotic types are independent of $j.$ Then there is a $g(y,\eta)\in\RR^\mu_G(\Omega\times\R^q,\g)$ which is the asymptotic sum $g(y,\eta)\sim\sum_{j=0}^\infty g_j(y,\eta),$ i.e.$,$ we have
$$g(y,\eta)-\sum_{j=0}^Ng_j(y,\eta)\in\RR^{\mu-(N+1)}_G(\Omega\times\R^q,\g)$$
for every $N\in\N.$ Every such $g$ is unique modulo $\RR^{-\infty}_G(\Omega\times\R^q,\g).$
\end{thm}

\nt Let us now generalise Definition \ref{Green} to the case of $2\times2$ block matrix symbols where the upper left corners are as before but the extra entries off the diagonal play the role of trace and potential symbols, respectively.

\begin{defn}\label{Green2}
An element
$$g(y,\eta)\in\bigcap_{s,g\in\R}S^\mu_\clas(\Omega\times\R^q;\K^{s,\gamma;g}(\X)
\oplus\C^{j_-},\K^{\infty,\sigma;\infty}(\X)\oplus\C^{j_+})$$
is called a Green symbol of order $\mu\in\R,$ associated with the weight data $\g:=(\gamma,\sigma,\Theta),$ for weights $\gamma,\sigma\in\R$ and some weight interval $\Theta,$ if
$$g(y,\eta)\in\bigcap_{s,g\in\R} S^\mu_\clas\Set{\Omega\times\R^q;\K^{s,\gamma;g}(\X)\oplus\C^{j_-},\S_\P^\sigma(\X)
\oplus\C^{j_+}}\!,$$
and
$$g^*(y,\eta)\in\bigcap_{s,g\in\R} S^\mu_\clas\Set{\Omega\times\R^q;\K^{s,-\sigma;g}(\X)\oplus\C^{j_+},
\S^{-\gamma}_\Q(\X)\oplus\C^{j_-}}$$
for some $(g$-dependent$)$ asymptotic types $\P$ and $\Q$ associated with $(\sigma,\Theta)$ and $(-\gamma,\Theta),$ respectively. The $(y,\eta)$-wise formal adjoint refers to the non-degenerate sesquilinear pairings
$$\Set{\K^{s,\gamma;g}(\X)\oplus\C^j}\times\Set{\K^{-s,-\gamma;-g}(\X)
\oplus\C^j}\rightarrow\C$$
via the scalar product $(u,v)_{\K^{0,0;0}(\X)\oplus\C^j},$ for $s,\gamma,g\in\R,$ $j\in\N.$ The set of all those operator-valued symbols will be denoted by $\RR^\mu_G(\Omega\times\R^q,\g;j_-,j_+)_{\P,\Q},$ and by $\RR^\mu_G(\Omega\times\R^q,\g;j_-,j_+)$\label{RGj} the union of these spaces over all $\P,\Q.$
\end{defn}

\nt Writing
\begin{equation}\label{g29}
g(y,\eta)=\Set{g_{i,j}(y,\eta)}_{i,j=1,2}
\end{equation}
with $g_{i,j}$ being the entries of the respective $2\times2$ block matrix, we also call $g_{21}(y,\eta)$ a trace and $g_{12}(y,\eta)$ a potential symbol. The notation Green symbol for matrices (\ref{g29}) extends Definition \ref{Green} and is used to emphazise that the formal properties of the block matrices (\ref{g29}) are similar to those of the upper left corners. Note that the lower right corners are $j_+\times j_-$ matrices of classical scalar symbols of order $\mu.$ The definition of the symbol spaces in Definition \ref{Green2} refers to the group action $\Kl\oplus\id$ in both spaces where \virg{id} means the trivial group action in the respective finite-dimensional space.

\begin{defn}\label{1.4.12}
The space ${L}^\mu_G(\Xw\times\Omega,\g;j_-,j_+),$ $\mu\in\R,$ $\g=(\gamma,\sigma,\Theta),$ of so-called Green operators\index{Green!operators!with constant discrete asymptotics}\index{operators!Green $-$!with constant discrete asymptotics}\index{asymptotics!constant discrete $-$!Green operators with $-$} on $\Xw\times\Omega$ with $($constant discrete$)$ asymptotics is defined as the set of all
\begin{equation}\label{50}
G=\Op_y(g)+C
\end{equation}
for a symbol $g(y,\eta)\in\RR_G^\mu(\Omega\times\R^q,\g;j_-,j_+)$ and a smoothing operator $C$ in the sense of \textup{Definition} $\ref{glatt}.$ In the case $j_-,j_+=0$ we simply write ${L}^\mu_G(\Xw\times\Omega,\g).$
\end{defn}

\nt We call an operator $G\in{L}^\mu_G(\Xw\times\Omega,\g;j_-,j_+)$ \textit{properly supported}\index{properly supported operators}\index{operators!properly supported $-$} (with respect to the variables on $\Omega)$ if the support of its operator-valued distributional kernel over $\Omega\times\Omega\ni(y,y')$ intersects every strip $\set{(y,y')\in\Omega\times\Omega:y\in K,y'\in\Omega}$ and $\set{(y,y')\in\Omega\times\Omega:y\in\Omega,y'\in K}$ for arbitrary $K\Subset\Omega,$ in a compact set.

\begin{thm}\label{4.13}
Every $G\in{L}^\mu_G(\Xw\times\Omega,\g;j_-,j_+)$ can be written in the form
$$G=G_0+C$$
for a properly supported operator $G_0\in{L}^\mu_G(\Xw\times\Omega,\g)$ and $C\in{L}^{-\infty}(\Xw\times\Omega,\g).$
\end{thm}

\nt In fact, it suffices to replace $g(y,\eta)$ in (\ref{50}) by $g_0(y,y',\eta):=\psi(y,y')g(y,\eta)$ for a $\psi(y,y')\in C^\infty(\Omega\times\Omega)$ of proper support such that $\psi(y,y')=1$ in a suitable neighbourhood of $\mathrm{diag}(\Omega\times\Omega).$\\

\nt Let us now add so-called smoothing Mellin symbols of the edge calculus with (constant discrete) asymptotics. Those are $(y,\eta)$-depending families of operators taking values in the space ${L}^\mu_{M+G}(\X,\g),$ for $\g=(\gamma,\gamma-\mu,(-(k+1),0]),$ $k\in\N,$ cf. Definition \ref{1.2.34}.

\begin{defn}\label{1.4.14}
Let $\RR^\mu_{M+G}(\Omega\times\R^q,\g)$ be defined as the space of all operator functions of the form $m(y,\eta)+g(y,\eta)$ for arbitrary $g(y,\eta)\in\RR^\mu_G(\Omega\times\R^q,\g)$ and
\begin{equation}\label{new40b}
m(y,\eta):=r^{-\mu}\omega(r[\eta])\sum_{j=0}^k\sum_{|\alpha|\le j}r^j\op_M^{\gamma_{j\alpha}-\frac n2}(f_{j\alpha})(y)\eta^\alpha\omega'(r[\eta])
\end{equation}
for arbitrary $f_{j\alpha}(y,z)\in C^\infty(\Omega,M_{\RR_{j\alpha}}^{-\infty}(X))$ for $($constant discrete$)$ asymptotic types $\RR_{j\alpha}$ and weights $\gamma_{j\alpha}$ satisfying
\begin{equation}\label{w39}
\gamma-j\le\gamma_{j\alpha}\le\gamma,\quad\pi_\C\RR_{j\alpha}\cap\Gamma_{\frac{n+1}2-\gamma_{j\alpha}}=\emptyset.
\end{equation}

\nt In a similar manner we define $\RR_{M+G}^\mu(\Omega\times\R^q,\g;j_-,j_+)$ to be the space of all $2\times2$ block matrix-valued symbols of the form $\diag(m(y,\eta),0)+g(y,\eta)$ for arbitrary $m(y,\eta)\in\RR_{M+G}^\mu(\Omega\times\R^q,\g),$ and $g(y,\eta)\in\RR_G^\mu(\Omega\times\R^q,\g;j_-,j_+).$
\end{defn}

\begin{prop}\label{1.4.15}
The space $\RR^\mu_{M+G}(\Omega\times\R^q,\g;j_-,j_+)$ is contained in $S^\mu_\clas\big(\Omega\times\R^q;\K^{s,\gamma;g}(\X)\oplus\C^{j_-},\K^{\infty,\gamma-\mu;\infty}(\X)\oplus\C^{j_+}\big)$ and in $S^\mu_\clas\big(\Omega\times\R^q;\K^{s,\gamma;g}_\P(\X)\oplus\C^{j_-},\K^{\infty,\gamma-\mu;\infty}_\Q(\X)\oplus\C^{j_+}\big)$ for every $s,g\in\R,$ and for any $($constant discrete$)$ asymptotic type $\P,$ for some resulting $\Q.$
\end{prop}

\begin{proof}
In fact, it suffices to apply Remark \ref{1.3.16} (which holds in analogous form also in the case of \Fr spaces $H,\til H)$ to see that, apart from the Green summand, the symbol $m(y,\eta)$ is a finite sum of operator functions satisfying the homogeneity relation for $\lambda\ge1,$ $|\eta|\ge C,$ with order $\mu-j+|\alpha|,$ and that we have the mapping properties (\ref{AB}).
\end{proof}

\nt Given a symbol $(m+g)(y,\eta)\in\RR_{M+G}^\mu(\Omega\times\R^q,\g;j_-,j_+)$ we define $\sigma_\wedge(m+g)(y,\eta),$ the homogeneous principal edge symbol of order $\mu,$ to be the homogeneous principal component of the classical operator-valued symbols in the sense of Proposition \ref{1.4.15}. In particular, for a Green symbol $g(y,\eta)\in\RR_G^\mu(\Omega\times\R^q,\g;j_-,j_+)$ we have $\sigma_\wedge(g)(y,\eta).$ Observe that for $m(y,\eta)$ of the form (\ref{new40b}) we have
\begin{equation}\label{new41}
\sigma_\wedge(m)(y,\eta)=r^{-\mu}\omega(r|\eta|)\sum_{j=0}^k\sum_{|\alpha|=j}r^j\op_M^{\gamma_{j\alpha}-\frac n2}(f_{j\alpha})(y)\eta^\alpha\omega'(r|\eta|).
\end{equation}

\begin{defn}\label{1.4.16}
The space ${L}^\mu_{M+G}(\Xw\times\Omega,\g;j_-,j_+)$ of so-called smoothing Mellin plus Green operators with $($constant discrete$)$ asymptotics$,$ for $\mu\in\R$ and weight data $\g=(\gamma,\gamma-\mu,\Theta),$ is defined as the set of all
\begin{equation}\label{newA45}
A=\Op_y(m+g)+C
\end{equation}
for a symbol $(m+g)(y,\eta)\in\RR_{M+G}^\mu(\Omega\times\R^q,\g;j_-,j_+)$ and a smoothing operator $C$ in the sense of \textup{Definition} $\ref{glatt}.$ For $\Theta=(-\infty,0]$ we define the space of smoothing Mellin plus Green operators as the intersection of the respective spaces over $\N\ni k$ for finite $\Theta=(-(k+1),0].$
\end{defn}

\nt For an operator $A\in{L}^\mu_{M+G}(\Xw\times\Omega,\g)$ written as in Definition \ref{1.4.16} we set
\begin{equation}\label{new41a}
\sigma_\wedge(A)(y,\eta):=\sigma_\wedge(m+g)(y,\eta).
\end{equation}

\begin{rem}
Similarly as in \textup{Theorem \ref{4.13}} we can represent every $A\in{L}^\mu_{M+G}(\Xw\times\Omega,\g)$ in the form $A=A_0+C$ for a properly supported operator $A_0\in{L}^\mu_{M+G}(\Xw\times\Omega,\g)$ and some $C\in{L}^{-\infty}(\Xw\times\Omega,\g).$\\

\nt In fact$,$ we may replace the symbols $(m+g)(y,\eta)$ in \textup{\eqref{newA45}} by $\psi(y,y')(m+g)(y,\eta)$ for a $\psi$ similarly as in connection with \textup{Theorem \ref{4.13}}.
\end{rem}

\begin{thm}
$A\in{L}^\mu_{M+G}(\Xw\times\Omega,\g;j_0,j_+)$ for $\g=(\gamma-\nu,\gamma-(\mu+\nu),\Theta),$ and $B\in{L}^\nu_{M+G}(\Xw\times\Omega,\mathbf h;j_-,j_0)$ for $\mathbf h=(\gamma,\gamma-\nu,\Theta),$ $A$ or $B$ properly supported$,$ imply $AB\in{L}^{\mu+\nu}_{M+G}(\Xw\times\Omega,\g\circ\mathbf h;j_-,j_+)$ for $\g\circ\mathbf h=(\gamma,\gamma-(\mu+\nu),\Theta),$ and we have
$$\sigma_\wedge(AB)(y,\eta)=\sigma_\wedge(A)(y,\eta)\sigma_\wedge(B)(y,\eta).$$
\end{thm}

\begin{thm}
An operator $A\in{L}^\mu_{M+G}(\Xw\times\Omega,\g)$ induces continuous operators
$$A:\WW^s_{\comp}(\Omega,\KsgX)\rightarrow\WW^{s-\mu}_{\loc}(\Omega,\K^{\infty,\gamma-\mu}(\X))$$
and
$$A:\WW^s_{\comp}(\Omega,\K^{\infty,\gamma-\mu}_\P(\X))\rightarrow\WW^{s-\mu}_{\loc}(\Omega,\K^{\infty,\gamma-\mu}_\Q(\X))$$
for every $s\in\R$ and every asymptotic type $\P$ with some resulting asymptotic type $\Q.$ If $A$ is properly supported we may write $\comp$ or $\loc$ on both sides.
\end{thm}

\begin{defn}\label{D1.4.23}
\begin{itemize}
\item[\textup{(i)}] An operator $1+A$ for $A\in{L}^0_{M+G}(\Xw\times\Omega,\g;j_-,j_+)$ for $\g=(\gamma,\gamma,(-(k+1),0]),$ is called elliptic if
\begin{equation}\label{new46_B}
\sigma_\wedge(1+A)(y,\eta)=1+\sigma_\wedge(A)(y,\eta):\KsgX\oplus\C^{j_-}\rightarrow\KsgX\oplus\C^{j_+}
\end{equation}
is a family of bijective operators for all $(y,\eta)\in\Omega\times(\R^q\setminus\{0\})$ and some $s\in\R.$
\item[\textup{(ii)}] An operator $1+B$ for $B\in{L}^0_{M+G}(\Xw\times\Omega,\g;j_+,j_-)$ properly supported$,$ is called a parametrix of $1+A$ if
$$(1+A)(1+B)=1\;\textup{ and }\;(1+B)(1+A)=1$$
modulo ${L}^{-\infty}.$
\end{itemize}
\end{defn}

\nt Note that Definition \ref{D1.4.23} contains as a special case the ellipticity in the space identity plus operators in ${L}^0_{M+G}.$ This may illustrate some phenomena, and we only asked a bijectivity condition between edge symbols, without taking into account additional edge conditions of trace and potential type.\\

\nt Note that the conormal symbol
\begin{equation}\label{iso131}
\sigma_\c\sigma_\wedge(1+A)(y,\eta):=1+f_{00}(y,\eta):H^s(X)\rightarrow H^s(X)
\end{equation}
for $z\in\Gamma_{\frac{n+1}2-\gamma}$ and fixed $y\in\Omega$ is a family of isomorphisms if and only if \eqref{new46_B} is a family of Fredholm operators. In particular, when we require the bijectivity of \eqref{new46_B} we implicitly assume the isomorphisms \eqref{iso131}. In general, when we assume for the moment that $A\in{L}^0_{M+G}(\Xw\times\Omega,\g)$ and that the operators \eqref{iso131} form a family of isomorphisms, then we can always find $j_-,j_+\in\N$ and a family of $2\times2$ block matrix isomorphisms
\begin{equation}\label{bi47}
\left(\begin{array}{cc}
\sigma_\wedge(1+A)&\sigma_\wedge(K)\\[4mm]
\sigma_\wedge(T)&\sigma_\wedge(Q)
\end{array}\right)(y,\eta):
\begin{array}{ccc}
\KsgX& &\K^{s,\gamma}(\X)\\
\oplus&\rightarrow&\oplus\\
\C^{j_-}& &\C^{j_+}
\end{array}
\end{equation}
where the matrix $\left(\begin{array}{cc}
0&\sigma_\wedge(K)\\[4mm]
\sigma_\wedge(T)&\sigma_\wedge(Q)
\end{array}\right)(y,\eta)$
has the form $\sigma_\wedge(g)(y,\eta)$ for some $g(y,\eta)\in\RR^\mu_G(\Omega\times\R^q,\g;j_-,j_+),$ cf. Definition \ref{Green2}. Note that the special case $j_-=j_+=0$ contains all typical features of the general operators for arbirtary $j_\pm;$ so we often content ourselves with operators of the form of an upper left corner. At present we impose Mellin amplitude functions $f_{j\alpha}(y,z)\in C^\infty(\Omega,M^{-\infty}_{\RR_{j\alpha}}(X))$ with constant discrete asymptotic types $\RR_{j\alpha},$ cf. Definition \ref{1.4.14}. However, if we intend to find a parametrix of an elliptic operator $1+A,A\in{L}^0_{M+G}(\Xw\times\Omega,\g)$ we also need a condition on a constant discrete behaviour of asymptotic types under inversion of the principal conormal symbol, namely,
\begin{equation}\label{Sig47}
\sigma_\c\sigma_\wedge(1+A)^{-1}(y,z)\in C^\infty(\Omega,M^0_\S(X))
\end{equation}
for some constant discrete Mellin asymptotic type $\S.$ (To finally dismiss such a restriction is just one of the essential results of the present calculus of operators with variable discrete asympotics.) Under the condition \eqref{Sig47} we have the following theorem.

\begin{thm}\label{inv}
An elliptic operator $1+A,$ $A\in{L}^0_{M+G}(\Xw\times\Omega,\g;j_-,j_+)$ has a properly supported parametrix $1+B,$ with $B\in{L}^0_{M+G}(\Xw\times\Omega,\g;j_+,j_-).$
\end{thm}

%1.4.4
%
%
%
%
%
\subsection{Edge operators and their symbolic structure}\label{S4.4}

Our next objective is to construct operators that furnish our edge algebra. 
The main ingredients are operator-valued symbols of the following kind.

\begin{defn}\label{1.4.17}
Let $\RR^\mu(\Omega\times\R^q,\g),$ $\Omega\subseteq\R^q$ open$,$ defined to be the space of operator functions $\a(y,\eta):=a(y,\eta)+(m+g)(y,\eta)$ for $a(y,\eta)$ as in $(\ref{new38})$ and $(m+g)(y,\eta)\in\RR^\mu_{M+G}(\Omega\times\R^q,\g).$ More generally$,$ let $\RR^\mu(\Omega\times\R^q,\g;j_-,j_+)$ denote the space of $2\times2$ block matrix functions of the form $\diag(\a(y,\eta),0)+g(y,\eta)$ for $\a(y,\eta)\in\RR^\mu(\Omega\times\R^q,\g),$ and $g(y,\eta)\in\RR_G^\mu(\Omega\times\R^q,\g;j_-,j_+)$.
\end{defn}

\nt By virtue of Theorem \ref{symb26}, Definition \ref{Green} and Proposition \ref{1.4.15} we have
$$\RR^\mu(\Omega\times\R^q,\g)\subset S^\mu\Set{\Omega\times\R^q;\K^{s,\gamma}(\X),
\K^{s-\mu,\gamma-\beta}(\X)}.$$

\nt On the space $\RR^\mu(\Omega\times\R^q,\g)$ we have a two-component principal symbolic structure, namely,
$$\sigma(\a)=(\sigma_\psi(\a),\sigma_\wedge(\a)).$$
Here, $\sigma_\psi(\a)$ is defined by the expression (\ref{sym38}), while
\begin{equation}\label{sig40}
\sigma_\wedge(\a)(y,\eta):=\sigma_\wedge(a)(y,\eta)+\sigma_\wedge(m+g)(y,\eta),
\end{equation}
cf. the formulas (\ref{sig1}) and (\ref{new41a}).

\begin{defn}\label{1.4.18}
The space ${{L}}^\mu(\Xw\times\Omega,\g),$ $\mu\in\R,$ $\g=(\gamma,\gamma-\mu,\Theta),$ of edge operators with $($constant discrete$)$ asymptotics is defined as the set of all
\begin{equation}\label{edg40}
A=\Op_y(\a)+\Phi A_{\mathrm{int}}\Phi'+C
\end{equation}
for a symbol $\a(y,\eta)\in\RR^\mu(\Omega\times\R^q,\g),$ $A_{\mathrm{int}}\in L_{\clas}^\mu(\X\times\Omega),$ functions $\Phi,\Phi'\in C^\infty(\R_+)$ vanishing near $r=0,$ and an operator $C\in{{L}}^{-\infty}(\Xw\times\Omega,\g),$ cf. \textup{Definition \ref{glatt}}. More generally$,$ let ${{L}}^\mu(\Xw\times\Omega,\g;j_-,j_+)$ denote the space of $2\times2$ block matrix operators of the form $\A=\diag(A,0)+G+C$ for $A\in{{L}}^\mu(\Xw\times\Omega,\g),$ $G=\Op_y(g)$ for some $g(y,\eta)\in\RR_G^\mu(\Omega\times\R^q,\g;j_-,j_+),$ and $C\in{{L}}^{-\infty}(\Xw\times\Omega,\g;j_-,j_+),$ cf. \textup{Definition \ref{glatt}}.
\end{defn}

\nt The dimensions $j_-$ and $j_+$ have the interpretation of fibre dimensions of smooth complex vector bundles $J_-$ and $J_+,$ respectively, over the edge. Here in the local theory the open set plays the role of a chart on an edge in general; $\Omega$ may assumed to be a contractible open set (e.g, a ball), and then $J_\pm$ are of course trivial. Nevertheless, in order to emphasise invariance properties we also write $J_\pm$ rather than $\Omega\times\C^{j_\pm},$ and the operators $A\in{{L}}^\mu(\Xw\times\Omega,\g;j_-,j_+)$ act as continuous operators
\begin{equation*}
\A=\left(\begin{array}{cc}
A&K\\[4mm]
T&Q
\end{array}\right):
\begin{array}{ccc}
C_0^\infty(\X\times\Omega)& &C^\infty(\X\times\Omega)\\
\oplus&\rightarrow&\oplus\\
C_0^\infty(\Omega,J_-)& &C^\infty(\Omega,J_+)
\end{array}
\end{equation*}
where $C_{(0)}^\infty(\Omega,J_\pm)$ means the spaces of smooth sections in the respective bundles. More generally, also in the upper left corners we may admit operators between sections of (smooth complex) vector bundles over $\Xw\times\Omega,$ or, in the simplest case, block matrices. However, here we focus on asymptotic effects; the case with bundles is a straightforward generalisation, and we ignore this case.

\begin{rem}\label{Re1.4.27}
Similarly as in the standard pseudo-differential calculus on an open manifold it can be proved that for every $A\in{L}^\mu(\Xw\times\Omega,\g;j_-,j_+)$ there is a properly supported $($with respect to the variables $r$ and $y)$ operator $A_0$ such that $A=A_0$ modulo ${L}^{-\infty}(\Xw\times\Omega,\g;j_-,j_+).$
\end{rem}

\nt As noted before, for simplicity, we often consider operators of the form of upper left corners rather than $2\times2$ block matrices in general. The results for block matrices are completely analogous; more details may be found in the text books on edge pseudo-differential operators, e.g., \cite{Schu20}.\\

\nt For an edge operator $A$ given in the form (\ref{edg40}) we form the principal symbolic hierarchy
$$\sigma(A):=(\sigma_\psi(A),\sigma_\wedge(A))$$
by $\sigma_\psi(A):=\sigma_\psi(\a)+\sigma_\psi(\Phi A_{\textup{int}}\Phi')$ where the first summand is defined in \eqref{sym38} while the second one is the homogenous principal symbol of the operator $\Phi A_{\textup{int}}\Phi'$ in the standard sense, and $\sigma_\wedge(A):=\sigma_\wedge(\a).$ Observe that
\begin{equation}\label{ssig}
\sigma_\psi(A)(r,x,y,\rho,\xi,\eta)=r^{-\mu}\til\sigma_\psi(A)(r,x,y,r\rho,\xi,r\eta)
\end{equation}
with $x$ referring to a chart on $X,$ say, $x$ varying in an open set $\Sigma\subset\R^q,$ where 
\begin{equation}\label{ssig2}
\til\sigma_\psi(A)(r,x,y,r\rho,\xi,r\eta)\in C^\infty(\Rr\times\Sigma\times\Omega\times(\R^{1+n+q}_{\til\rho,\xi,\til\eta}\setminus\{0\}))
\end{equation}
is homogeneous in $(\til\rho,\xi,\til\eta)\ne0$ of order $\mu,$ and smooth in $r$ up to $r=0.$

\begin{thm}
Let $A\in{L}^\mu(\Xw\times\Omega,\g;j_0,j_+),$ $B\in{L}^\nu(\Xw\times\Omega,{\mathbf h};j_-,j_0)$ for $\g=(\gamma-\nu,\gamma-(\mu+\nu),\Theta),$ ${\mathbf h}=(\gamma,\gamma-\nu,\Theta),$ and let $A$ or $B$ be properly supported$,$ cf. \textup{Remark \ref{Re1.4.27}}. Then $AB\in{L}^{\mu+\nu}(\Xw\times\Omega,\g\circ{\mathbf h};j_-,j_+),$ and we have
$$\sigma_\psi(AB)=\sigma_\psi(A)\sigma_\psi(B),\quad\sigma_\wedge(AB)=\sigma_\wedge(A)\sigma_\wedge(B).$$
If $A$ or $B$ belongs to the subspace with subscript $M+G$ $(G)$ then also the composition $AB.$
\end{thm}

\begin{thm}\label{cont42}
An $A\in{L}^\mu(\Xw\times\Omega,\g),$ $\g=(\gamma,\gamma-\mu,\Theta),$ induces continuous operators
\begin{equation}\label{42a}
A:\WW^s_{[\comp)}(\Omega,\KsgX)\rightarrow\WW^{s-\mu}_{[\loc)}(\Omega,\K^{s-\mu,\gamma-\mu}(\X)),
\end{equation}
\begin{equation}\label{42b}
A:\WW^s_{[\comp)}(\Omega,\K^{s,\gamma}_\P(\X))\rightarrow\WW^{s-\mu}_{[\loc)}(\Omega,\K^{s-\mu,\gamma-\mu}_\Q(\X))
\end{equation}
for every $s\in\R$ and every $($constant discrete$)$ asymptotic type $\P$ with some resulting $($constant discrete$)$ asymptotic type $\Q.$ If $A$ is properly supported$,$ then in $(\ref{42a})$ and $(\ref{42b})$ we may write $[\comp)$ or $[\loc)$ on both sides.
\end{thm}

\nt For purposes below we define ${L}^{\mu-1}(\Xw\times\Omega,\g)$ for $\g=(\gamma,\gamma-\mu,\Theta)$ as the space of all $A\in{L}^\mu(\Xw\times\Omega,\g)$ such that $\sigma(A)=0$ (i.e., both components vanish). Every such $A$ has again a pair $\sigma(A)=(\sigma_\psi(A),\sigma_\wedge(A))$ of principal symbols, this time of homogeneity $\mu-1.$ Then, inductively, for every $j\in\N,$ $j\ge1,$ we set
$${L}^{\mu-j}(\Xw\times\Omega,\g):=\set{A\in{L}^{\mu-(j-1)}(\Xw\times\Omega,\g):\sigma(A)=0}.$$

\begin{rem}
Let $A\in{L}^{\mu-1}(\Xw\times\Omega,\g)$ and $\psi,\psi'\in C_0^\infty(\Rr\times\Omega).$ Then
$$\psi A\psi':\WW^s(\R^q,\KsgX)\rightarrow\WW^{s-\mu}(\R^q,\K^{s-\mu,\gamma-\mu}(\X))$$
is a compact operator for every $s\in\R.$
\end{rem}

%1.4.5
%
%
%
%
%
\subsection{Ellipticity in the edge calculus}\label{S4.5}

\begin{defn}\label{E1}
An operator $A\in{L}^\mu(\Xw\times\Omega,\g),$ $\g=(\gamma,\gamma-\mu,\Theta),$ is called $\sigma_\psi$-elliptic if $\sigma_\psi(A)(r,x,y,\rho,\xi,\eta)\ne0$ for all $r,x,y,$ $r>0,$ and $(\rho,\xi,\eta)\ne0,$ and if $\til\sigma_\psi(A)(r,x,y,\til\rho,\xi,\til\eta)\ne0$ for all $r,x,y,$ up to $r=0,$ and $(\til\rho,\xi,\til\eta)\ne0.$
\end{defn}

\nt We call an operator $A\in{L}^\mu(\Xw\times\Omega,\g)$ $\sigma_\psi$-elliptic close to the edge if the conditions $\sigma_\psi(A)\ne0$ and $\til\sigma_\psi(A)\ne0$ hold for all $0<r<R$ and $0\le r<R,$ respectively, for some $R>0.$

\begin{thm}\label{T41}
Let $A\in{L}^\mu(\Xw\times\Omega,\g)$ be $\sigma_\psi$-elliptic close to the edge. Then for every $y\in\Omega$ there exists a discrete set $D(y)\subset\C,$ $D(y)\cap\set{c\le\Re z\le c'}$ finite for every $c\le c',$ such that
\begin{equation}\label{Sig49}
\sigma_\wedge(A)(y,\eta):\KsgX\rightarrow\K^{s-\mu,\gamma-\mu}(\X),
\end{equation}
$\eta\ne0,$ is a family of Fredholm operators for every $\gamma\in\R,$ such that $\Gamma_{\frac{n+1}2-\gamma}\cap D(y)=\emptyset.$
\end{thm}

\nt In the elliptic theory of edge operators we usually assume that for some $\gamma\in\R$ the criterion of Theorem \ref{T41} is satisfied for all $y\in\Omega.$ From Theorem \ref{T1.2.47} (ii) we know that \eqref{Sig49} is Fredholm if and only if $\sigma_\wedge(A)(y,\eta)$ is elliptic in the cone algebra over $\X.$ This refers to the principal symbols from the cone theory, namely, the components of \eqref{new22}. In the present case those concern $\sigma_\wedge(A),$ i.e., we have to consider
\begin{equation}\label{cone49}
\sigma(A)=(\sigma_\psi\sigma_\wedge(A),\sigma_{\mathrm c}\sigma_\wedge(A),\sigma_{\mathrm E}\sigma_\wedge(A))
\end{equation}

\nt It turns out that the $\sigma_\psi$-ellipticity of $A$ in the sense of Definition \ref{E1} entails the ellipticity condition of the cone calculus for the first and the third component of \eqref{cone49}. What concerns $\sigma_{\mathrm c}\sigma_\wedge(A)$ we have in the present case
$$\sigma_{\mathrm c}\sigma_\wedge(A)(y,z)=\til h(0,y,z,0)+f_{00}(y,z)$$
where $f_{00}(y,z)$ comes from the $(M+G)$-summand in $A,$ cf.~the formula \eqref{newA45}, while the conormal symbol $\til h(0,y,z,0)$ of the cone operator $\sigma_\wedge(A)(y,\eta)$ with $a(y,\eta)$ being of the form \eqref{new38}, is coming from \eqref{new38a}. All this refers to the representation of $A$ as \eqref{edg40}, and $\sigma_{\mathrm c}\sigma_\wedge(A)(y,\eta)=\sigma_{\mathrm c}\sigma_\wedge(\a)(y,\eta)$ for $\sigma_\wedge(\a)(y,\eta)$ from \eqref{sig40}.\\

\nt Then, another well-known aspect is to require the existence of elliptic edge conditions of trace and potential type, i.e., on the edge-symbolic level a smooth family of block matrix isomorphisms
\begin{equation}\label{matr}
\left(\begin{array}{cc}
\sigma_\wedge(A)&\sigma_\wedge(K)\\[4mm]
\sigma_\wedge(T)&\sigma_\wedge(Q)
\end{array}\right)(y,\eta):
\begin{array}{ccc}
\KsgX& &\K^{s-\mu,\gamma-\mu}(\X)\\%[1mm]
\oplus&\rightarrow&\oplus\\%[1mm]
J_{-,y}& &J_{+,y}
\end{array}
\end{equation}
for suitable (smooth complex) vector bundles $J_\pm$ over $\Omega$ (which are, of course, trivial when, for instance, $\Omega$ is contractible). The relation (\ref{matr}) is required for an $s=s_0,$ and then, as soon as $\sigma_\wedge(K),\sigma_\wedge(T),\sigma_\wedge(Q)$ are homogeneous principal components of corresponding entries as in Definition \ref{Green2} (with $j_\pm$ being the fibre dimension of $J_\pm)$ the operators (\ref{matr}) are isomorphisms for all $s\in\R.$

\begin{rem}\label{co58}
It turns out$,$ that when $\sigma_\wedge(A)(y,\eta)$ can be completed by extra finite rank entries to a family of isomorphisms $(\ref{matr}),$ it is possible to find a Green symbol $g_1(y,\eta)$ of the kind of an upper left corner in \textup{Definition \ref{Green2}} and a smoothing Mellin symbol $m_1(y,\eta)$ of the form $(\ref{new40b}),$ such that
$$\sigma_\wedge(A)(y,\eta)+\sigma_\wedge(m_1+g_1)(y,\eta):\KsgX\rightarrow\K^{s-\mu,\gamma-\mu}(\X)$$
is a family of isomorphisms. Since in principle the associated operator $A+M_1+G_1$ is of the same nature as $A$ itself and because asymptotic properties of solutions are the main issue here$,$ the essential aspects concern elliptic operators $A\in{L}^\mu(\Xw\times\Omega,\g)$ without extra data of trace and potential type.
\end{rem}

\begin{rem}
Let $A_1,A_2\in{{L}}^\mu(\Xw\times\Omega,\g),$ and suppose that $A_1=A_2$ modulo $L^{-\infty}(\X\times\Omega).$ Then we have
$$A_1=A_2\qquad\mod\,{{L}}^\mu_{M+G}(\Xw\times\Omega,\g).$$
\end{rem}

\begin{defn}\label{1.4.22}
An operator $A\in{L}^\mu(\Xw\times\Omega,\g)$ for $\g=(\gamma,\gamma-\mu,\Theta)$ is called elliptic close to the edge $\Omega$ if $A$ is $\sigma_\psi$-elliptic close to the edge$,$ cf. \textup{Definition \ref{E1},} and if
\begin{equation}\label{new43a}
\sigma_\wedge(A)(y,\eta):\KsgX\rightarrow\K^{s-\mu,\gamma-\mu}(\X)
\end{equation}
is a family of isomorphisms for all $(y,\eta)\in\Omega\times(\R^q\setminus\{0\}),$ for some $s=s_0\in\R.$
\end{defn}

\nt It follows, in particular, that if (\ref{new43a}) is an isomorphism for $s=s_0$ then so is for all $s\in\R,$ cf. Theorem \ref{1.2.41} (ii). The operators (\ref{new43a}) are elliptic elements of the cone algebra over the (open stretched) infinite cone $\X$ for every fixed $y\in\Omega,$ $\eta\in\R^q\setminus\{0\},$ cf. Definition \ref{Ell_def} (ii).\\

\nt The ellipticity in the cone algebra means the bijectivity of the respective principal symbols from the subordinate cone calculus, especially, of the conormal symbol
\begin{equation}\label{con43}
\sigma_{\mathrm c}\sigma_\wedge(A)(y,z)=h_0(0,y,z,0)+f_{00}(y,z):H^s(X)\rightarrow H^{s-\mu}(X),
\end{equation}
cf.~relations (\ref{new38a}) and (\ref{new41}), which is a family of isomorphisms for every $s\in\R.$ Here $\sigma_{\mathrm c}\sigma_\wedge(A)(y,z)\in C^\infty(\Omega,M_\RR^\mu(X))$ for some (constant discrete) asymptotic type $\RR$ where $\pi_\C\RR\cap\Gamma_{\frac{n+1}2-\gamma}=\emptyset$ (concerning notation, cf. Definition \ref{1.2.24}).\\

\nt In the construction of the local parametrices of operators that are elliptic in the sense of Definition \ref{1.4.22} we first invert the components of $\sigma(A)=(\sigma_\psi(A),\sigma_\wedge(A))$ and then, via an operator convention we pass to an operator
\begin{equation}\label{P43}
P\in{L}^{-\mu}(\Xw\times\Omega,\g^{-1}),\qquad\g^{-1}=(\gamma-\mu,\gamma,\Theta)
\end{equation}
such that
$$\sigma(P)=(\sigma_\psi(A)^{-1},\sigma_\wedge(A)^{-1}).$$

\nt The inversion process for $\sigma_\wedge(A)$ contains the inversion of (\ref{con43}). In the present special subclass which admits only constant (in $y)$ discrete asymptotics the relation (\ref{P43}) can be true only when
\begin{equation}\label{inv43}
\sigma_{\mathrm c}\sigma_\wedge(A)^{-1}(y,z)\in C^\infty(\Omega,M_\S^{-\mu}(X))
\end{equation}
for some (constant discrete) asymptotic type $\S.$ This is, of course, an additional condition, now imposed to illustrate how asymptotics of solutions are generated. Then we have the following result.

\begin{thm}\label{ell43}
If $A\in{L}^\mu(\Xw\times\Omega,\g)$ is elliptic in the sense of \textup{Definition \ref{1.4.22}} and if the condition \textup{(\ref{inv43})} is satisfied$,$ then there exists a $($properly supported with respect to the $(y,y')$ variables$)$ $P\in{L}^{-\mu}(\Xw\times\Omega,\g^{-1})$ such that
$$\varphi(1-PA)\psi\in{L}^{-\infty}(\Xw\times\Omega,(\gamma,\gamma,\Theta)),$$
$$\psi(1-AP)\varphi\in{L}^{-\infty}(\Xw\times\Omega,(\gamma-\mu,\gamma-\mu,\Theta)),$$
for every $\varphi,\psi\in C_0^\infty(\Rr\times X\times\Omega),$ with $\supp\varphi,$ $\supp\psi\subseteq[0,R]\times X$ for $R>0$ as in \textup{Definition \ref{E1}}.
\end{thm}

\nt In general, (\ref{inv43}) will be violated, i.e., the points of $\pi_\C\S$ may depend on $y,$ and the associated multiplicities may have `jumps' with varying $y,$ which is just the main issue in our program of variable discrete asymptotics in \cite{Schu63}. Note that the localising factors $\varphi,\psi$ come from the fact that we do not impose $\sigma_\psi$-ellipticity of $A$ for $r>R+\varepsilon$ for some $\varepsilon>0$ and hence we cannot expect any smoothing remainders for $r$ too large.

\begin{thm}
Let $A\in{L}^\mu(\Xw\times\Omega,\g)$ satisfy the conditions of \textup{Theorem \ref{ell43}}. Then
$$Au=f$$
for $f\in\WW^{s-\mu}_{[\loc)}(\Omega,\K^{s-\mu,\gamma-\mu}_\Q(\X)),$ $s\in\R$ fixed$,$ and $u\in\WW^{-\infty}(\R^q,\K^{-\infty,\gamma}(\X)),$ with $\supp u\subset[0,R]\times X\times K$ for some $K\Subset\Omega$ implies
$$u\in\WW^s(\R^q,\K^{s,\gamma}_\P(\X))$$
for every constant asymptotic type $\Q$ with some resulting constant asymptotic type $\P.$
\end{thm}

\begin{proof}
We choose any $\psi\in C_0^\infty(\Rr\times X\times\Omega)$ such that $\psi\equiv1$ on $\supp u.$ Thus $u=\psi u,$ and by virtue of Theorem \ref{cont42} we have
$$A\psi u=f\quad\Rightarrow\quad PA\psi u=Pf\in\WW^s_{[\loc)}(\Omega,\K^{s,\gamma}_{\til\Q}(\X))$$
for some resulting asymptotic type $\til\Q,$ i.e.,
$$\varphi PA\psi u=\varphi Pf\in\WW^s_{[\loc)}(\Omega,\K^{s,\gamma}_{\til\Q}(\X)).$$
From $\varphi PA\psi u=\varphi\psi u-\varphi(1-PA)\psi u$ and because the smoothing part of the remainder, $\varphi(1-PA)\psi u,$ belongs to the space $\WW^\infty_{[\loc)}(\Omega,\K^{\infty,\gamma}_{\til{\til\Q}}(\X))$ for some asymptotic type $\til{\til\Q},$ it follows that 
$$\varphi\psi u\in\WW^s_{[\loc)}(\Omega,\K^{s,\gamma}_{\til\Q}(\X))+\WW^\infty_{[\loc)}(\Omega,\K^{\infty,\gamma}_{\til{\til\Q}}(\X)).$$

\nt Choosing $\varphi$ in such a way that $\varphi|_{\supp\psi}\equiv1$ we obtain $u\in\WW^s_{[\loc)}(\Omega,\K^{s,\gamma}_\Q(\X)).$
\end{proof}

\nt For future references we sketch more details on the way to obtain a parametrix and to derive asymptotics of solutions. From Definition \ref{1.4.18} the operator $A$ is of the form (\ref{edg40}) where by Definition \ref{1.4.17} the amplitude function $\a(y,\eta)$ contains an $a(y,\eta)$ as in (\ref{new38}) and an $(m+g)(y,\eta)\in\RR^\mu_{M+G}(\Omega\times\R^q;\g).$ Recall that (\ref{new38}) is already the result of some quantisation (containing what we did through Theorem \ref{symb26}), starting from local operator-valued symbols (\ref{new37a}). Under the condition of $\sigma_\psi$-ellipticity close to $r=0$ (which only depends on $p(r,y,\rho,\eta))$ by ``elementary means'' of the general parameter-dependent pseudo-differential calculus we find a Leibniz inverse $\sharp^{-1}$ of $r^{-\mu}p(r,y,\rho,\eta)$ of the form $r^\mu p^{(-1)}(r,y,\rho,\eta)$ for a
$$p^{(-1)}(r,y,\rho,\eta)=\til p^{(-1)}(r,y,r\rho,r\eta),$$
$\til p^{(-1)}(r,y,\til\rho,\til\eta)\in C^\infty(\Rr\times\Omega,L_\clas^{-\mu}(X;\R_{\til\rho,\til\eta}^{1+q})),$ such that
$$r^\mu p^{(-1)}(r,y,\rho,\eta)\sharp_{r,y}r^{-\mu}p(r,y,r\rho,r\eta)=1,$$
modulo $C^\infty(\R_+\times\Omega,L^{-\infty}(X;\R_{\rho,\eta}^{1+q})).$ Then, applying Theorems \ref{qua} and \ref{sm29}, we find an $a^{(-1)}(y,\eta)$ belonging to $r^\mu p^{(-1)}$ near $r=0$ of analogous structure as $a(y,\eta)$ (that belongs to $r^{-\mu}p$ near $r=0)$ such that  
$$a^{(-1)}(y,\eta)\sharp_y a(y,\eta)=1\qquad\mod \RR_{M+G}^0(\Omega\times\R^q,\g_0)$$
for $\g_0=(\gamma,\gamma,(-\infty,0]).$\\

\nt In other words, setting $A^{(-1)}:=\Op_y(a^{(-1)})$ it follows that
$$A^{(-1)}A=1+M+G$$
for some $M+G\in{L}^0_{M+G}(\Xw\times\Omega).$ At this moment we do not claim that $A^{(-1)}$ is also elliptic with respect to $\sigma_\wedge;$ however, similarly as in the cone calculus we know that
$$\sigma_\mathrm c\sigma_\wedge(A^{(-1)})(y,z+\mu)\sigma_\mathrm c\sigma_\wedge(A)(y,z)=1+f(y,z)$$
for an $f(y,z)\in C^\infty(\Omega,M_\RR^{-\infty}(X))$ and some constant (in $y)$ discrete Mellin asymptotic type $\RR.$ Since for every fixed $y_0\in\Omega$ we have an inverse $(1+f(y_0,z))^{-1}=1+l(y_0,z)$ for some $l\in M_\S^{-\infty}(X)$ we obtain for $z\in\Gamma_{\frac{n+1}2-(\gamma-\mu)}$
$$\sigma_\mathrm c\sigma_\wedge(A)(y_0,z)^{-1}=(1+l(y_0,z-\mu))\sigma_\mathrm c\sigma_\wedge(A^{(-1)})(y_0,z),$$
i.e., we completed the inverse of the subordinate principal conormal symbol of the given operator $A$ at a point $y_0\in\Omega.$ Then, if we associate with the conormal symbol $1+l(y_0,z)$ an operator-valued symbol $1+\omega(r[\eta])\op_M^{\gamma-\frac n2}l)(y_0)\omega'(r[\eta])\in\RR_{M+G}^0(\R^q,\g),$ the operator
$$P_0:=\Op_y(1+\omega(r[\eta])\op_M^{\gamma-\frac n2}(l)(y_0)\omega'(r[\eta]))A^{(-1)}\in{L}^{-\mu}(\Xw\times\Omega,\g^{-1})$$
is not only $\sigma_\psi$-elliptic but also $\sigma_\wedge$-elliptic at $y_0.$ Thus there is a neighbourhood $U\subset\Omega$ of $y_0$ such that $P_0$ is $\sigma_\wedge$-elliptic for all $y\in U.$ If we localise the task to assign the asymptotics of solutions to $U$ (which is reasonable anyway) without loss of generality we may assume that (after changing notation) our $P_0$ is $(\sigma_\psi,\sigma_\wedge)$-elliptic (with respect to $\sigma_\wedge$ everywhere in $\Omega).$ In other words, $P_0$ is an elliptic `crude' parametrix of $A$ in the sense
$$P_0A=1+M_{-1}+G_0$$
for $M_{-1}+G_0\in\RR_{M+G}^0(\Omega\times\R^q,(\gamma,\gamma,\Theta))$ but of vanishing principal conormal symbol. By Proposition \ref{inv} we find an operator $1+N$ for $N\in{L}^{-\mu}_{M+G}(\Xw\times\Omega,(\gamma-\mu,\gamma-\mu,\Theta))$ such that
$$(1+N)P_0A\in{L}^{-\infty}(\Xw\times\Omega,(\gamma,\gamma,\Theta)).$$

\cleardoublepage

\addcontentsline{toc}{section}{References}
%\bibliography{master}

\cleardoublepage

%\newpage

\twocolumn

\addtolength{\hoffset}{-10mm}

\addtolength{\textwidth}{20mm}

\addcontentsline{toc}{section}{Index}
\printindex

\section*{List of Symbols\hfill}
\addcontentsline{toc}{section}{List of Symbols}
%\listofsymbols
\nt\addsymbol \prec: {prec}
\addsymbol \ang\cdot: {angeta}
\addsymbol [\cdot]: {kleta}\\
\addsymbol a_{\mathrm L},a_{\mathrm R}: {left}\\
%\addsymbol \A'(K): {Astr}
%\addsymbol \A'(K,E): {AKE}\\
%\addsymbol C^\infty(\Omega,\A'(K))^\bullet: {Cpoint}
%\addsymbol C^\infty(\Omega,\A(\C\setminus K))^\bullet: {Cpoint2}
%\addsymbol C^\infty(\Omega,\A'(K,E))^\bullet: {Cinf}\\
\addsymbol \dslash\rho: {dsla}
\addsymbol \Diff^\nu(X): {Diffnu}
\addsymbol \Diff^\mu_\mathrm{deg}(M): {Diffdeg}\\
\addsymbol \E_\P(\X): {singdiscr}\\
%\addsymbol \E_K(\X): {EKX}\\
\addsymbol \Gamma_\beta: {Gambeta}\\
\addsymbol g_X: {RiemMetr}
$\g,$~~\pageref{boldg}, \pageref{g43}\\
\addsymbol \g\circ\til\g: {gcompg}
\addsymbol \g^{-1}: {g-1}\\
\addsymbol \H^{s,\gamma}(\X): {symbol:HsGa}
\addsymbol H^s(\R\times X): {HS13}
\addsymbol \hat H^s(\R_\tau\times X): {Fu13}
\addsymbol H^{s,\gamma}(M): {HsgM}
\addsymbol H^s_\cone(X^\asymp): {Hscone}
\addsymbol H^{s;g}_\cone(\X): {Hscong}
\addsymbol \hat H^s(\Gamma_\beta\times X): {HhatG}
\addsymbol H^{s,\gamma}_\P(M): {HsgP}
\addsymbol H^s(\R^q,H): {HsRH}\\
\addsymbol \kappa_\lambda^g: {Klhom}\\
\addsymbol \K^{s,\gamma;g}(\X): {symbol:Ksgag}
\addsymbol \K^{s,\gamma}(\X): {symbol:Ksga}
\addsymbol \K^{\infty,\gamma}(\X): {new10a}
\addsymbol \K^{s,\gamma;g}_\Theta(\X): {symbol:KTh}
\addsymbol \K^{s,\gamma;g}_\P(\X): {new14}
\addsymbol \K^{s,\gamma}_\Theta(\X): {Kcfl}
\addsymbol \K^{s,\gamma}_\P(\X): {Kcas}
\addsymbol K^{s,\gamma}(\X): {new111}\\
\addsymbol L^{-\infty}(X): {Lclass}
\addsymbol L^\mu_{(\clas)}(\Omega;\R^l): {Lmu}
\addsymbol L^{-\infty}(\Omega;\R^l): {Linf}
\addsymbol L^\mu_{(\clas)}(X;\R^l): {new10}
\addsymbol L^{-\infty}(X;\R^l): {Linf2}
\addsymbol L^\mu_{(\clas)}(X;\Gamma_{\frac12-\gamma}): {opvalsym}
\addsymbol L^{\mu;\nu}_\clas(X^\asymp): {Lclmn}
\addsymbol L^\mu_{(\clas)}(\Omega;H,\til H): {Lmu2}
\addsymbol L^\mu_{(\clas)}(\R^q;H,\til H)_{\mathrm{b}}: {Lmub}
\addsymbol L_G(\X,\g)_{\P,\Q}: {AAGPQ}
\addsymbol L_G(\X,\g): {AAG}
\addsymbol L_{M+G}^\mu(\X,\g): {AAMGm}
\addsymbol L^{\mu}(M,\g): {AAmMg}
\addsymbol L^{-\infty}(\Xw\times\Omega,\g): {glatt}
\addsymbol L^\mu_G(\Xw\times\Omega,\g;j_-,j_+): {1.4.12}
\addsymbol L^{\mu}_G(\Xw\times\Omega,\g): {1.4.12}
\addsymbol L^{\mu}_{M+G}(\Xw\times\Omega,\g): {1.4.16}
\addsymbol L^{\mu}(\Xw\times\Omega,\g): {1.4.18}\\
\addsymbol M_\gamma: {symbol:wM}
\addsymbol \M_\varphi: {Mphi}
\addsymbol M_\O^\mu(X;\R^q): {1.2.16}
\addsymbol M_\O^\mu(X): {MOmX}
\addsymbol M_\O^\mu(\R^q): {MOm}
\addsymbol M_\RR^{-\infty}(X): {MRinf}
\addsymbol M_{\RR}^{\mu}(X): {MRm}\\
\addsymbol \op_M^\gamma: {opMell}
\addsymbol \Op_{(y)}: {new39}
\addsymbol \textup{Os}[a]: {osc23}\\
\addsymbol \pi_\C{\P}: {Pip}
\addsymbol \pi_\C{\RR}: {Pir}\\
\addsymbol \overline\P: {Pcon}\\
%\addsymbol \P_U: {PrestrU}\\
\addsymbol \RR^\mu_G(\Omega\times\R^q,\g): {RG}
\addsymbol \RR^\mu_G(\Omega\times\R^q,\g)_{\P,\Q}: {RmuPQ}
\addsymbol \RR^\mu_G(\Omega\times\R^q,\g;j_-,j_+)_{\P,\Q}: {RGj}
\addsymbol \RR^\mu_G(\Omega\times\R^q,\g;j_-,j_+): {RGj}
\addsymbol \RR^\mu_{M+G}(\Omega\times\R^q,\g): {1.4.14}
\addsymbol \RR^\mu(\Omega\times\R^q,\g): {1.4.17}\\
$\sigma_{\mathrm c}^{\mu-j},$~~\pageref{consym}, \pageref{consym2}\\
\addsymbol \sigma_\psi: {symbs}
\addsymbol \sigma_\mathrm{c}: {symbs}
\addsymbol \sigma_\mathrm{E}: {new22}
\addsymbol \sigma_\mathrm{e}: {cs3}
\addsymbol \sigma_{\psi,\mathrm{e}}: {cs3}\\
\addsymbol S^\mu_{(\clas)}(U\times\R^n): {stsym}
\addsymbol \S^\gamma(\X): {symbol:Sga}
\addsymbol \S_\O(\X): {symbol:SgaO}
\addsymbol \S^\gamma_\Theta(\X): {symbol:SgaTh}
\addsymbol \S^\gamma_\P(\X): {SgP}
\addsymbol S^\mu_{(\clas)}(U\times\R^q;H,\til H): {symbsp}
\addsymbol S^\mu_{(\clas)}(\R^q;H,\til H): {symbsp}
\addsymbol S^\mu_{(\clas)}(U\times\R^q;H,\til H)_{\kappa,\til\kappa}: {new212}
\addsymbol S^{\bs{\mu};\bs{\nu}}(\R^{2q};V): {symb23}
\addsymbol S^{\bs{\infty};\bs{\infty}}(\R^{2q};V): {symb23}
\addsymbol S^\mu_{(\clas)}(\R^d\times\R^q;H,\til H)_{\mathrm{b}}: {symbb}
\addsymbol \S'(\R^q,H): {Sprime}\\
\addsymbol T^\delta\P: {TdP}
\addsymbol T^\beta f: {Tbeta}\\
\addsymbol u_\mathrm{sing},u_\mathrm{flat}: {decomp}\\
%\addsymbol \U(\Omega): {UOm}\\
%\addsymbol \VV: {VV}\\
\addsymbol \WW^s(\R^q,H): {symbol:Ws}
\addsymbol \WW^s(\R^q,H)_\kappa: {Wskappa}
\addsymbol \WW^s_\comp(\Omega,H): {comp}
\addsymbol \WW^s_\loc(\Omega,H): {loc}
\addsymbol \WW^s(\R^q,\KsgX): {WsKsg}
\addsymbol \WW^s(\R^q,\K^{s,\gamma;g}(\X))_{\kappa^g}: {new311}
\addsymbol \WW^{s,\gamma}_\P(\X\times\R^q): {WsP}\\
\addsymbol \Xw: {xw1}
\addsymbol X^\asymp: {xsymp}
\end{document}